\documentclass[12pt]{amsart}

\usepackage{framed}
\usepackage{braket}
\usepackage{amsmath,amssymb,amsthm}
\usepackage{mathtools}
\usepackage{color}
\usepackage{bm}
\usepackage[all]{xy}
\usepackage{amscd}
\usepackage{graphicx}
\usepackage{ascmac}
\usepackage{BOONDOX-frak, mathrsfs}
\usepackage[top=30truemm,bottom=30truemm,left=25truemm,right=25truemm]{geometry}
\usepackage{mathtools}
\mathtoolsset{showonlyrefs=true}
\usepackage[dvipsnames]{xcolor}

\makeatletter
\@addtoreset{equation}{section}

\makeatother

\newcommand{\Z}{\mathbb{Z}}
\newcommand{\Q}{\mathbb{Q}}
\newcommand{\C}{\mathbb{C}}

\newcommand{\F}{\mathbb{F}}
\newcommand{\R}{\mathbb{R}}
\newcommand{\trivial}{\bf{1}}

\newcommand{\OO}{\mathcal{O}}
\newcommand{\GG}{\mathcal{G}}

\newcommand{\PP}{\mathcal{P}}

\newcommand{\fq}{\mathfrak{q}}
\newcommand{\RG}{{\rm R}\Gamma}
\newcommand{\xx}{\boldsymbol{x}}

\newcommand{\ff}{\mathfrak{f}}

\newcommand{\cD}{\mathcal{D}}

\newcommand{\cE}{\mathcal{E}}
\newcommand{\cH}{\mathcal{H}}

\newcommand{\wtil}[1]{\widetilde{#1}}
\newcommand{\ol}[1]{\overline{#1}}
\newcommand{\parenth}[1]{\left( #1 \right)}

\newcommand{\ND}[1]{\textcolor{PineGreen}{#1}}

\newcommand{\otimesL}{\otimes^{\mathbb{L}}}

\DeclareMathOperator{\Gal}{Gal}
\DeclareMathOperator{\Coker}{Coker}
\DeclareMathOperator{\Ker}{Ker}
\DeclareMathOperator{\Imag}{Im}
\DeclareMathOperator{\Fitt}{Fitt}

\DeclareMathOperator{\ram}{ram}
\DeclareMathOperator{\fin}{fin}

\DeclareMathOperator{\ord}{ord}
\DeclareMathOperator{\Hom}{Hom}
\DeclareMathOperator{\RHom}{RHom}
\DeclareMathOperator{\rank}{rank}

\DeclareMathOperator{\ab}{ab}
\DeclareMathOperator{\Det}{Det}
\DeclareMathOperator{\Colman}{Col}
\DeclareMathOperator{\cyc}{cyc}
\DeclareMathOperator{\NN}{N}
\DeclareMathOperator{\Tr}{Tr}

\DeclareMathOperator{\perf}{perf}
\DeclareMathOperator{\loc}{loc}
\DeclareMathOperator{\Real}{Re}
\DeclareMathOperator{\Twist}{Tw}
\DeclareMathOperator{\twist}{tw}

\DeclareMathOperator{\rec}{rec}
\DeclareMathOperator{\Iw}{Iw}
\DeclareMathOperator{\id}{id}
\DeclareMathOperator{\res}{res}

\DeclareMathOperator{\ch}{ch}

\DeclareMathOperator{\cores}{cores}
\DeclareMathOperator{\LT}{LT}
\DeclareMathOperator{\inv}{inv}
\DeclareMathOperator{\Spec}{Spec}

\DeclareMathOperator{\ev}{ev}

\DeclareMathOperator{\isom}{isom}

\let\oldenumerate\enumerate
\renewcommand{\enumerate}{
   \oldenumerate
   \setlength{\itemsep}{1pt}
   \setlength{\parskip}{0pt}
   \setlength{\parsep}{0pt}
}
\let\olditemize\itemize
\renewcommand{\itemize}{
   \olditemize
   \setlength{\itemsep}{1pt}
   \setlength{\parskip}{0pt}
   \setlength{\parsep}{0pt}
}

\theoremstyle{plain}
\newtheorem{thm}{Theorem}[section]
\newtheorem{lem}[thm]{Lemma}
\newtheorem{conj}[thm]{Conjecture}
\newtheorem{prop}[thm]{Proposition}
\newtheorem{cor}[thm]{Corollary}

\theoremstyle{definition}
\newtheorem{defn}[thm]{Definition}
\newtheorem{setting}[thm]{Setting}
\newtheorem{rem}[thm]{Remark}

\title[Local ETNC]
{On the Local equivariant Tamagawa number conjecture for Tate motives}

\author{Mahiro Atsuta}
\address{Department of Mathematics, Faculty of Science and Technology, Tokyo University of Science. 
2641 Yamazaki, Noda City, Chiba 278-8510, Japan}
\email{mahiro\_atsuta@rs.tus.ac.jp}

\author{Naoto Dainobu}
\address{Faculty of Mathematics, Kyushu University, 744, Motooka, Nishi-ku, Fukuoka, 819-0395, Japan.}
\email{dainobu.naoto.819@m.kyushu-u.ac.jp}

\author{Takenori Kataoka}
\address{Department of Mathematics, Faculty of Science Division II, Tokyo University of Science.
1-3 Kagurazaka, Shinjuku-ku, Tokyo 162-8601, Japan}
\email{tkataoka@rs.tus.ac.jp}


\makeatletter
\@namedef{subjclassname@2020}{\textup{2020} Mathematics Subject Classification}
\makeatother

\date{\today}

\begin{document}
\begin{abstract}
The local equivariant Tamagawa number conjecture (local ETNC) for a motive predicts a precise relationship between the local arithmetic complex and the root numbers which appear in the (conjectural) functional equations of the $L$-functions. 
In this paper, we prove the local ETNC for the Tate motives under a certain unramified condition at $p$. 
Our result gives a generalization of the previous works by Burns--Flach and Burns--Sano. 
Our strategy basically follows those works and builds upon the classical theory of Coleman maps and its generalization by Perrin-Riou.
\end{abstract}

\maketitle


\section{Introduction}\label{Intro}

One of the main themes in number theory is the relationship between arithmetic objects and the special values of $L$-functions. 
The Tamagawa number conjecture formulated by Block--Kato \cite{BK90} is a strong conjecture on such a relationship, which is then refined by Burns--Flach \cite{BuFl01} to the equivariant Tamagawa number conjecture (ETNC).
The formulation of the ETNC predicts that the determinant module of the associated global arithmetic complex has a basis that is related to the $L$-values.

A local analogue of the ETNC, which we refer to as the local ETNC, is also formulated by Burns--Flach \cite{BuFl01}, Breuning \cite{Bre04}, Fukaya--Kato \cite{FK06}, among others.
The motivation behind the local ETNC is the compatibility between the ETNC, duality, and the functional equation of the $L$-functions. We also note that the local ETNC can be regarded as a combination of both ``local $\varepsilon$-conjecture'' and
``global $\varepsilon$-conjecture'' (cf.~Kato \cite{Katopre}). The formulation of the local ETNC predicts that the determinant module of the local arithmetic complex has a basis that is now related to the root numbers. 

The main theme of this paper is the local ETNC for Tate motives.
One of our motivations is a recent work of Burns--Sano \cite{BS20functional}, whose results can be summarized as follows:
\begin{itemize}
\item[(I)]
They established the local (non-equivariant) Tamagawa number conjecture for $\Z_p (j)$ when $p$ is unramified in the base field (\cite[Theorem 2.3]{BS20functional}).
\item[(II)]
They applied (I) to construct a higher rank Euler system for $\Z_p(1)$ when the base field is moreover totally real (\cite[Theorem 5.2]{BS20functional}).
\end{itemize}
In (I) they did not study equivariant versions, so naturally the resulting Euler system in (II) does not know anything about specialization at characters that do not factor through the $p$-power cyclotomic extension. 
Also, they ignored delicate sign problems.

In this paper, we prove the equivariant refinement of (I) of \cite{BS20functional}, that is, we prove the local ETNC for $\Z_p(j)$, still assuming an unramified condition at $p$.
Throughout, we also take care of the delicate signs.

In a forthcoming paper, we plan to discuss its application to the construction of higher rank Euler systems, which refines (II) of \cite{BS20functional}.

\subsection{Main result}\label{ss:main_lETNC}

Let us briefly explain the statement of the local ETNC (see \S \ref{sec:form_lETNC} for the details).
We fix an odd prime number $p$.

Let $K/k$ be a finite abelian extension of number fields with Galois group $G = \Gal(K/k)$.
We write $S_{\ram}(K/k)$ for the set of finite primes of $k$ that are ramified in $K/k$. 
Let $S$ be a finite set of places of $k$ that contains $S_{\ram} (K/k) \cup S_p \cup S_\infty$, 
where $S_p$ and $S_{\infty}$ respectively denote the sets of $p$-adic primes and archimedean places of $k$. 
We put $S_f := S \setminus S_\infty $

Let $j \in \Z_{\geq 1}$ be a positive integer.
We consider the graded invertible module over $\Z_p[G]$ defined by
\[
\Xi_{K / k, S}^{\loc} (j):=
\parenth{\bigotimes_{v \in S_f} \Det^{-1}_{\Z_p[G]} \RG(K_{v}, \Z_p (j)) }
\otimes
\Det^{-1}_{\Z_p[G]} (X_{K} (j))
\]
with $X_K (j) = \bigoplus_{\iota: K \hookrightarrow \C} \Z_p (j)$, where $\iota$ runs over the embeddings. 
As in Definition \ref{vartheta_j}, 
we have an isomorphism
\[
\vartheta_{K/k, S}^{\loc, j}
  : \C_p \otimes_{\Z_p} \Xi_{K / k, S}^{\loc} (j)
 \xrightarrow{\sim} \C_p [G]. 
\]
Then the local ETNC (Conjecture \ref{conj:local eTNC}) predicts that there is a basis 
\[
z_{K/k, S}^{\loc, j} \in \Xi_{K / k, S}^{\loc} (j)
\]
whose image in $\C_p[G]$ under $\vartheta_{K/k, S}^{\loc, j}$ is described by the root numbers.

Now we state the main result.
As usual we put $\mu_{p^\infty} = \bigcup_{n \geq 0} \mu_{p^n}$, where $\mu_{p^n}$ is the group of the $p^n$-th power roots of unity.
We set $K_n = K(\mu_{p^n})$ and $K_\infty = K(\mu_{p^\infty})$.

\begin{thm}\label{thm:lETNC}
Suppose $k/\Q$ is unramified at $p$.
Let $K/k$ be a finite abelian extension that is unramified at all $p$-adic primes.
Then for any intermediate number field $M$ of $K_\infty/k$, the local ETNC (Conjecture \ref{conj:local eTNC}) for $\Z_p(j)$ is valid for $M/k$ and $j \geq 1$.
Furthermore, we have a basis of the projective limit,
\[
Z_{K_\infty / k, S}^{\loc, j} \in \Xi_{K_\infty / k, S}^{\loc} (j) = \varprojlim_{n \geq 0} \Xi_{K_n / k, S}^{\loc} (j),
\]
whose image to $\Xi_{M / k, S}^{\loc} (j)$ is the desired basis $z_{M/k, S}^{\loc, j}$ for each $M$.
\end{thm}

By specializing to $K = k$, we recover the result of Burns--Sano \cite{BS20functional}.
See \S \ref{ss:previous} for relations with other previous works.

Let us briefly sketch the proof of this theorem.
The invertible module $\Xi_{K_\infty / k, S}^{\loc} (j)$ is the tensor product of local determinant modules with respect to places of $k$, and our construction of the basis $Z_{K_\infty / k, S}^{\loc, j}$ is naturally a combination of bases for those places.
Here we explain key ingredients to construct the local bases.
\begin{itemize}
\item
For the $p$-adic primes, we use the classical theory of Coleman map \cite{Col79} that describes the projective limit of the local unit groups (\S \ref{sec:Coleman}).
The unramified condition at $p$ is required here.
\item
For the non-$p$-adic finite primes, we use a result of Greither, Kurihara, and the third author \cite{GKK_09} (\S \ref{sec:non p adic prime}).
\item
For the archimedean places, a basis can be constructed quite easily.
However, we need appropriate modification to obtain the final desired formula (\S \ref{sec:arch}, \S \ref{sec:pf_lETNC}).
\end{itemize}

After completing the proof of Theorem \ref{thm:lETNC}, we will show additional properties of $Z_{K_\infty / k, S}^{\loc, j}$.
One is the behavior under the twists, i.e., the relations between different $j$'s (\S \ref{App:twist}).
We will also give a definition of $Z_{K_\infty / k, S}^{\loc, j}$ for $j \in \Z_{\leq 0}$ by using the local Tate duality and then observe its properties (\S \ref{section:j <1}).

\subsection{Relation with previous works}\label{ss:previous}

Let us explain relations with previous works on the local ETNC for $\Z_p(j)$.

As already mentioned, the work \cite{BS20functional} of Burns--Sano is a prototype of this present work. 
We give equivariant refinements without sign ambiguity.

Burns--Flach \cite{BuFla06} proved the local ETNC for Tate motives when the base field is $k = \Q$.
This also serves as a prototype; indeed, the discussion on $p$-adic primes (\S \ref{sec:Coleman}) is a direct generalization of \cite[\S 5 and \S 6]{BuFla06}. 

We briefly explain some difficulties in generalizing the base field from $\Q$.
As is well-known, the functional equations of $L$-functions involve Gauss sums of Dirichlet characters.
Therefore, to prove Theorem \ref{thm:lETNC}, 
we have to handle the Gauss sums for any $K_n / k$ for any $n \in \Z_{\geq 0}$ in a uniform manner. 
To do this, we need to generalize the Davenport--Hasse relation; such a generalization was provided to the authors by Daichi Takeuchi (\S \ref{sec:DH}).
Also, after careful computations, we will notice that a first candidate for the element $Z_{K_{\infty}/k, S}^{\loc, j}$ does not satisfy the desired interpolation properties.
We then need to modify it by multiplying an auxiliary unit of the algebra $\Z_p[[\Gal (K_\infty / k)]]$, whose existence is not obvious. 

In essence, our formulation of the local ETNC coincides with the ``local $\varepsilon$-conjecture'' and ``global $\varepsilon$-conjecture'' formulated by Kato \cite{Katopre}, when applied to Tate motives.
The local $\varepsilon$-conjecture asserts the existence of a canonical basis of the determinant module of a cohomology complex associated to any $p$-adic representation of $\Gal(\ol{\Q_l}/\Q_l)$ in both cases $l=p$ and $l \neq p$. 
The global $\varepsilon$-conjecture asserts that these conjectural bases are compatible with 
a zeta element, which is also a conjectural basis of the determinant module of a certain Galois cohomology complex and is expected to be related to the special values of the $L$-functions. 
This compatibility has a close connection with the functional equation of the $L$-functions.
In \cite{Katopre} Kato proved these conjectures in a special case using the theory of Coleman maps, and our discussion in \S \ref{sec:Coleman} generalizes this argument.
We also note that Kato's formulation also involves compatibility with deformations of representations, so it is much stronger than the local ETNC. 

Note also that, precisely speaking, for each finite prime, the local basis we will construct and that predicted by the local $\varepsilon$-conjecture (cf.~\cite{FK06} and \cite{Katopre}) coincide only up to units.
The latter basis is related to the local $\varepsilon$-factor (i.e., the Gauss sum), while our bases for non-$p$-adic finite primes ignore this (\S \ref{sec:non p adic prime}) and, instead, the bases for $p$-adic primes involve the Gauss sums for all finite primes (\S \ref{sec:arch} and \S \ref{sss:p}).
The results are consistent after combining all the places.

Concerning the $p$-adic local case (\S \ref{sec:Coleman}), Benois--Berger \cite{BeBe08} and Loeffler--Venjakob--Zerbes \cite{LVZ15} 
proved the local $\varepsilon$-conjecture for crystalline representations of 
$G_{\Q_p}$. 
It would be natural to apply their results, but their formulations are substantially different from ours, and the precise relation remains unclear.

\subsection{Notation}\label{ss:notation}

Throughout this paper, we use the following notation. 

Let $G$ be a finite abelian group.
For a commutative ring $R$, we write
\[
\# : R[G] \to R[G], 
\quad
x \mapsto x^\#
\]
for the $R$-algebra automorphism of the group ring $R[G]$ (the involution)
defined by $\sigma \mapsto \sigma^{-1}$ for any $\sigma \in G$. 
For any $R[G]$-module $M$, 
let $M^\#$ be the $R[G]$-module that is the same as $M$ as an $R$-module but $G$ acts on $M^\#$ via the involution $\#$. 
Similar notations will be used for the completed group ring $R[[\GG]]$ when $\GG$ is a profinite group

For any subgroup $H$ of $G$, we set
\[
N_H := \sum_{\sigma \in H} \sigma \in \Z[G],
\qquad
e_H := \frac{N_H}{\# H} \in \Q[G].
\]
This $e_H$ is an idempotent element.

For character $\chi$ of $G$, we write $\Q (\chi) := \Q (\Imag(\chi))$ and define
\[
e_\chi := \frac{1}{\# G}\sum_{\sigma \in G} \chi (\sigma) \sigma^{-1} \in \Q(\chi)[G]
\]
for the associated idempotent element. 
We regard $\Q(\chi)$ as a $\Q[G]$-algebra via $\chi$. 
For a $\Q [G]$-module $M$, 
put $M^\chi := \Q (\chi) \otimes_{\Q[G]} M$, and for an element $a \in M$, 
we write $a^\chi \in M^\chi$ for the $\chi$-component of $a$.

\subsection{Acknowledgments}\label{s:ack}
We are very grateful to Daichi Takeuchi for providing the proof of a generalization of the Davenport--Hasse relation (Lemma \ref{lem:Davenport--Hasse} (iii), Proposition \ref{app:HD}) 
to the authors. 
We are also very grateful to Kentaro Nakamura and Takamichi Sano for their very helpful advice.
We also thank Masato Kurihara for his encouragement during this project. 
This work was supported by JSPS KAKENHI Grant Numbers 22K13898, 23KJ1943, 25KJ0259.

\section{Formulation of the local ETNC}\label{sec:form_lETNC}

In this section, we formulate the local equivariant Tamagawa number conjecture (local ETNC) for Tate motives. 
Our formulation is an equivariant version of the local Tamagawa number conjecture for Tate motives which is formulated by Burns and Sano in
\cite[Conjecture 2.2]{BS20functional}. 
The local ETNC for general motives is formulated in \cite[Conjecture 8]{BuFl01}.
The formulation requires the theory of determinant modules; throughout, we follow the convention of Knudsen--Mumford \cite{KM76}.
See \S \ref{App:det} for the convention.

We fix an algebraic closure $\ol{\Q}$ of $\Q$ and regard $\ol{\Q} \subset \C$. 
We fix an odd prime number $p$.
We fix an isomorphism $\C \simeq \C_p$ and identify these fields.

Let $K/k$ be a finite abelian extension of number fields with Galois group $G = \Gal(K/k)$.

\subsection{The local complexes}\label{ss:setup}

 Let $v$ be a finite prime of $k$. We set $K_v = k_v \otimes_k K \simeq \prod_{w \mid v} K_w$, where $w$ runs over the primes of $K$ lying above $v$. For any $j \in \Z$,
we consider the semi-local Galois cohomology complex
\[
\Delta_{K_v} (j) := \RG (K_v, \Z_p (j))
= \bigoplus_{w|v}\RG (K_w, \Z_p (j)).
\]
This is a perfect complex over $\Z_p[G]$ and, more precisely, quasi-isomorphic to a perfect complex over $\Z_p[G]$ which concentrates on degrees $0, 1$, and $2$.

The purpose of this subsection is to introduce, for $j \geq 1$, the construction of the isomorphisms 
\[
\begin{cases}
\vartheta_{K_v}^{j} : 
\Q_p \otimes_{\Z_p} \Det_{\Z_p[G]}^{-1} (\Delta_{K_v}(j))
\xrightarrow{\sim} 
\Q_p [G]
&\text{if } v \nmid p, \\
\phi_{K_v}^j: \Q_p \otimes_{\Z_p} \Det_{\Z_p[G]}^{-1} (\Delta_{K_v}(j))
\xrightarrow{\sim} 
\Det_{\Q_p[G]}(K_v)
&\text{if }v \mid p.
\end{cases}
\]

We note that the isomorphism $\Q_p \otimesL_{\Z_p} \Delta_{K_v}(j) \simeq \RG (K_v, \Q_p (j))$ induces
\[
\Q_p \otimes_{\Z_p} \Det_{\Z_p[G]}^{-1} (\Delta_{K_v}(j)) \simeq \Det_{\Q_p[G]}^{-1} (\RG (K_v, \Q_p (j))).
\]
The construction of the isomorphisms $\vartheta_{K_v}^{j}$ and $\phi_{K_v}^j$ depends on whether $j = 1$ or $j \geq 2$,in addition to whether $v \mid p$ or $v \nmid p$.

Suppose $j \geq 2$ and $v \nmid p$.
In this case, the cohomology group $H^i(K_v, \Z_p (j))$ is finite for all $i\in\{0,1,2\}$ and (equivalently) $\Q_p \otimesL_{\Z_p} \Delta_{K_v}(j)$ is acyclic.
We then define $\vartheta_{K_v}^{j}$ as the canonical isomorphism as in Proposition \ref{prop:det_ses3}.

Suppose $j \geq 2$ and $v \mid p$. 
Then $H^i (K_v, \Z_p (j))$ is finite except for $i = 1$, and we have a canonical isomorphism
\[
\Q_p \otimes_{\Z_p} \Det^{-1}_{\Z_p[G]} (\Delta_{K_v}(j)) \simeq \Det_{\Q_p[G]}(H^1(K_v, \Q_p(j)))
\]
by Proposition \ref{prop:det_ses4}. 
Since we have $H^1_f (K_v, \Q_p(j)) = H^1 (K_v, \Q_p(j))$ by $j \geq 2$ for the finite part, we define the isomorphism $\phi_{K_v}^j$ by composing the above isomorphism and the Bloch--Kato logarithm map 
\[
\log_{\Q_p (j)} : H^1_f(K_v, \Q_p(j))
 \xrightarrow{\sim} K_v.
\]

Suppose $j = 1$.
The cohomology groups of $\Delta_{K_v}(1)$ are described as
\[
H^i (K_v, \Z_p (1)) \simeq 
\begin{cases}
\bigoplus_{w|v}\widehat{K_w^{\times}} & \text{if } i=1, \\
\bigoplus_{w|v}\Z_p  & \text{if } i=2, \\
0 & \text{otherwise.} 
\end{cases}
\]
Here, $\widehat{K_w^{\times}}$ denotes the $p$-adic completion of $K_w^{\times}$.
The isomorphisms for $i = 1, 2$ are given by the Kummer map and the invariant map ${\rm inv}_{K_v} = \bigoplus_{w \mid v} {\rm inv}_{K_w}$ in the local class field theory, respectively. In addition, we note that we have the normalized valuation map
\[
\ord_{K_v}: \bigoplus_{w|v} \widehat{K_w^{\times}} \twoheadrightarrow \bigoplus_{w|v} \Z_p.
\]

Now suppose $j = 1$ and $v \nmid p$.
Then the kernel of $\ord_{K_v}$ is finite, so we can construct an isomorphism
\begin{equation}\label{eq:ord_inv_l}
\Q_p \otimes_{\Z_p} H^1 (K_v, \Z_p (1)) 
\overset{{\rm inv}_{K_v}^{-1} \circ \ord_{K_v}}{\xrightarrow{\sim}} \Q_p \otimes_{\Z_p} H^2 (K_v, \Z_p (1)). 
\end{equation}
Using this, we define the isomorphism $\vartheta_{K_v}^{1}$ as 
\begin{align*}
\vartheta_{K_v}^{1} 
& : \Q_p \otimes_{\Z_p} \Det_{\Z_p[G]}^{-1} (\Delta_{K_v}(1))
\simeq  \Det_{\Q_p[G]}^{-1} (\RG (K_v, \Q_p (1)))\\
&\simeq  \Det_{\Q_p[G]}(H^1(K_v, \Q_p(1)))\otimes_{\Q_p[G]} \Det_{\Q_p[G]}^{-1}(H^2(K_v, \Q_p(1)))\\
&\simeq \Det_{\Q_p[G]}(H^2(K_v, \Q_p(1)))\otimes_{\Q_p[G]}\Det_{\Q_p[G]}^{-1}(H^2(K_v, \Q_p(1))) 
\overset{\ev}{\simeq} \Q_p[G]
\end{align*}
by using Proposition \ref{prop:det_ses4} and the evaluation map in \S \ref{App:det_defn}.

Finally, we suppose $j = 1$ and $v \mid p$. 
Then we have an exact sequence
\begin{equation}\label{eq:H1_Kum}
0 \to \bigoplus_{w|v} U^1_{K_w} \to \bigoplus_{w|v} \widehat{K_w^{\times}} \overset{\ord_{K_v}}{\to} 
\bigoplus_{w|v} \Z_p \to 0,
\end{equation}
where $U^1_{K_v} = \bigoplus_{w \mid v} U^1_{K_w}$ is the direct sum of the principal unit groups.
We also have
\[
H^1_f (K_v, \Q_p(1)) = \Q_p \otimes_{\Z_p} U^1_{K_v}
\]
via the Kummer map. Then the above exact sequence yields an exact sequence
\begin{equation}\label{Kummer seq}
0 \to H^1_f (K_v, \Q_p(1)) 
\to \Q_p \otimes_{\Z_p} H^1 (K_v, \Z_p (1)) 
\overset{{\rm inv}_{K_v}^{-1} \circ \ord_{K_v}}{\to} \Q_p \otimes_{\Z_p} H^2 (K_v, \Z_p (1)) \to 0.
\end{equation}
Using Proposition \ref{prop:det_ses2} applied to this sequence, 
we define $\phi_{K_v}^1$ as
\begin{align*}
\phi_{K_v}^1 
&: \Q_p \otimes_{\Z_p} \Det^{-1}_{\Z_p[G]} (\Delta_{K_v}(1))
\simeq \Det_{\Q_p[G]}^{-1} (\RG (K_v, \Q_p (1)))\\
&\simeq \Det_{\Q_p[G]}(H^1(K_v, \Q_p(1)))\otimes_{\Q_p[G]} \Det_{\Q_p[G]}^{-1}(H^2(K_v, \Q_p(1)))\\
&\simeq \Det_{\Q_p[G]}(H^1_f(K_v, \Q_p(1)))\otimes_{\Q_p[G]} \Det_{\Q_p[G]}(H^2(K_v, \Q_p(1)))\otimes_{\Q_p[G]} \Det_{\Q_p[G]}^{-1}(H^2(K_v, \Q_p(1)))\\
 & \overset{\ev}{\simeq} \Det_{\Q_p[G]}(H^1_f(K_v, \Q_p(1))) \overset{\log_{\Q_p (1)}}{\simeq} \Det_{\Q_p[G]}(K_v)
\end{align*}
by Proposition \ref{prop:det_ses4}.
The final isomorphism is induced by the Bloch--Kato logarithm $\log_{\Q_p (1)} : H^1_f (K_v, \Q_p(1)) \xrightarrow{\sim} K_v$. 

\subsection{The period map}\label{period map}

We write $\Sigma_K = \{ \iota : K \hookrightarrow \C \}$ for the set of the embeddings, so $\# \Sigma_K  =[K: \Q]$. 
We put $r_k := [k : \Q]$. 

The $\Z_p$-module $\Z_p (1) = \varprojlim_n \mu_{p^n}(\C)$ has a basis
\[
e_{p^\infty} := \left(\exp \left(\frac{2 \pi  \sqrt{-1}}{p^n} \right)\right)_n.
\]
Then for each $j \in \Z$, we have a $\Z_p$-basis $e_{p^\infty}^{\otimes j}$ of $\Z_p(j) = \Z_p(1)^{\otimes j}$. 
For each $j \in \Z$, 
we put 
\[
X_K (j) = \bigoplus_{\iota \in \Sigma_K} \Z_p (j). 
\]
As in \cite[\S 2.1]{BKS20}, we use $\{ e_\iota^j \mid \iota \in \Sigma_K\}$ as a $\Z_p$-basis of $X_K (j)$, where $e_{\iota}^j$ is the element of $X_K (j)$ whose 
$\iota$-component is $e_{p^\infty}^{\otimes j}$ and the other components are zero. 

We also consider the Betti cohomology space
\[
H_K (j) := \bigoplus_{\iota \in \Sigma_K} (2\pi \sqrt{-1})^j \Q
\]
for any $j \in \Z$. 
We use $\{ b_\iota^j \mid \iota \in \Sigma_K\}$ as a $\Q$-basis of $H_K (j)$, where $b_{\iota}^j$ is the element of $H_K (j)$ whose $\iota$-component is $(2 \pi \sqrt{-1})^j$ and the other components are zero. 

The group $G$ acts on $\Sigma_K$ by $\sigma \cdot \iota := \iota \circ \sigma^{-1}$ for any $\sigma \in G$ and $\iota \in \Sigma_K$.
Then $G$ also acts on $X_K(j)$ and $H_K (j)$ by $\sigma \cdot (a_\iota)_\iota := (a_\iota)_{\iota \circ \sigma^{-1}}$, making them free modules of rank $r_k$ over $\Z_p [G]$ and $\Q[G]$, respectively.
Moreover, we obtain a $\Q_p [G]$-isomorphism
\begin{equation}\label{Y Betti}
\Q_p \otimes_{\Z_p} X_K(j)
\stackrel{\sim}{\rightarrow}  \Q_p \otimes_\Q H_K (j)
\end{equation}
by sending $1 \otimes e_{\iota}^j$ to $1 \otimes b_{\iota}^j$. 

The complex conjugation $c_{\R} \in \Gal (\C / \R)$ acts on $H_K (j)$ by 
$c_{\R} \cdot(a_\iota)_\iota := (c_{\R} ( a_\iota))_{c_{\R} \circ \iota}$. 
We write $H_K (j)^+$ for the subspace of $H_K (j)$ on which $c_{\R}$ acts trivially. 
Then we have an $\R[G]$-isomorphism 
\begin{equation}\label{embedding}
\R \otimes_\Q K \simeq
\R \otimes_\Q (H_K (j)^+ \oplus H_K (j-1)^+ )
\end{equation}
that is given by 
\[
1 \otimes x \mapsto \begin{cases}
\left(\left(\Real(\iota(x))\right)_{\iota \in \Sigma_K}, \left(\sqrt{-1} \Imag(\iota(x)) \right)_{\iota \in \Sigma_K} \right)
& \text{if } j \text{ is even,} \\
\left(\left(\sqrt{-1} \Imag(\iota(x)) \right)_{\iota \in \Sigma_K}, \left( \Real(\iota(x))\right)_{\iota \in \Sigma_K} \right)
& \text{if } j \text{ is odd} 
\end{cases}
\]
for $x \in K$.
By composing this with the isomorphism
\begin{equation}\label{Hj}
\C \otimes_\Q (H_K (j)^+ \oplus H_K (j-1)^+)
 \stackrel{\sim}{\rightarrow} \C \otimes_\Q H_K (j); \;
1 \otimes  (x,y) \mapsto 2x + \frac{1}{2} \cdot 2 \pi \sqrt{-1} y, 
\end{equation}
we obtain a $\C[G]$-isomorphism (period map)
\[
\alpha^j_K: \C \otimes_\Q K \stackrel{\sim}{\rightarrow} \C \otimes_\Q H_K (j).
\]
By a direct computation, we obtain an explicit formula
\begin{equation}\label{alpha}
\alpha_K^j (1 \otimes x) = 
\begin{cases}\displaystyle
\sum_{\iota \in \Sigma_K}
\parenth{\big{(}2 \Real (\iota (x)) - \pi \Imag (\iota (x))\big{)}\cfrac{1}{(2\pi \sqrt{-1})^j}}b_{\iota}^j 
& \text{if } j \text{ is even,} \\
\displaystyle
\sum_{\iota \in \Sigma_K}
\parenth{\big{(}2 \sqrt{-1} \Imag (\iota (x)) + \pi \sqrt{-1} \Real (\iota (x))\big{)}\cfrac{1}{(2\pi \sqrt{-1})^j}}
b_{\iota}^j &  \text{if } j \text{ is odd}
\end{cases}
\end{equation}
for $x \in K$.

\subsection{The module $\Xi$}\label{ss:Xi}

As in the introduction, let $S_p$ (resp.~$S_\infty$) be the set of all $p$-adic primes (resp.~all archimedean places) of $k$. 
Let $S$ be a finite set of places of $k$ such that 
$S$ contains $S_p \cup S_\infty \cup S_{\ram} (K/k)$ and set $S_f := S \setminus S_\infty$.

\begin{defn}\label{def:Xi}
For any $j\in \Z$, 
we define a graded invertible module $\Xi_{K / k, S}^{\loc} (j)$ as
\[
\Xi_{K / k, S}^{\loc} (j):= \parenth{\bigotimes_{v \in S_f} \Det_{\Z_p[G]}^{-1} (\Delta_{K_v}(j))}
\otimes_{\Z_p[G]}
\Det_{\Z_p[G]}^{-1} (X_K (j))
\]
by using the semi-local complexes $\Delta_{K_v}(j) = \RG (K_v, \Z_p (j))$.
\end{defn}

\begin{rem}\label{rem:grade_zero}
The grade of $\Xi_{K / k, S}^{\loc} (j)$ is zero because
\begin{itemize}
\item
the Euler characteristic of $\Delta_{K_v}(j)$ is $-[k_v: \Q_p]$ when $v \mid p$,
\item
the Euler characteristic of $\Delta_{K_v}(j)$ is zero when $v \nmid p$,
\item
the rank of $X_K (j)$ is $r_k = [k: \Q]$,
\end{itemize}
and $[k: \Q] = \sum_{v \mid p} [k_v: \Q_p]$.
The first two items follow from the local Euler--Poincar\'{e} characteristic formula.
\end{rem}

We set 
\[
K_p := \Q_p \otimes_{\Q} K 
\simeq \prod_{v \in S_p} K_v
\simeq \prod_{w \mid p} K_w,
\]
where $v$ runs over $p$-adic primes of $k$, whereas $w$ runs over $p$-adic primes of $K$.

\begin{defn}\label{vartheta_j}
For any $j \in \Z_{\geq 1}$, we define an isomorphism  
\[
\vartheta_{K/k, S}^{\loc, j} : 
\C_p \otimes_{\Z_p} \Xi_{K/k, S}^{\loc} (j) \stackrel{\sim}{\rightarrow}  \C_p [G]
\]
as follows.
Firstly, the isomorphisms $\vartheta_{K_v}^j$ and $\phi_{K_v}^j$ in \S \ref{ss:setup} induce
\begin{align}\label{use_later}
\Q_p \otimes_{\Z_p} \bigotimes_{v \in S_f} \Det_{\Z_p[G]}^{-1} (\Delta_{K_v}(j))
& \simeq \bigotimes_{v \in S_p} \Det_{\Q_p[G]} (K_v) \otimes_{\Q_p[G]} \bigotimes_{v \in S_f \setminus S_p} \Q_p[G]\\
& \simeq \Det_{\Q_p[G]}(K_p)
\simeq \Q_p \otimes_{\Q} \Det_{\Q[G]}(K).   
\end{align}
Combining these, we can define $\vartheta_{K/k, S}^{\loc, j}$ as the composition
\begin{align*}
\C_p \otimes_{\Z_p} \Xi_{K/k, S}^{\loc} (j)
& \overset{\eqref{Y Betti}\text{ and }\eqref{use_later}}{\simeq}
\left(\C_p \otimes_{\Q} \Det_{\Q[G]}(K) \right) \otimes_{\C_p[G]} 
\left(\C_p \otimes_{\Q} \Det_{\Q[G]}^{-1} (H_K(j)) \right) \\
&\overset{\wedge \alpha_K^j}{\simeq}
\left(\C_p \otimes_{\Q} \Det_{\Q[G]}(H_K (j)) \right) \otimes_{\C_p[G]} 
\left(\C_p \otimes_{\Q} \Det_{\Q[G]}^{-1} (H_K(j)) \right)\\ 
& \overset{\ev}{\simeq} \C_p[G],
\end{align*}
where the isomorphism $\alpha_K^j$ is applied by using the fixed isomorphism $\C \simeq \C_p$.
\end{defn}

\subsection{The $L$-functions}\label{ss:L-fct}

We introduce the $L$-functions. 
For a finite prime $v$ of $k$, we define $N(v)$ as the cardinality of the residue field of $k$ at $v$.
For any $\C$-valued character $\chi$ of $G$ and any finite set $S$ of primes of $k$, 
we define the $L$-function $L_S(\chi, s)$ by
\[
L_S(\chi, s) 
= \prod_{v \notin S} \parenth{1 - \chi(v)N(v)^{-s}}^{-1},
\]
where $v$ runs over the finite primes of $k$ not in $S$ and we set $\chi(v) = 0$ if $\chi$ is ramified at $v$.  
Here we note that $\chi (v)$ means the image of the {\it arithmetic} Frobenius of $v$ by the character $\chi$. 
It is well-known that this infinite product converges absolutely for 
$s \in \C$ whose real part is larger than $1$,
and has a meromorphic continuation to the whole complex plane $s \in \C$. 
We write $L(\chi, s):= L_{\emptyset}(\chi, s)$.

For any integer $j$, we write $L_S^\ast(\chi, j) \in \C^\times $ 
for the leading term of the Taylor expansion of  $L_S (\chi, s)$ at $s=j$. 
Here we note that $L_S (\chi, s)$ has a pole only if $\chi$ is a trivial character $\trivial$, in which case $L_S({\trivial}, s)$ has a simple pole at $s = 1$
 (note that $L ({\trivial}, s)$ is the Dedekind zeta function of $k$). 
Using this, we  define 
\[
\Theta_{K/k, S}^\ast (j) := \sum_{\chi} L_S^\ast(\chi, j) e_\chi 
\in \C[G]^\times, 
\]
where $\chi$ runs over the characters of $G$ and $e_\chi := \frac{1}{\# G}\sum_{\sigma \in G} \chi (\sigma) \sigma^{-1}$ is the associated idempotent. 
We write $\Theta_{K/k}^\ast (j) := \Theta_{K/k, \emptyset}^\ast (j)$. 

\subsection{Local equivariant Tamagawa number conjecture}\label{local eTNC}

For each finite prime $v$ of $k$, let $D_v$ and $I_v$ be the decomposition group and the inertia group of $v$ in $G$ respectively, and $\sigma_{K / k, v} \in G / I_v$ the arithmetic Frobenius of $v$.

For any integer $j$ and any finite prime $v$ of $k$, 
we write $\epsilon_v^j (K/k) \in \Q[G]$ for the Euler factor for $\Z_p (j)$, which is given by
\[
\epsilon_v^j (K/k) := 
1 - e_{I_v} N (v)^{-j} \sigma_{K/k, v}
\]
(recall that $e_{I_v} = N_{I_v}/\# I_v$ is the idempotent element and note that $e_{I_v} \sigma_{K / k, v}\in \Q[G]$ is well-defined). 
Also, we put
\[
\delta_v (K/k):=
\epsilon_v^0 (K/k) + e_{D_v} \cdot \cfrac{1}{\# (D_v / I_v)}. 
\]
We define $\delta_{K/k, {\trivial}} \in \Q [G]^\times$ as the element satisfying
\[
\chi (\delta_{K/k, {\trivial}}) = \begin{cases}
-1 & \text{if } \chi \text{ is trivial,}
\\
1 & \text{otherwise} .
\end{cases}
\]

Then the local equivariant Tamagawa number conjecture for $K/k$ is formulated as follows. 
Recall that $S$ is a finite set of places of $k$ with $S \supset S_p \cup S_\infty \cup S_{\ram} (K/k)$ and we put $S_f := S \setminus S_\infty$.

\begin{conj}[local equivariant Tamagawa number conjecture]\label{conj:local eTNC}
For any $j \in \Z_{\geq 1}$, 
there is a (unique) $\Z_p[G]$-basis $z_{K/k, S}^{\loc, j} \in \Xi_{K/k, S}^{\loc} (j) $ 
such that we have 
\[
\vartheta_{K/k, S}^{\loc, j} (z_{K/k, S}^{\loc, j}) 
= (-1)^{r_k (j-1)} \times
\begin{cases}\displaystyle
\cfrac{\Theta_{K/k, S}^\ast (1-j)^\#}{\Theta_{K/k, S}^\ast (j)}
& \text{if } j \geq 2,  \\
\displaystyle
\delta_{K/k, {\trivial}} \cdot 
\left(\prod_{v \in S_f} 
\delta_v(K/k)^\# \right)
\cfrac{\Theta_{K/k}^\ast (0)^\#}{\Theta_{K/k, S}^\ast (1)}
 & \text{if }  j=1.
\end{cases}
\]
\end{conj}
\begin{rem}
When $K=k$, this conjecture is formulated in Burns--Sano \cite[Conjecture 2.2]{BS20functional}.
However, there is a small difference between their formulation and ours;
when $K=k$ and $j=1$, theirs replaces $\Theta_{K/k}^\ast (0)^\#$ by $\Theta_{K/k, S}^\ast (0)^\#$.
The current authors believe that our formulation is correct.
\end{rem}

\subsection{Changes of $K$ and $S$}

For completeness, we establish functorial properties of Conjecture \ref{conj:local eTNC} when $K$ and $S$ vary. 

\begin{prop}\label{prop:localETNC_func}
Let $K'$ be an intermediate field of $K/k$ and set $G' = \Gal(K'/k)$.
If Conjecture \ref{conj:local eTNC} is valid for $K / k$, then it is also valid for $K'/k$.
More precisely, the image of $z_{K / k, S}^{\loc, j} \in \Xi_{K/k, S}^{\loc} (j)$ under the natural map
\[
\Xi_{K/k, S}^{\loc} (j) \twoheadrightarrow \Z_p [G'] \otimes_{\Z_p[G]} \Xi_{K/k, S}^{\loc} (j) 
\simeq \Xi_{K'/k, S}^{\loc} (j)
\]
satisfies the required property of $z_{K' / k, S}^{\loc, j}$.
\end{prop}

To show this proposition, we need the following local observation.

\begin{lem}\label{lem:compatibility1}
Let $K'$ be an intermediate field of $K/k$ and set $G' = \Gal(K'/k)$.
Let $v$ be a finite prime of $k$ and $j \in \Z_{\geq 1}$.
We define $d_v^j (K / K') \in \Q [G']^\times$ by  
\[
d_v^j (K / K') := \begin{cases}
(1- e_{D_v'}) +  f_{K / K', v} \cdot e_{D_v'} & \text{if } j =1, \\
1  & \text{if } j \geq 2, 
\end{cases}
\]
where $f_{K / K', v} \in \Z_{\geq 1}$ is the residue degree of $K/K'$ at $v$ and 
$D_v'$ is the decomposition group of $v$ in $G'$. 
Then we have a commutative diagram
\[\xymatrix{
\Q_p \otimes_{\Z_p} \Det_{\Z_p[G]}^{-1} (\Delta_{K_v}(j)) \ar[r]^-{\vartheta_{K_v}^j}_-{\sim} 
\ar@{->>}[d]
& \Q_p [G] \ar@{->>}[d]^-{d_v^j (K / K') \cdot \res_{K/ K'}} 
\\
\Q_p \otimes_{\Z_p} \Det_{\Z_p[G']}^{-1} (\Delta_{K'_v}(j)) \ar[r]^-{\vartheta_{K_v'}^j}_-{\sim} 
& \Q_p [G']
}
\]
when $v \nmid p$, and
\[
\xymatrix{
\Q_p \otimes_{\Z_p} \Det_{\Z_p[G]}^{-1} (\Delta_{K_v}(j)) \ar[r]^-{\phi_{K_v}^j}_-{\sim} 
\ar@{->>}[d]
& \Det_{\Q_p[G]} (K_v) \ar@{->>}[d]^-{d_v^j (K / K') \cdot \wedge \Tr_{K / K'}} 
\\
\Q_p \otimes_{\Z_p} \Det_{\Z_p[G']}^{-1} (\Delta_{K'_v}(j)) \ar[r]^-{\phi_{K_v'}^j}_-{\sim} 
& 
\Det_{\Q_p[G']} (K_v')
}
\]
when $v \mid p$,
where the left vertical arrows are induced by the corestriction maps.
\end{lem}

\begin{proof}
When $j \geq 2$, the claim for $v \nmid p$ is clear, while the claim for $v \mid p$ also follows from the commutative diagram
\[
\xymatrix{
H^1_f (K_v , \Q_p (j)) 
\ar[r]^-{\log_{\Q_p (j)}} \ar[d]_-{\cores_{K/K'}}
&K_v \ar[d]^-{\Tr_{K / K'}} \\
H^1_f (K_v' , \Q_p (j)) 
\ar[r]^-{\log_{\Q_p (j)}}
&K_v'.
}
\]
This diagram is also valid for $j = 1$.
When $j = 1$, both the claims for $v \nmid p$ and $v \mid p$ follow by using the commutative diagram
\[
\xymatrix@C=40pt@M=8pt{
\Q_p \otimes_{\Z_p} H^1 (K_v, \Z_p (1)) 
\ar[r]^-{{\rm inv}_{K_v}^{-1} \circ \ord_{K_v}} \ar[d]^-{ \cores_{K / K'}}
&\Q_p \otimes_{\Z_p} H^2 (K_v, \Z_p (1)) \ar[d]^-{f_{K/K', v}  \cdot \cores_{K/ K'}} \\
\Q_p \otimes_{\Z_p} H^1 (K'_v, \Z_p (1)) 
\ar[r]^-{{\rm inv}_{K_v'}^{-1} \circ \ord_{K_v'}}
&\Q_p \otimes_{\Z_p} H^2 (K'_v, \Z_p (1)),   
}
\]
which is a well-known property of the invariant maps in local class field theory. 
\end{proof}

\begin{proof}[Proof of Proposition \ref{prop:localETNC_func}]
As for the archimedean places, we have a commutative diagram 
\[
\xymatrix{
\C \otimes_\Q K  \ar[r]^-{\alpha_{K}^j}_-{\sim} \ar[d]_-{\Tr_{K / K'} }
&
\C \otimes_\Q H_{K} (j) \ar[d] 
\\
\C \otimes_\Q K'  \ar[r]^-{\alpha_{K'}^j}_-{\sim} 
&
\C \otimes_\Q H_{K'} (j), 
}\]
where the right vertical arrow sends $ (a_{\iota} )_{\iota : K \hookrightarrow \C} $ to 
$  \parenth{ \sum_{\iota|_{K'} = \iota'} a_{\iota}}_{\iota' : K' \hookrightarrow \C}$. 
By combining with Lemma \ref{lem:compatibility1}, we obtain a commutative diagram
\[
\xymatrix{
\C_p \otimes_{\Z_p}\Xi_{K/k, S}^{\loc} (j) \ar[r]^-{\vartheta_{K/k, S}^{\loc, j}}_-{\sim} \ar@{->>}[d]
& \C_p [G] \ar@{->>}[d]^-{\prod_{v \in S_f } d_v^j (K / K') \cdot \res_{K/K'}} 
\\
\C_p \otimes_{\Z_p} \Xi_{K'/k, S}^{\loc} (j) \ar[r]^-{\vartheta_{K'/k, S}^{\loc, j}}_-{\sim}
& \C_p [G']. 
}
\]
Now what we have to show is, for $j = 1$
\[
\delta_v (K' / k)^\# = d_v^1 (K / K') \cdot \res_{K/K'} (\delta_v (K/k)^\#) 
\] 
holds for any finite prime $v$ of $k$.
This equality can be checked by a direct computation.
\end{proof}

\begin{prop}\label{prop:S}
The validity of Conjecture \ref{conj:local eTNC} does not depend on the choice of $S$. 
\end{prop}

\begin{proof}
It is enough to show that, for each prime $v \not \in S_p \cup S_{\ram}(K/k)$, there is a basis of 
$\Det_{\Z_p[G]}^{-1} (\RG(K_v, \Z_p (j)))$ that is sent by $\vartheta_{K_v}^j$ to 
\[
\begin{cases}
	\cfrac{\epsilon_v^{1-j} (K / k)^\#}{\epsilon_v^{j} (K / k)} 
		& \text{if } j \geq 2, \\
	\cfrac{\delta_v (K/k)^\#}{\epsilon_v^{1} (K / k)} 
		& \text{if } j =1.
\end{cases}
\]
This follows from Corollary \ref{cor:unr_basis} below and we omit the details for now.
\end{proof}

\subsection{Iwasawa theoretic objects}\label{ss:Iw_limit}

For later use, we introduce Iwasawa-theoretic versions of the complex $\Delta_{K_v}(j)$ and the module $X_K(j)$.

Set $K_n = K(\mu_{p^n})$, $K_{\infty} = K(\mu_{p^{\infty}})$, and $\Lambda := \Z_p [[\Gal (K_\infty / k)]]$. 
Then for $j \in \Z$ and any finite prime $v$ of $k$, 
we have the semi-local Iwasawa cohomology complex
\[
\Delta_{K_\infty, v} (j) = \RG_{\Iw} (K_{\infty, v}, \Z_p(j)),
\]
whose cohomology groups are the projective limits of those of $\Delta_{K_n, v}(j) = \RG(K_{n, v}, \Z_p(j))$. Concretely, when $j = 1$, their cohomology groups are described as
\begin{equation}\label{Iwasawa cohomology}
H^i (\Delta_{K_\infty, v} (1)) \simeq 
\begin{cases}
\varprojlim_n \bigoplus_{w_n \mid v} \widehat{K_{n, w_n}^\times}& \text{if } i=1, \\
\bigoplus_{w_\infty \mid v}\Z_p  & \text{if } i=2, \\
0 & \text{otherwise} 
\end{cases}
\end{equation}
(cf.~\cite[equation (19)]{BuFla06}), where the projective limit is taken with respect to the norm maps and 
$w_n$ (resp. $w_\infty$) runs over the primes of $K_n$ (resp. $K_\infty$) lying above $v$. 
When $v \nmid p$, this shows $H^1 (\Delta_{K_\infty, v} (1)) \simeq \bigoplus_{w_\infty \mid v} \Z_p(1)$.
When $v \mid p$, by taking the projective limit of \eqref{eq:H1_Kum}, we have an exact sequence 
\begin{equation}\label{Kummer seq of infty}
0 \to U_{K_\infty, v} \to H^1 (\Delta_{K_\infty, v} (1)) \to \bigoplus_{w_\infty \mid  v}\Z_p \to 0,
\end{equation}
where we put $U_{K_\infty, v} := \varprojlim_{n} U^1_{K_{n, v}}$.
For general $j \in \Z$, since $K_{\infty}$ contains $\mu_{p^{\infty}}$, we obtain a similar description of $H^i(\Delta_{K_\infty, v}(j))$ by simply twisting this formula.
See \S \ref{App:twist} for the details on the twists.

We define
\[
X_{K_\infty} (j) := \varprojlim_n X_{K_n} (j)
\]
with respect to the natural maps, and 
\[
\Xi_{K_{\infty} / k, S}^{\loc} (j) 
:= \parenth{\bigotimes_{v \in S_f} \Det_{\Lambda}^{-1} (\Delta_{K_\infty, v} (j))}
\otimes_{\Lambda}
\Det_{\Lambda}^{-1} (X_{K_{\infty}} (j)).
\]
Note that we have a natural isomorphism $\Xi_{K_{\infty} / k, S}^{\loc} (j) \simeq \varprojlim_n \Xi_{K_n / k, S}^{\loc} (j)$.

\section{Local bases at $p$-adic primes}\label{sec:Coleman}

Fix an odd prime number $p$. 
Let $K/k$ be a finite abelian extension of number fields such that $K /\Q$ is unramified at $p$.
We put $K_{\infty} = K(\mu_{p^{\infty}})$. For a $p$-adic place $v$ of $k$, we consider the semi-local Iwasawa cohomology complex $\Delta_{K_\infty, v} (j)$ over
 $\Lambda := \Z_p[[\Gal (K_\infty / k)]]$ defined in \S \ref{ss:Iw_limit}. 

The goal of this section is to construct an isomorphism
\[
\Phi^j_{K_\infty, v} : \Det_{\Lambda}^{-1} (\Delta_{K_\infty, v} (j)) \xrightarrow{\sim} \Det_{\Lambda} (R_{K_v}), 
\]
where $R_{K_v}$ is a certain explicit $\Lambda$-module.
The construction depends on Coleman's work on the so-called Coleman maps, which requires the unramified assumption at $p$.
We will establish an interpolation properties of $\Phi^j_{K_\infty, v}$ by using Perrin-Riou's interpretation of Coleman's work.

\subsection{The statement of the main result}\label{State_2}

We introduce some notation. For each $n \in \Z_{\geq 0}$, let $K_n = K (\mu_{p^n})$, $k_n = k(\mu_{p^n})$, $\GG_{K_n}:= \Gal (K_n / k)$, and we further put 
$\GG_{K_{\infty}}:= \Gal (K_{\infty} / k)$.
We set $\OO_{K_v} = \prod_{w \mid v} \OO_{K_w} $, where $w$ runs over the primes of $K$ lying above $v$ and 
$\OO_{K_w}$ is the valuation ring of $K_w$. Note that, since $v$ is unramified in $K/k$, we have $\OO_{K_v} = \OO_{k_v} \otimes_{\OO_k} \OO_K$ and this is free of rank one over $\OO_{k_v}[\Gal(K/k)]$.

We consider the formal power series ring $\OO_{K_v} [[T]]$ over the ring $\OO_{K_v}$, which is equipped with the action of the Galois group $\GG_{K_\infty} $ defined as follows.
Let $\chi_{\cyc} : \GG_{K_\infty} \to \Z_p^\times$ be the $p$-adic cyclotomic character.
For $\sigma \in \GG_{K_\infty} $ and $f(T)=\sum_{i \geq 0} a_i T^i \in \OO_{K_v} [[T]]$, we define
\[
(\sigma f) (T) := f^{\bar{\sigma}} ((1+T)^{\chi_{\cyc} (\sigma)} -1),
\]
where $\bar{\sigma} \in \Gal(K/k)$ is the restriction of $\sigma$ and $f^{\bar{\sigma}}(T):=\sum_{i \geq 0} \bar{\sigma}(a_i) T^i$. This $\GG_{K_{\infty}}$-action defines a $\Lambda$-module structure on $\OO_{K_v} [[T]]$. 

The ring $\OO_{K_v}[[T]]$ also has the Frobenius operator $\varphi$.
For each prime $w$ of $K$ lying above $v$, noting that $K_w/\Q_p$ is unramified,
we write $\sigma_{K_w / \Q_p} \in \Gal (K_w / \Q_p)$ for the arithmetic Frobenius. 
This $\sigma_{K_w/\Q_p}$ defines a $\Z_p$-linear operator on $\OO_{K_w}$.
We then define a $\Z_p$-linear operator $\sigma_{K / \Q}$ on $\OO_{K_v}$ by taking the direct sum of $\sigma_{K_w/\Q_p}$ for $w \mid v$.
Then the Frobenius operator $\varphi$ on $f(T)=\sum_{i \geq 0} a_i T^i \in \OO_{K_v} [[T]]$ is defined by
\[
(\varphi f) (T) := f^{\sigma_{K / \Q}} ((1+T)^p-1),
\]
where $f^{\sigma_{K / \Q}} (T) := \sum_{i \geq 0} \sigma_{K/\Q} (a_i) T^i$. 

We put ${r_{k_v}}:= [k_v:\Q_p]$. We consider a $\Lambda$-submodule
\[
R_{K_v} := \left\{ f(T) \in \OO_{K_v} [[T]] \, \middle| \, \sum_{\zeta \in \mu_p} f (\zeta (1+T)-1) =0 \right\} \subset \OO_{K_v} [[T]], 
\]
which is known to be a free $\Lambda$-module of rank $r_{k_v}$ (see \cite[Lemme 1.5]{Per90}). 
We define a derivative operator $D := (1+T)\dfrac{d}{dT}$ on $R_{K_v}$, which induces an isomorphism
 $D^m : R_L \xrightarrow{\sim} R_L (m)$ as $\Lambda$-modules for any integer $m \in \Z$. 

Now we present the main results in this section. We fix a $\Z_p$-basis $\xi = (\zeta_{p^n})_n$ of $\Z_p(1) = \varprojlim_n \mu_{p^n}(K_\infty)$, on which the isomorphism $\Phi_{K_{\infty}, v}^j$ in the following theorem depends.

\begin{thm}\label{thm:Phi}
For each $j \in \Z_{\geq 1}$, there is a (unique) isomorphism
\[
\Phi^j_{K_\infty, v} : \Det_{\Lambda}^{-1} (\Delta_{K_\infty, v} (j)) \xrightarrow{\sim} \Det_{\Lambda} (R_{K_v}) 
\]
which interpolates the isomorphism $\phi_{K_n, v}^j$ in \S\ref{ss:setup} for all $n \in \Z_{\geq 0}$ as in Theorem \ref{thm:explicit computation} below.
\end{thm}

The interpolation properties of $\Phi^j_{K_\infty, v}$ give explicit descriptions of the map $\varpi_n$, for all $n \geq 0$, defined by the  commutative diagram
\begin{equation}\label{eq:Phi_com}
\xymatrix{
	\Det^{-1}_{\Lambda} (\Delta_{K_\infty, v} (j)) \ar[r]_-{\simeq}^-{\Phi_{K_\infty, v}^j} \ar[d]
	& \Det_{\Lambda} (R_{K_v}) \ar[d]^{\varpi_n}\\
	\Q_p \otimes \Det_{\Z_p[\GG_{K_n}]}^{-1} (\Delta_{K_n, v} (j)) \ar[r]_-{\phi_{K_n, v}^j}^-{\simeq}
	& \Det_{\Q_p[\GG_{K_n}]} (K_{n, v}),
}
\end{equation}
where the left vertical arrow is the natural one.
For a finite character $\chi$ of $\GG_{K_{\infty}}$, we write $n_{\chi} \geq 0$ for the minimum integer such that $\chi$ factors through $\GG_{K_{n_{\chi}}}$.

\begin{thm}\label{thm:explicit computation}
 For $j \in \Z_{\geq 1}$, $n \in \Z_{\geq 0}$, and a $\C_p$-valued character $\chi$ of $\GG_{K_n}$, 
and for $f_1 (T), \dots, f_{r_{k_v}}(T) \in R_{K_v}$,
we have
\begin{align*}
&\varpi_n(\wedge_{i=1}^{r_{k_v}} f_i (T))^\chi 
= (-1)^{r_{k_v} (j-1)} \cdot  ((j-1)!)^{r_{k_v}} 
\cdot [K_n : K_{n_{\chi}}]^{-r_{k_v}} \\
&\times
\begin{cases}
(-1)^{n_\chi (r_{k_v}-1)} \cdot p^{ r_{k_v}n_\chi (j - 1)} \cdot \sigma_{K_n/k_n, v}^{-n_\chi} \cdot 
(\wedge_{i=1}^{r_{k_v}} f_i (\zeta_{p^{n_\chi}}-1))^\chi & \text{if } n_\chi > 0, \\
\cfrac{1-N (v)^{j-1} \cdot \sigma_{K_n/k_n, v}^{-1}}{1-N(v)^{-j} \cdot \sigma_{K_n/k_n, v}} 
\cdot (\wedge_{i=1}^{r_{k_v}} f_i (0))^\chi & \text{if } n_\chi = 0 \text{ and } (j \neq 1 \text{ or } \ker \chi \not\supset D_{K_n, v}), \\
\cfrac{1}{(1-N(v)^{-1}) [K_w:k_v]}
\cdot (\wedge_{i=1}^{r_{k_v}} f_i (0))^\chi & \text{if } j = 1 \text{ and } \ker \chi \supset D_{K_n, v}, 
\end{cases}
\end{align*}
where $D_{K_n, v} \subset \GG_{K_n}$ is the decomposition group of $v$ and 
$\sigma_{K_n / k_n, v} \in \Gal (K_n /k_n)$ is the arithmetic Frobenius of $v$. 
\end{thm}

Although we assume $j \geq 1$ in Theorems \ref{thm:Phi} and \ref{thm:explicit computation}, we can construct an isomorphism $\Phi_{K_{\infty}, v}^j$ for any integer $j \in \Z$ (see \S \ref{ss:twist_basis} and \S \ref{ss:twist}). 

We give a corollary of the theorems before proving them.  
As in \S \ref{local eTNC}, we set
\[
\epsilon^j_v (K_n/k) := 
1 - e_{I_{K_n, v}} N (v)^{-j} \sigma_{K_n/k, v}
\in \Q[\GG_n]
\]
and
\[
\delta_v (K_n/k):=
\epsilon_v^0 (K_n/k) + e_{D_{K_n, v}} \cdot \cfrac{1}{\# (D_{K_n, v} / I_{K_n,v}) }
\in \Q[\GG_n],
\]
where $I_{K_n, v} \subset \GG_{K_n}$ is the inertia subgroup and $\sigma_{K_n / k, v} \in \GG_{K_n} / I_{K_n, v}$ is the arithmetic Frobenius of $v$. 
For each $n \geq 0$, using the fixed primitive $p^n$-th roots of unity $\xi = (\zeta_{p^n})_n$, we define $a_n \in K_n \subset K_{n, v}$ by
\[
a_n := 
\begin{cases}
\frac{1}{p^{n-1}} (\zeta_p + \zeta_{p^2} + \cdots + \zeta_{p^n}) 
& \text{if } n \geq 1, \\
-1 & \text{if } n=0.
\end{cases}
\]
This is a trace compatible system that has been used by, e.g., Benois--Do \cite{BD}.

\begin{cor}\label{cor:main thm1}
For $j \in \Z_{\geq 1}$, $n \in \Z_{\geq 0}$ and a $\C_p$-valued character $\chi$ of $\GG_{L_n}$, 
and for $x_1, \dots, x_{r_{k_v}} \in \OO_{K_v}$, we have
\begin{align*}
& \varpi_n (\wedge_{i=1}^{r_{k_v}} x_i (1+T))^\chi
= (-1)^{r_{k_v} (j-1)} \cdot  ((j-1)!)^{r_{k_v}}\\
& \quad \times
\begin{cases}
(-1)^{n_\chi (r_{k_v}-1)} \cdot N (v)^{j n_\chi -1} \cdot \sigma_{K_n/k_n, v}^{-n_\chi} \cdot 
(\wedge_{i=1}^{r_{k_v}} a_n x_i)^\chi 
	& \text{if } n_\chi > 0, \\
(-1)^{r_{k_v}} 
\cdot \left( \cfrac{\epsilon_v^{1-j}(K_n/k)^{\#}}{\epsilon_v^j(K_n/k)} \right)^{\chi}
\cdot (\wedge_{i=1}^{r_{k_v}} a_n x_i)^\chi 
	& \text{if } n_\chi = 0 \text{ and } j \neq 1, \\
(-1)^{r_{k_v}}
\cdot \left( \cfrac{\delta_v(K_n/k)^{\#}}{\epsilon_v^1(K_n/k)} \right)^{\chi}
\cdot (\wedge_{i=1}^{r_{k_v}} a_n x_i)^\chi 
	& \text{if } n_\chi = 0 \text{ and } j = 1.
\end{cases}&
\end{align*}
\end{cor}
\begin{proof}
We can easily compute that 
\[
(a_n x)^\chi = 
\begin{cases}
\frac{1}{p^{n-1}}(\zeta_{p^{n_\chi}} x )^\chi & \text{if } n_\chi > 0, \\
-[K_n : K]^{-1} x^\chi & \text{if } n_\chi = 0, 
\end{cases}
\]
for any $x \in \OO_{K_v}$. 
Then this corollary immediately follows from Theorem \ref{thm:explicit computation}. 
\end{proof}

\subsection{Construction of the isomorphism}\label{const_iso}

In this section, we construct the isomorphism $\Phi^j_{K_\infty, v}$ in Theorem \ref{thm:Phi}. 
We begin with a review of the classical theory of the Coleman map for the Galois representation $\Q_p(1)$.

\begin{thm}[Coleman \cite{Col79}]
For each $u = (u_{K_n})_n \in U_{K_\infty, v} := \varprojlim_n \bigoplus_{w_n \mid v} U^1_{K_n, w_n}$, 
there is a unique power series $g_u(T) \in \OO_{K_v} [[T]]^{\times}$ such that for any $n \geq 0$, we have
\[
(g_u)^{\sigma_{K/\Q}^{-n}} (\zeta_{p^n} -1) = u_{K_n}.
\]
\end{thm}

We call $g_u$ the Coleman power series of $u \in U_{K_\infty, v}$. 
Coleman \cite{Col79} showed that we have $(1 - \frac{\varphi}{p}) \log (g_u) \in R_{K_v}$ for any $u \in U_{K_\infty , v}$, so we can construct the Coleman map $\Colman_{K_\infty , v}$ as
\[
\Colman_{K_\infty , v} : U_{K_\infty , v} \to R_{K_v}\; ;\; u \mapsto \left(1 - \frac{\varphi}{p}\right) \log (g_u), 
\]
which is a $\Lambda$-homomorphism. The following result is essentially proved by Coleman. 

\begin{thm}[{\cite[equation (21) and Lemma 5.1]{BuFla06}, \cite[Theorems 3 and 4]{Col83}}]\label{thm:Coleman map}
We have an exact sequence of $\Lambda$-modules 
\begin{equation}
0 \to \bigoplus_{w_\infty \mid v} \Z_p (1) \to U_{K_\infty , v} \overset {\Colman_{K_\infty, v}}{\to} R_{K_v} \to 
 \bigoplus_{w_\infty \mid v} \Z_p (1) \to 0,
\end{equation}
where $w_\infty$ runs over the prime of $K_\infty$ lying above $v$, 
$\bigoplus_{w_\infty \mid v}  \Z_p (1) \to U_{K_\infty , v}$ is the natural inclusion and
 $R_{K_v} \to \bigoplus_{w_\infty \mid v} \Z_p (1)$ is defined by $f (T) \mapsto \Tr_{K_v /\Q_p}( (D f) (0) ) \xi_v$. 
Here, we put $\Tr_{K_v / \Q_p}:= \oplus_{w \mid v} \Tr_{K_w / \Q_p} : K_v \to \bigoplus_{w \mid v} \Q_p$ and $\xi_v := (\xi_{w_\infty})_{w_\infty \mid v}$ with $\xi_{w_\infty}$ the image of $\xi$ by the 
natural inclusion $K_\infty \hookrightarrow K_{\infty, w_\infty}$. 
\end{thm}

For any $j \in \Z$, 
we define the $(j-1)$-th twisted Coleman map $\Colman_{K_\infty ,v}^j$ as the 
composite map
\[
\Colman_{K_\infty , v}^j : U_{K_\infty , v} (j-1) \to R_{K_v} (j-1) \overset{D^{1-j}}{\xrightarrow{\sim}} R_{K_v}, 
\]
where the first map is just the $(j-1)$-th twist of $\Colman_{K_\infty,v}$ and the second map is defined by the isomorphism
\[
R_{K_v}(j-1) := R_{K_v} \otimes_{\Z_p} \Z_p (j-1)  \xrightarrow{\sim} R_{K_v} ; f(T) \otimes \xi^{\otimes j-1} 
 \mapsto D^{1-j} f(T).
 \]
We clearly have $\Colman_{K_\infty, v}^1 = \Colman_{K_\infty, v}$. 
Moreover, Theorem \ref{thm:Coleman map} induces an exact sequence of $\Lambda$-modules 
\begin{equation}\label{eq:Col_ex}
0 \to \bigoplus_{w_\infty \mid v} \Z_p (j) \to U_{K_\infty , v} (j-1) \overset {\Colman_{K_\infty , v}^j}{\to} R_{K_v} \to
\bigoplus_{w_\infty \mid v} \Z_p (j) \to 0,
\end{equation}
where $\bigoplus_{w_\infty \mid v} \Z_p (j) \to U_{K_\infty, v} (j-1)$ is the $(j-1)$-th twist of the natural inclusion and
$R_{K_v} \to \bigoplus_{w_\infty \mid v} \Z_p (j) $ is defined by $f (T) \mapsto \Tr_{K_v /\Q_p}( (D^j f) (0) ) \xi^{\otimes j}_v$, 
where $\xi^{\otimes j}_v := (\xi_{w_\infty}^{\otimes j})_{w_\infty \mid v}$.  

\begin{defn}
For each $j \in \Z$, we now construct the desired isomorphism
\[
\Phi^j_{K_\infty, v} : \Det_{\Lambda}^{-1} (\Delta_{K_\infty , v} (j)) \xrightarrow{\sim} \Det_{\Lambda} (R_{K_v}). 
\]
To this end, let $Q$ be the total ring of fractions of $\Lambda$. 
Since $Q \otimes_{\Lambda} \Z_p (m) = 0$ for any $m \in \Z$, 
we have an isomorphism
\begin{align*}
\wtil{\Phi_{K_\infty , v}^j}: 
Q \otimes_{\Lambda} \Det_{\Lambda}^{-1} (\Delta_{K_\infty , v} (j)) 
& \simeq \Det_{Q}^{-1} (Q \otimesL_{\Lambda} \Delta_{K_\infty , v} (j))\\
&\overset{\eqref{Iwasawa cohomology}}{\xrightarrow{\sim}}
\Det_{Q} (Q \otimes_{\Lambda} H^1 (\Delta_{K_\infty ,v} (j)) ) \\
&\overset{\eqref{Kummer seq of infty}}{\xrightarrow{\sim}}
\Det_{Q} (Q \otimes_{\Lambda} U_{K_\infty, v} (j-1)) \\
&\overset{\eqref{eq:Col_ex}}{\xrightarrow{\sim}}
\Det_{Q} (Q \otimes_{\Lambda} R_{K_v} )
 \simeq
Q \otimes_{\Lambda} \Det_{\Lambda} (R_{K_v}).
\end{align*}
By the argument in \cite[Proposition 5.2]{BuFla06}, where the base field is $\Q$, the restriction of $\wtil{\Phi_{K_\infty , v}^j}$ gives us the desired isomorphism $\Phi^j_{K_\infty , v}$.
\end{defn}

\subsection{The work of Perrin-Riou}\label{diagram}

A key to prove Theorem \ref{thm:explicit computation} is an interpretation of the Coleman map \`{a} la Perrin-Riou \cite{PR}. We first review the specialization map $\Xi_{n, j} : R_{K_v} \to K_{n, v}$. For any $n \geq m \geq 0$, we write 
\[
\Tr_{K_{n, v} / K_{m, v}} : K_{n, v} := 
\bigoplus_{w_n \mid v} K_{n, w_n} \to  \bigoplus_{w_m \mid v} K_{m, w_m} =: K_{m, v}
\]
for the map which is induced by the trace maps. 
\begin{defn}
For any $j \in \Z_{\geq 1}$ and $n \in \Z_{\geq 0}$, 
we define a homomorphism $\Xi_{n, j}: R_{K_v} \to K_{n, v}$ by 
\begin{equation}\label{def:Xi_map}
\displaystyle
\Xi_{n,j} (f(T)) := 
\begin{cases}\displaystyle
p^{n(j-1)} \parenth{\sum_{k=0}^{n-1} \cfrac{f^{\sigma_{K / \Q}^{k-n}}(\zeta_{p^{n-k}} -1)}{p^{jk}} 
- \cfrac{(1-p^j \sigma_{K / \Q}^{-1})^{-1}f^{\sigma_{K / \Q}^{-1}} (0)}{p^{j(n-1)}}} & \text{ if } n \geq1, \\
\Tr_{K_{1, v} / K_v} \parenth{\Xi_{1,j} (f (T))} & \text{ if } n =0, 
\end{cases}
\end{equation}
for any $f(T) \in R_{K_v}$. 
Here, the operator $\sigma_{K / \Q}$ is defined in \S \ref{State_2}. 
By \cite[Lemme 2.2.2]{BD}, the map $\Xi_{n,j}$ coincides with
the map $\Xi_{n, \Q_p (j)}$ in \cite{PR} in our semi-local setting. 
\end{defn}

The following claim, due to Perrin-Riou \cite{PR}, states that $\Colman_{K_\infty , v}^j$ for $j\in \Z_{\geq 1}$ interpolates the Bloch--Kato logarithm $\log_{\Q_p(j)} : H^1_f (K_{n, v}, \Q_p (j)) \to K_{n,v}$ for all $n \geq 0$.

\begin{prop}\label{key diagram} 
For any $n \in \Z_{\geq 0}$ and $j \in \Z_{\geq 1}$, 
we have a commutative diagram
\[
\xymatrix@C=40pt{
U_{K_\infty , v} (j-1)  \ar[r]^-{\Colman_{K_\infty , v}^j} \ar[d]
&R_{K_v} \ar[d]^-{(-1)^{j-1} \cdot (j-1)! \cdot \Xi_{n,j}} \\
H^1_f (K_{n, v}, \Q_p (j)) \ar[r]^-{\log_{\Q_p(j)}}_-{\simeq}
&K_{n, v},
}\]
where the left vertical map is induced by the natural composite map
\[
U_{K_\infty, v} (j-1) \overset{\eqref{Kummer seq of infty}}{\hookrightarrow} H^1 (\Delta_{K_\infty, v} (j)) 
\to H^1 (\Delta_{K_n, v} (j)) \hookrightarrow    
\bigoplus_{w_n \mid v} H^1(K_{n, w_n}, \Q_p (j)). 
\]
\end{prop}

\begin{proof}
The map $(-1)^{j-1} \Colman_{K_\infty, v}^j$ coincides with the inverse of Perrin-Riou's isomorphism $\Omega_{\Q_p (j), j}^\xi$ in \cite{PR}. 
Here, the sign $(-1)^{j-1}$ appears because we defined $\Colman^j_{K_{\infty}, v}$ simply by twisting, while in \cite{PR} the twisting property of $\Omega_{\Q_p (j), j}^\xi$ involves $(-1)^{j-1}$.
The proposition directly follows from the properties of $\Omega_{\Q_p (j), j}^\xi$ stated in \cite[Th\'eor\`eme (i) in p. 82]{PR}. 
\end{proof}

The $\chi$-component of the image of $\Xi_{n, j}$ can be described as the following.

\begin{lem}\label{lem:Xi}
For any $j \in \Z_{\geq 1}$, $n \in \Z_{\geq 0}, \chi \in \widehat{\GG}_{K_n}$ and $f (T) \in R_{K_v}$, 
we have 
\[
\Xi_{n,j} (f(T))^\chi = 
\begin{cases}
p^{jn_\chi -n} \cdot \sigma_{K_n / \Q_n}^{-n_\chi} \cdot f(\zeta_{p^{n_\chi}} -1)^\chi & \text{ if } n_\chi > 0, \\
\
[K_n : K]^{-1} \cdot \cfrac{1-p^{j-1}\sigma_{K_n / \Q_n}^{-1}}{1-p^{-j} \sigma_{K_n / \Q_n}} \cdot f(0)^\chi & \text{ if } n_\chi = 0. 
\end{cases}
\]
Here, $\sigma_{K_n / \Q_n}$ is the operator on $K_{n,v}$ which is defined by 
$\sigma_{K_n / \Q_n} \parenth{ (a_{w_n})_{w_n \mid v}} = 
\parenth{\sigma_{K_{n}, w_n} (a_{w_n}) }_{w_n \mid v}$, 
where $\sigma_{K_{n}, w_n} \in \Gal (K_{n, w_n} / \Q_p (\mu_{p^n}))$ is the arithmetic Frobenius. 

\end{lem}

\begin{proof}
This lemma should be known to experts, but we include a proof due to the lack of references.

To begin with, for $0 \leq m \leq n$ such that $m \neq n_\chi$, 
we show
\begin{equation}\label{claim}
f(\zeta_{p^m}-1)^\chi = 
\begin{cases}
-\cfrac{1}{p-1} f(0)^\chi & \text{ if } m=1 \text{ and } n_\chi =0, \\
0 & \text{otherwise},   
\end{cases}
\end{equation}
for any $f (T) \in R_{K_v}$.

If $0 \leq m < n_\chi$, we have
\[
e_\chi f(\zeta_{p^m}-1) = \frac{1}{\# \GG_{K_{n_\chi}}}\sum_{\sigma \in \GG_{K_{n_\chi}}}  \chi (\sigma)f(\zeta_{p^m}-1)^{\sigma^{-1}}
=
\frac{1}{\# \GG_{K_{n_\chi}}}\sum_{\bar{\sigma} \in \GG_{K_m}} f(\zeta_{p^m}-1)^{\bar{\sigma}^{-1}}
\sum_{\overset{\sigma \in \GG_{K_{n_\chi}}}{\sigma|_{K_{m}} = \bar{\sigma} }}  \chi (\sigma). 
\]
Since $m < n_\chi$, the last sum is equal to $0$ for any $\bar{\sigma} \in \GG_{K_m}$. 
Therefore, $f(\zeta_{p^m}-1)^\chi = 0$ follows.

If $n_\chi < m$, we have
\[
e_\chi f(\zeta_{p^m}-1) = \frac{1}{\# \GG_{K_n}}\sum_{\sigma \in \GG_{K_{n_\chi}}}
\chi (\sigma) \sigma^{-1} \cdot \Tr_{K_{n, v} / K_{n_\chi , v}} (f(\zeta_{p^m}-1)). 
\]
Since $f (T) \in R_{K_v}$, we have 
\[
0 = \sum_{\zeta \in \mu_p} f (\zeta (1+T)-1)|_{T=\zeta_{p^m}-1} = 
\begin{cases}
\Tr_{K_{m, v} / {K_{m-1, v}}} (f(\zeta_{p^m}-1)) & \text{ if } m \geq 2, \\
\Tr_{K_{1, v} / {K_v}} (f(\zeta_{p}-1)) + f (0) & \text{ if } m = 1.
\end{cases}
\]
Taking $n_\chi < m \leq n$ into account, we obtain
\[
e_\chi f(\zeta_{p^m}-1) =
\begin{cases}
0 & \text{ if } m \geq 2, \\
 -\cfrac{1}{p-1} e_\chi f (0) & \text{ if } m = 1 \; (\text{ and } n_\chi =0).
\end{cases}
\]
This establishes claim \eqref{claim}. 

Now we are ready to compute $\Xi_{n,j} (f(T))^\chi $. 
If $n_\chi > 0$, \eqref{claim} implies that  
the $\chi$-component of \eqref{def:Xi_map} is equal to
\[
\Xi_{n,j} (f(T))^\chi 
= \parenth{p^{n(j-1)} \cfrac{f^{\sigma_{K / \Q}^{-n_\chi}}(\zeta_{p^{n_\chi} } -1)}{p^{j(n-n_\chi)}}}^\chi 
= p^{jn_\chi - n} \cdot \sigma_{K_n / \Q_n}^{-n_\chi} \cdot f(\zeta_{p^{n_\chi}}-1)^\chi.  
\]
Also, if $n > n_\chi = 0$, \eqref{claim} implies that 
\begin{align*}
\Xi_{n,j} (f(T))^\chi 
&= p^{n(j-1)} \parenth{\cfrac{f^{\sigma_{K / \Q}^{-1}} (\zeta_p -1)}{p^{j(n-1)}} - 
\cfrac{(1-p^j \sigma_{K / \Q}^{-1})^{-1} f^{\sigma_{K / \Q}^{-1}} (0)}{p^{j(n-1)}}}^\chi \\
&=\cfrac{p^j}{p^n}
\parenth{\left(-\cfrac{\sigma_{K / \Q}^{-1}}{p-1}-\cfrac{\sigma_{K / \Q}^{-1}}{1-p^j \sigma_{K / \Q}^{-1}}\right)f (0) }^\chi = [K_n : K]^{-1} \cdot \cfrac{1-p^{j-1}\sigma_{K_n / \Q_n}^{-1}}{1-p^{-j} \sigma_{K_n / \Q_n}} \cdot f(0)^\chi.
\end{align*}
If $n=n_\chi=0$, we have 
\begin{align*}
\Xi_{0,j} (f(T))^\chi 
&= \Tr_{K_{1, v} / K_v} (\Xi_{1,j} (f(T)))^\chi \\
&= p^{j-1} \Tr_{K_{1, v} / K_v} \parenth{f^{\sigma_{K / \Q}^{-1}} (\zeta_p -1) - 
(1-p^j \sigma_{K / \Q}^{-1})^{-1} f^{\sigma_{K / \Q}^{-1}} (0)}^\chi \\
&=p^{j-1} \Tr_{K_{1, v} / K_v} 
\parenth{\left(-\cfrac{\sigma_{K / \Q}^{-1}}{p-1}-\cfrac{\sigma_{K /\Q}^{-1}}{1-p^j \sigma_{K / \Q}^{-1}}\right)f (0) }^\chi \\
&= p^{j-1} (p-1) \parenth{\left(-\cfrac{\sigma_{K / \Q}^{-1}}{p-1}-\cfrac{\sigma_{K / \Q}^{-1}}{1-p^j \sigma_{K / \Q}^{-1}}\right)f (0)}^\chi 
= \cfrac{1-p^{j-1}\sigma_{K / \Q}^{-1}}{1-p^{-j} \sigma_{K / \Q}} \cdot f(0)^\chi.
\end{align*}
This completes the proof of Lemma \ref{lem:Xi}. 
\end{proof}

\subsection{The proof of Theorem \ref{thm:explicit computation}}\label{ss:interp_p}

We fix $n \in \Z_{\geq0}$, $j \geq 1$, and a character $\chi \in \widehat{\GG_{K_n}}$.
We define a prime ideal  
\[
\fq_\chi := \ker (\Lambda \rightarrow 
\Z_p [\GG_{K_n}] \xrightarrow{\chi} \Z_p[\chi])
\]
of $\Lambda$, where the first map is the natural restriction map. 
We write $\Lambda_{\fq_\chi}$ for the localization of $\Lambda$ at $\fq_\chi$. 
Then $\Lambda_{\fq_\chi}$ is a discrete valuation ring and 
its residue field is $\Q_p (\chi)$.
For any $\Lambda$-module $M$, 
we write $M_{\fq_\chi} := \Lambda_{\fq_\chi} \otimes_{\Lambda} M$. 
Similarly, we put 
\[
\Delta_{K_\infty , v} (j)_{\fq_\chi}
:= \Lambda_{\fq_\chi} \otimesL_{\Lambda} \Delta_{K_\infty , v} (j), \;\;\;
\Delta_{K_n , v} (j)^\chi := \Q_p (\chi) \otimesL_{\Z_p[\GG_{K_n}]} \Delta_{K_n , v} (j) .  
\]
We have a canonical isomorphism
\begin{equation}\label{complex descent chi}
\Lambda_{\fq_\chi} / \fq_\chi \otimesL_{\Lambda_{\fq_\chi}}
\Delta_{K_\infty, v} (j)_{\fq_\chi} \simeq 
\Delta_{K_n, v} (j)^\chi. 
\end{equation}

Since $\parenth{ \bigoplus_{w_\infty \mid v} \Z_p (j)}_{\fq_\chi} = 0 $, Theorem \ref{thm:Coleman map} implies that 
the localization of $\Colman_{K_\infty , v}^j$ at $\fq_{\chi}$ is isomorphism, and Proposition \ref{key diagram} yields the following commutative diagram :
\begin{equation}\label{eq:Col_com}
\xymatrix{
 U_{K_\infty , v} (j-1)_{\fq_\chi}
  \ar[r]_-{\simeq}^-{\Colman_{K_\infty , v}^j} \ar[d]
&(R_{K_v})_{\fq_{\chi}} \ar[d]^-{ (-1)^{j-1} \cdot (j-1)! \cdot \Xi_{n,j}} \\
H^1_f (K_{n, v}, \Q_p (j))^{\chi} \ar[r]^-{\log_{\Q_p(j)}}_-{\simeq}
&(K_{n, v})^{\chi}. 
}
\end{equation}

\subsubsection{The case $j \geq 2$ or  $\ker \chi \not\supset D_{K_n , v}$}

First suppose $j \geq 2$ or $\ker \chi \not\supset D_{K_n , v}$. 
In this case, by simply taking the determinant modules of \eqref{eq:Col_com}, we obtain the $\chi$-component of the desired diagram \eqref{eq:Phi_com}.
Therefore, 
the $\chi$-component of the map $\varpi_n$ to be computed is the same as the map induced by $(-1)^{j-1} \cdot (j-1)! \cdot \Xi_{n,j}$.

By Lemma \ref{lem:Xi}, for $f_1(T), \dots, f_{r_{k_v}}(T) \in R_{K_v}$, we have
\[
\wedge \Xi_{n, j}(\wedge_{i=1}^{r_{k_v}} f_i(T))^{\chi}
= 
\begin{cases}
p^{r_{k_v} (j n_\chi -n)} \cdot \det_{\Q_p[\GG_{K_n}]} (\sigma_{K_n / \Q_n} \mid K_{n, v})^{-n_\chi} 
\cdot \wedge_{i=1}^{r_{k_v}} f_i (\zeta_{p^{n_\chi}} -1)^\chi & \text{ if } n_\chi > 0, \\
[K_n : K]^{-r_{k_v}} \cdot 
\cfrac{\det_{\Q_p[\GG_{K_n}]} (1- p^{j-1} \sigma_{K_n / \Q_n}^{-1} \mid K_{n, v})}{\det_{\Q_p[\GG_{K_n}]} (1- p^{-j} \sigma_{K_n / \Q_n} \mid K_{n, v})} \cdot \wedge_{i=1}^{r_{k_v}} f_i (0)^\chi & \text{ if } n_\chi = 0. 
\end{cases}
\]
Here, for a $\Q_p[\GG_{K_n}]$-endomorphism $f$ on $K_{n,v}$, we write $\det_{\Q_p[\GG_{K_n}]} (f \mid K_{n, v})$ for the determinant of $f$. Therefore, to show Theorem \ref{thm:explicit computation} in this case, we only have to show the following.

\begin{lem}\label{lem:det1}
The following are true. 
\begin{itemize}
\item[(i)] We have 
$\det_{\Q_p[\GG_{K_n}]} (\sigma_{K_n / \Q_n}  \mid K_{n, v}) = (-1)^{r_{k_v} -1}\sigma_{K_n / k_n , v}$. 
\item[(ii)] For any $a \in \Q_p$, we have 
$\det_{\Q_p[\GG_{K_n}]} (1- a \sigma_{K_n / \Q_n}  \mid K_{n, v}) = 1 - a^{r_{k_v}} \sigma_{K_n / k_n , v} $.   
\end{itemize}
These formulas also hold when we replace 
$\sigma_{K_n / \Q_n} $ and $\sigma_{K_n / k_n, v}$ by
$\sigma_{K_n / \Q_n}^{-1} $ and $\sigma_{K_n / k_n, v}^{-1}$ respectively. 
\end{lem}

\begin{proof}
By the normal basis theorem, $K_n$ is free of rank one over $k[\GG_{K_n}]$  
and $k_v$ is free of rank one over $\Q_p[\Gal (k_v / \Q_p)]$. 
We take generators $x$ and $y$ of $K_n$ and $k_v$ as a $k[\GG_{K_n}]$-module and 
 $\Q_p[\Gal (k_v / \Q_p)]$-module respectively. 
 Then, $\{ \sigma_{K_n / \Q_n}^i (x  \otimes y) \mid  1 \leq i \leq r_{k_v}\}$ is a basis of 
 $K_{n, v} \simeq K_n \otimes_k k_v$ as a $\Q_p[\GG_{K_n}]$-module. 
Therefore, the action of $\sigma_{K_n / \Q_n}$ on $K_{n, v}$ over $\Q_p[\GG_{K_n}]$ can be presented by a matrix
\[
\begin{pmatrix}
0 & 0 &  \dots & 0 & \sigma_{K_n / k_n, v} \\
1 & 0 &   \dots & 0 & 0 \\
0 & 1 &   \dots & 0 & 0 \\
\vdots & \vdots & \ddots & \vdots & \vdots \\
 0  & 0 & \dots & 1 & 0 
\end{pmatrix}.
\]
By elementary computation of the determinants, we obtain the formulas (i) and (ii). 
The final claim also follows in the same way.
\end{proof}

\subsubsection{The case $j = 1$ and  $\ker \chi  \supset D_{K_n,v}$}\label{chi = 1 and j=1}

Suppose $j = 1$ and $\ker \chi  \supset D_{K_n,v}$.
Since both $H^2 (\Delta_{K_{\infty , v}} (1))_{\fq_{\chi}}$ and $H^2 (K_{n, v}, \Q_p (1))^{\chi}$ are isomorphic to 
$\Q_p (\chi)$ and in particular non-zero, we need some technical descent arguments in this case. 

We write $\fq := \fq_{\chi}$ for simplicity. 
We fix a topological generator $\gamma$ of $\Gal(K_\infty / K_1 )$. 
Then $\omega := 1 - \gamma$ is a uniformizer of $\Lambda_{\fq}$. 
For any $\Lambda_{\fq}$-module $M$, 
we put $\overline{M}:= M / \omega M$. 

By Proposition \ref{prop:Det02}, we see that the following diagram is commutative :  
\[
\xymatrix{
	\Det_{\Lambda_{\fq}}^{-1}(\Delta_{K_{\infty, v}}(1)_{\fq}) \ar[r]^-{\simeq} \ar@{->>}[dd]
	& \Det_{\Lambda_{\fq}} (U_{K_\infty, v, \fq}) 
	\ar@{->>}[d]\\
	& \Det_{\Q_p (\chi)} (\ol{U_{K_\infty, v, \fq}})
	\ar[d]_{\simeq}^{\ND{-}\alpha}\\
	\Det_{\Q_p (\chi)}^{-1} (\Delta_{K_n, v}(1)^{\chi}) \ar[r]_-{\simeq}
	& \Det_{\Q_p (\chi)} (H^1_f(K_{n, v}, \Q_p(1))^{\chi} ), 
}
\]
where the map $\alpha$ is induced by the exact sequence
\[
0 
\to (\oplus_{w \mid v} \Q_p)^\chi 
\overset{f_1}{\to} \ol{U_{K_\infty, v, \fq}} 
\overset{g}{\to} H^1_f (K_{n, v}, \Q_p(1))^{\chi} 
\overset{f_2}{\to} (\oplus_{w \mid v} \Q_p)^\chi 
\to 0
\]
as in the beginning of \S \ref{App:det_prop}, where the map $g$ is the natural one. 
We need explicit descriptions of the maps $f_1$ and $f_2$. Let
\[
\eta_{K_\infty} := ( (\zeta_{p^m} -1)_{w_m \mid v})_m \in (\varprojlim_{m} 
\bigoplus_{w_m \mid v} \widehat{K_{m, w_m}^{\times}})_{\fq}
\overset{\eqref{Iwasawa cohomology}}{\simeq} H^1 (\Delta_{K_\infty, v} (1))_{\fq}. 
\]
The valuation of $\eta_{K_{\infty}}$ is one.
Therefore, the map $f_1$ is characterized by 
\[
f_1( (1_w)_{w\mid v}) = \ol{\omega \eta_{L_{\infty}}}
\]
(implicitly we have $\omega \eta_{L_\infty} \in U_{K_\infty, v,  \fq}$).
On the other hand, $f_2$ is described as
\[
f_2(x) = \cfrac{\Tr_{K_{n, v} / \Q_p}(\log_p (x))}{\log_p (\chi_{\cyc}(\gamma))}
\] 
for $x \in H^1_f (K_{n, v}, \Q_p(1))^{\chi} =( \Q_p \otimes_{\Z_p} U_{K_n, v}  )^{\chi}$ by \cite[Lemma 5.8 and its proof]{Fl04}.
Here, $\log_p (-)$ is the $p$-adic logarithm map and 
$\Tr_{K_{n, v} / \Q_p} : K_{n,v} = \oplus_{w_n \mid v} K_{n, w_n}  \to \oplus_{w \mid v} \Q_p$ is the map 
induced by trace maps. 

The aforementioned exact sequence fits into a commutative diagram
\begin{equation}\label{last diagram 1}
\xymatrix{
	0 \ar[r] 
	& (\oplus_{w \mid v} \Q_p)^\chi \ar[r]^-{f_1} \ar@{=}[d] 
	& \ol{U_{K_\infty, v , \fq}}
	\ar[r]^-{g} \ar[d]_-{\simeq}^-{\overline{\Colman_{K_\infty, v,  \fq}}} 
	& H^1_f (K_{n, v}, \Q_p(1))^{\chi} \ar[r]^-{f_2} \ar[d]^-{\log_{\Q_p(1)}=\log_p}_-{\simeq}
	& (\oplus_{w \mid v} \Q_p)^\chi \ar[r] \ar@{=}[d] & 0 \\
	0 \ar[r] 
	& (\oplus_{w \mid v} \Q_p)^\chi \ar[r]^-{g_1} 
	& \overline{(R_{K_v})_\fq} \ar[r]^-{\overline{\Xi_{n, 1}}} 
	& K_{n, v}^{\chi} \ar[r]^-{g_2} 
	&(\oplus_{w \mid v} \Q_p)^\chi \ar[r] &0,
}
\end{equation}
where $\overline{\Colman_{K_\infty, v,  \fq}}$ and $\overline{\Xi_{n, 1}}$ are induced by 
$\Colman_{K_\infty , v}$ and $\Xi_{n, 1}$ respectively, and the maps $g_1$ and $g_2$ are defined by the commutativity of the other squares. We note that Bloch--Kato logarithm $\log_{\Q_p (1)}$ for $\Q_p(1)$ is equal to the $p$-adic logarithm $\log_p$. The commutativity of the middle square above follows from \eqref{eq:Col_com}.

Let us describe the maps $g_1$ and $g_2$ explicitly. The map $(R_{K_v})_\fq \to K_{n, v}^{\chi} ; \; f(T) \mapsto f(0)$ induces the isomorphism 
$\ol{(R_{K_v})_\fq} \xrightarrow{\sim} K_{n, v}^{\chi}$, through which we identify these.  

From the description of $f_2$, we immediately have
\[
g_2(x) = \cfrac{\Tr_{K_{n, v} / \Q_p}(x)}{\log_p (\chi_{\cyc}(\gamma))} . 
\] 
On the other hand, we have
\[
g_1((1)_{w \mid v}) = \left( -\left(1-\frac{1}{p}\right) \log_p (\chi_{\cyc} (\gamma)) \right)_{w \mid v}.
\]
In fact, the Coleman power series associated with 
$\omega \eta_{K_\infty} = 
\parenth{\parenth{\frac{\zeta_{p^m}-1}{\zeta_{p^m}^{\chi_{\cyc} (\gamma)}-1}}_{w_n \mid v}}_{K_n} \in U_{K_\infty, v}$ is $ \cfrac{T}{(1+T)^{\chi_{\cyc}(\gamma)}-1}$, and we can compute
\begin{align*}
\overline{\Colman_{K_\infty, v,  \fq}} (f_1( (1_w)_{w \mid v}) ) 
&= \left(1-\cfrac{\varphi}{p}\right) 
\log \parenth{\cfrac{T}{(1+T)^{\chi_{\cyc}(\gamma)}-1}}|_{T=0} \\
&= \left(1-\cfrac{1}{p}\right) 
\log_p (\chi_{\cyc}(\gamma)^{-1}) = 
-\left(1-\cfrac{1}{p}\right) \log_p (\chi_{\cyc}(\gamma)).
\end{align*}

The diagram \eqref{last diagram 1} induces the right two commutative squares of the diagram (the left square is the previous one)
\[
\xymatrix{
	\Det_{\Lambda_{\fq}}^{-1}(\Delta_{K_{\infty , v}}(1)_{\fq}) \ar[r]^-{\simeq} \ar@{->>}[dd]
	& \Det_{\Lambda_{\fq}} (U_{K_\infty, v , \fq}) \ar@{->>}[d] \ar[r]^-{\wedge \Colman}
	& \Det_{\Lambda_{\fq}} ((R_{K_v})_\fq) \ar@{->>}[d] \\
	& \Det_{\Q_p (\chi)} (\ol{U_{K_\infty , v , \fq}}) \ar[d]_{\simeq}^{-\alpha} \ar[r]^-{\wedge \ol{\Colman}}
	& \Det_{\Q_p (\chi)} (\ol{(R_{K_v})_\fq}) \ar[d]^{-\beta}_{\simeq} \\
	\Det_{\Q_p (\chi)}^{-1} (\Delta_{K_n , v}(1)^{\chi}) \ar[r]_-{\simeq}
	& \Det_{\Q_p (\chi)} (H^1_f(K_{n, v}, \Q_p(1))^{\chi}) \ar[r]_-{\wedge \log_{\Q_p(1)}}
	& \Det_{\Q_p (\chi)} (K_{n, v}^{\chi}),
}
\]
where $\beta$ is induced by the lower sequence of \eqref{last diagram 1}.

Thanks to this diagram, the map $\Det_{\Lambda_{\fq}} ((R_{K_v})_\fq ) \to 
\Det_{\Q_p (\chi)} (K_{n, v}^\chi)$ to be computed coincides with $-\beta$.
By the definition of $\beta$, the descriptions of $g_1$ and $g_2$, and by Lemma \ref{lem:Xi}, 
we have 
\begin{align*}
(-1) \cdot \beta (\wedge_{i=1}^{r_{k_v}} \overline{f_i (T)})) 
&= 
\parenth{(1-p^{-1})\log_p (\chi_{\cyc} (\gamma)) }^{-1}
\cdot \cfrac{\log_p (\chi_{\cyc} (\gamma))}{[K_{n, w_n}:\Q_p]} \\
&\times [K_n : K]^{-(r_{k_v} -1)} \cdot
\cfrac{ \det_{\Q_p[\GG_{K_n}]} (1-\sigma_{K_n / \Q_n}^{-1} \mid K_{n, v} / \oplus_{w \mid v} \Q_p)}
{\det_{\Q_p[\GG_{K_n}]} (1-p^{-1}\sigma_{K_n / \Q_n} \mid K_{n, v} / \oplus_{w \mid v} \Q_p)}  
\cdot (\wedge_{i=1}^{r_{k_v}}f_i (0))^{\chi} \\
= [K_n : K]^{-r_{k_v}} \cdot
& \cfrac{1}{[K_w :\Q_p] \cdot (1-p^{-1})} 
\cdot \cfrac{ \det_{\Q_p[\GG_{K_n}]} (1-\sigma_{K_n / \Q_n}^{-1} \mid K_{n, v} / \oplus_{w \mid v} \Q_p)}
{\det_{\Q_p[\GG_{K_n}]} (1-p^{-1}\sigma_{K_n / \Q_n} \mid K_{n, v} / \oplus_{w \mid v}\Q_p)}  
\cdot (\wedge_{i=1}^{r_F}f_i (0))^{\chi}, 
\end{align*}
where the first equality is obtained by applying Lemma \ref{lem:Det01} 
to the bottom exact sequence in \eqref{last diagram 1} (here, 
$K_{n, v} / \oplus_{w \mid v} \Q_p$ is the quotient as $\Q_p[\GG_{K_n}]$-modules). 
Therefore, in the case $j = 1$ and $\ker \chi \supset D_{K_n , v}$, 
Theorem \ref{thm:explicit computation} follows from this formula and Lemma \ref{lem:det2} below.

\begin{lem}\label{lem:det2}
For any $a \in \Q_p$ and any $\chi \in \widehat{\GG_{K_n}}$ such that $\ker \chi \supset D_{K_n, v}$, we have  
\[\textstyle
\det_{\Q_p[\GG_{K_n}]} (1-a \sigma_{K_n / \Q_n} \mid K_{n, v} / \oplus_{w \mid v} \Q_p)^\chi = 
\begin{cases}
\cfrac{1-a^{r_{k_v}}}{1-a} & \text{ if } a \neq 1, \\
r_{k_v} = [k_v : \Q_p] & \text{ if } a = 1. 
\end{cases}
\]
This formula also holds when we replace 
$\sigma_{K_n / \Q_n} $ by $\sigma_{K_n / \Q_n}^{-1} $. 
\end{lem}

\begin{proof}
By Lemma \ref{lem:det1}(ii), we obtain 
\[\textstyle
\det_{\Q_p[\GG_{K_n}]} (1-a \sigma_{K_n / \Q_n} \mid K_{n,v} )^\chi = 1 - a^{r_{k_v}}. 
\]
We also have $\det_{\Q_p[\GG_{K_n}]} (1-a \sigma_{K_n / \Q_n} \mid \oplus_{w \mid v} \Q_p )^\chi = 1 - a$.
These imply
\[
\textstyle \det_{\Q_p[\GG_{K_n}]} (1-a \sigma_{K_n / \Q_n} \mid K_{n,v}/ \oplus_{w \mid v}\Q_p )^\chi 
= \cfrac{1-a^{r_{k_v}}}{1-a} = 1 + a + \dots + a^{r_{k_v}-1}
\]
as long as $a \neq 1$.
Since $\det_{\Q_p[\GG_{K_n}]} (1-a \sigma_{K_n / \Q_n} \mid K_{n, v}/ \oplus_{w \mid v}\Q_p )^\chi$ is a polynomial in $a$ of degree at most $r_{k_v} - 1$, the formula for $a = 1$ is also valid. 
The final claim also follows in the same way.
\end{proof}

\section{Local bases at non-$p$-adic primes}\label{sec:non p adic prime}

We fix an odd prime number $p$ and 
let $K/k$ be an abelian extension of number fields.  
Set $K_{\infty} = K(\mu_{p^{\infty}})$.

We fix a finite prime $v$ of $k$ such that $v$ is not lying above $p$.
We consider the semi-local Iwasawa cohomology complex 
$\Delta_{K_\infty , v} (j)$ over $\Lambda := \Z_p[[\Gal (K_\infty / k)]]$.
The goal of this section is to construct an explicit basis of $\Det_{\Lambda}^{-1} (\Delta_{K_\infty, v} (j))$. 
Conceptually this should be related to the local $\varepsilon$-conjecture for $l \neq p$, 
which has been established by Yasuda \cite{Yasu09}. 

\subsection{The statement}\label{ss:basis_l}

Set $K_n = K(\mu_{p^n})$ and $\GG_{K_n} = \Gal(K_n/k)$. 
As before, we set
\[
\epsilon_v^j (K_n/k) := 
1 - e_{I_{K_n , v}} N (v)^{-j} \sigma_{K_n / k , v}
\in \Q_p[\GG_n]
\]
and 
\[
\delta_v (K_n/k) := 
\epsilon_v^0 (K_n / k) + e_{D_{K_n, v}} \# (D_{K_n, v} / I_{K_n, v})^{-1}
\in \Q_p[\GG_n]
\]
where $I_{K_n , v} , D_{K_n, v} \subset \GG_{K_n}$ are the inertia group and the decomposition group of $v$
respectively and $\sigma_{K_n /k , v} \in \GG_{K_n} / I_{K_n, v}$ is the arithmetic Frobenius of $v$. 

The following is the main result in this section. 
\begin{thm}\label{thm:det non p prime}
For any $j \in \Z_{\geq 1}$, there is a (unique) basis 
\[
\cH_{K_\infty /k, v}^j \in \Det_{\Lambda}^{-1} (\Delta_{K_\infty, v} (j)) 
\]
which is sent by the composite map 
\[
 \Det_{\Lambda}^{-1} (\Delta_{K_\infty , v} (j)) \to
 \Q_p \otimes_{\Z_p} \Det_{\Z_p[\GG_{K_n}]}^{-1} (\Delta_{K_n , v} (j)) \overset{\vartheta_{K_n, v}^j}{\xrightarrow{\sim}}
\Q_p [\GG_{K_n}],
\]
for any $n \in \Z_{\geq 0}$, to
\[
\begin{cases}
	\cfrac{\epsilon_v^{j-1} (K_n / k)}{\epsilon_v^{j} (K_n / k)} 
		& \text{if } j \geq 2, \\
	\cfrac{\epsilon_v^{0} (K_n / k) - e_{D_{K_n, v}} \# (D_{K_n, v} / I_{K_n, v})^{-1}}{\epsilon_v^{1} (K_n / k)} 
		& \text{if } j =1.
\end{cases}
\]
\end{thm}

If $K_\infty / k$ is unramified at $v$, 
it is convenient to modify the basis as follows. 

\begin{cor}\label{cor:unr_basis}
Assume that $K / k$ is unramified at $v$. Then, 
for any $j \in \Z_{\geq 1}$, there is a (unique) basis 
\[
\cE_{K_\infty /k, v}^j \in \Det_{\Lambda}^{-1} (\Delta_{K_\infty, v} (j)) 
\]
which is sent by the composite map
\[
 \Det_{\Lambda}^{-1} (\Delta_{K_\infty, v} (j)) \to
 \Q_p \otimes_{\Z_p} \Det_{\Z_p[\GG_{K_n}]}^{-1} (\Delta_{K_n , v} (j)) \overset{\vartheta_{K_n, v}^j}{\xrightarrow{\sim}}
\Q_p [\GG_{K_n}],
\]
for any $n \in \Z_{\geq 0}$, to
\[
\begin{cases}
	\cfrac{\epsilon_v^{1-j} (K_n / k)^\#}{\epsilon_v^{j} (K_n / k)} 
		& \text{if } j \geq 2, \\
	\cfrac{\delta_v (K_n / k)^\#}{\epsilon_v^{1} (K_n / k)} 
		& \text{if } j =1.
\end{cases}
\]
\end{cor}

\begin{proof}
Put $\cE_{K_\infty /k, v}^j := - N (v)^{j-1} \sigma_{K_\infty / k, v}^{-1} \cdot \cH_{K_\infty /k, v}^j$, 
where $\sigma_{K_\infty / k, v} \in \Gal (K_\infty / k)$ is the arithmetic Frobenius of $v$.  
Then the interpolation properties follow from simple computations.
\end{proof}

\subsection{Construction of the basis}\label{ss:c_basis_l}

Let $Q$ be the total ring of fractions of $\Lambda$. 
Since $Q \otimesL_{\Lambda} \Delta_{K_\infty, v} (j)$ is acyclic, we have a canonical isomorphism 
\begin{equation}\label{canonical det}
Q \otimes_{\Lambda} \Det_{\Lambda}^{-1} (\Delta_{K_\infty , v} (j))
 \simeq Q
\end{equation}
for any $j \in \Z$ by Proposition \ref{prop:det_ses3}.
Set $\GG_{K_\infty} := \Gal (K_\infty / k)$ and we write $I_{K_\infty, v} \subset \GG_{K_{\infty}}$ for the inertia group 
of $v$.
Note that $I_{K_\infty, v}$ is a finite group. 
Define 
\[
h_{K_\infty / k , v}^j := 
\cfrac{1-e_{I_{K_\infty , v}} N (v)^{1-j}\sigma_{K_\infty / k, v}}{1-e_{I_{K_\infty , v}} N(v)^{-j}\sigma_{K_\infty / k, v}} 
\in Q.
\]

\begin{prop}\label{prop:h}
Under the canonical isomorphism \eqref{canonical det}, we have
\[
\Det_{\Lambda}^{-1} (\Delta_{K_\infty, v} (j)) = (h_{K_\infty / k, v}^j)
\]
for any $j \in \Z. $
\end{prop}

\begin{proof}
The claim for $j = 1$ is proved in \cite[Proposition 3.13]{GKK_09}.
The general case follows simply by twisting.
\end{proof}
We define the $\Lambda$-basis $\cH_{K_\infty /k, v}^j$ of $\Det_{\Lambda}^{-1} (\Delta_{K_\infty, v} (j))$ in Theorem \ref{thm:det non p prime} as the element which corresponds to $h_{K_\infty / k, v}^j$ by Proposition \ref{prop:h}. 

\subsection{The interpolation properties}\label{ss:interp_l}

We prove the interpolation properties in Theorem \ref{thm:det non p prime}. 
Take $n \in \Z_{\geq 0}$ and let $\chi$ be a character of $\GG_{K_n}$.
In the same way as in \S \ref{ss:interp_p}, 
we define $\fq_\chi$, $\Lambda_{\fq_\chi}$, $\Delta_{K_{\infty}, v} (j)_{\fq_\chi}$, and $\Delta_{K_n , v} (j)^\chi$.
We have a canonical isomorphism
\begin{equation}\label{complex v descent chi}
\Lambda_{\fq_\chi} / \fq_\chi \otimesL_{\Lambda_{\fq_\chi}} \Delta_{K_\infty, v} (j)_{\fq_\chi}
 \simeq 
\Delta_{K_n , v} (j)^\chi. 
\end{equation}

We consider the map
\[
f_{\fq_\chi}^{j} : 
\Det_{\Lambda_{\fq_\chi}}^{-1} (\Delta_{K_{\infty, v}} (j)_{\fq_\chi})
\overset{\eqref{complex v descent chi}}{\twoheadrightarrow}
\Det_{\Q_p (\chi)}^{-1} (\Delta_{K_{n,v}} (j)^\chi) 
\overset{\vartheta_{K_n, v}^{j}}{\xrightarrow{\sim}}
\Q_p (\chi). 
\]
We use the same symbol $\cH_{K_\infty / k , v}^j$ for the image of the basis $\cH_{K_\infty /k , v}^j$ by the natural map  
$\Det_{\Lambda}^{-1} (\Delta_{K_\infty, v} (j)) \to 
\Det_{\Lambda_{\fq_\chi}}^{-1} (\Delta_{K_\infty , v} (j)_{\fq_\chi})$ if no confusion occurs. 

In the rest of this section, we show
\[
f_{\fq_\chi}^{j}
(\cH_{K_\infty / k, v}^{j})
=
\begin{cases}
\cfrac{\epsilon_v^{j-1} (K_n / k)^\chi}{\epsilon_v^{j} (K_n / k)^\chi} & \text{if } j \geq 2, \\
\parenth{\epsilon_v^{0} (K_n / k) - e_{D_{K_n, v}} \cfrac{1}{\# (D_{K_n, v} / I_{K_n, v})}}^\chi
(\epsilon_v^{1} (K_n / k)^\chi)^{-1} & \text{if } j =1
\end{cases}
\]
for any $\chi \in \widehat{\GG_{K_n}}$, which establishes Theorem \ref{thm:det non p prime}.

\subsubsection{The case $j \geq 2$ or  $\ker \chi \not\supset D_{K_n, v}$}

We first consider the case where $j \geq 2$ or $\ker \chi \not\supset D_{K_n, v}$. 
In this case, both 
$\Delta_{K_\infty, v} (j)_{\fq_\chi}$ and $\Delta_{K_n, v} (j)^\chi $ are acyclic and 
we have a commutative diagram 
\[
\xymatrix{
\Det_{\Lambda_{\fq_\chi}}^{-1} (\Delta_{K_\infty, v} (j)_{\fq_\chi}) 
\ar[r]^-{\sim} \ar[d]^-{\eqref{complex v descent chi}}
& \Lambda_{\fq_\chi} \ar[d]^-{\text{mod } \fq_\chi} \\
\Det_{\Q_p (\chi)}^{-1} (\Delta_{K_n, v} (j)^\chi) \ar[r]^-{\sim}
&\Q_p (\chi), 
}
\]
where both horizontal arrows are canonical maps in Proposition \ref{prop:det_ses3}. 
The bottom arrow in this diagram is equal to the $\chi$-component of 
$\vartheta_{K_n, v}^{j}$, and we have the desired interpolation property 
\[
f_{\fq_\chi}^{j}
(\cH_{K_\infty / k , v}^{j})
= \chi (h^{j}_{K_\infty / k, v}) 
=
\cfrac{\epsilon_v^{j-1} (K_n / k)^\chi}{\epsilon_v^{j} (K_n / k)^\chi}
\]
in this case. 

\subsubsection{The case $j = 1$ and  $\ker \chi \supset D_{K_n, v}$}

We consider the case where $j =1$ and $\ker \chi \supset D_{K_n, v}$.
We put $\fq = \fq_{\chi}$ for simplicity.
In this case, 
we have $H^1 (\Delta_{K_\infty, v} (1)_{\fq}) = 0$, but $H^2 (\Delta_{K_\infty, v} (1)_{\fq}) \simeq 
H^1(\Delta_{K_n, v} (1)^{\chi}) \simeq H^2(\Delta_{K_n, v} (1)^{\chi}) \simeq \Q_p (\chi)$ are non-zero.
Thus we need some technical descent arguments 
as in \S \ref{chi = 1 and j=1}. 

We fix a topological generator $\gamma$ of $\Gal (K_\infty / K_n)$ and put $\omega:= 1-\gamma$. 
Then $\omega$ is a uniformizer of $\Lambda_{\fq}$. 
Since $\Lambda_{\fq}$ is a discrete valuation ring, 
we can identify $\Delta_{K_\infty, v} (1)_{\fq} = [C^1 \to C^2]$, where both $C^1$ and $C^2$ are free $\Lambda_{\fq}$-modules of the same rank. 
Then we have $\Delta_{K_n, v} (1)^{\chi}= [\overline{C}^1 \to \overline{C}^2]$ with $\overline{C}^i := C^i / \omega$. 
By Proposition \ref{prop:DetA}, we obtain a commutative diagram
\[
\xymatrix{
\Det^{-1}_{\Lambda_{\fq}} (\Delta_{K_\infty, v} (1)_{\fq}) \ar[r]^-{\times \omega^{-1}}_-{\sim} 
\ar@{->>}[d]^-{\eqref{complex v descent chi}} 
&  \Det_{\Lambda_{\fq}} (H^1(\Delta_{K_\infty, v} (1)_{\fq})) = \Lambda_{\fq} \ar@{->>}[d]^-{\text{mod } \omega} \\
\Det^{-1}_{\Q_p (\chi)} (\Delta_{K_n , v} (1)^{\chi}) \ar[r]^-{\beta_\omega}_-{\sim} 
& \Det_{\Q_p (\chi)} (H^1(\Delta_{K_\infty , v} (1)_{\fq}) / \omega) = \Q_p (\chi). 
}
\]
Here the bottom horizontal map $\beta_{\omega}$ is induced by the isomorphism 
\begin{equation}\label{Bockstain map 1}
\beta_\omega' : H^1 (\Delta_{K_n,  v} (1)^{\chi}) \xrightarrow{\sim} 
H^2 (\Delta_{K_\infty, v} (1)_{\fq}) \simeq 
H^2 (\Delta_{K_n,  v} (1)^{\chi}), 
\end{equation}
where the first map is induced by the exact triangle 
$\Delta_{K_\infty, v} (1)_{\fq} \xrightarrow{\times \omega}
\Delta_{K_\infty, v} (1)_{\fq} \to \Delta_{K_n,  v} (1)^{\chi}$ and the 
second isomorphism is the natural one. 
Then, by \cite[Lemma 5.8 and its proof]{Fl04}, we know that 
the isomorphism $\inv_{K_n, v}^{-1} \circ \ord_{K_n, v} :  H^1 (\Delta_{K_n,  v} (1)^{\chi}) \xrightarrow{\sim} 
H^2 (\Delta_{K_n,  v} (1)^{\chi})$ in \S \ref{ss:setup}
 coincides with $c_{v, \gamma}^{-1}\cdot \beta_\omega'$, 
where $c_{v, \gamma} \in \Z_p$ is defined by $\gamma^{-c_{v, \gamma}} = \sigma_{K_\infty / K_n, v}$ 
($\sigma_{K_\infty / K_n, v}$ is the arithmetic Frobenius of $v$ in $\Gal (K_\infty / K_n)$). 
From this, we have a commutative diagram
\[
\xymatrix@C=50pt{
\Det^{-1}_{\Lambda_{\fq}} (\Delta_{K_\infty, v} (1)_{\fq}) \ar[r]^-{\times \omega^{-1}}_-{\sim} 
\ar@{->>}[d]^-{\eqref{complex v descent chi}} 
&\Lambda_{\fq} \ar@{->>}[d]^-{\text{mod } \omega}
\\
\Det^{-1}_{\Q_p (\chi)} (\Delta_{K_n, v} (1)^{\chi}) \ar[r]^-{c_{v, \gamma} \cdot\vartheta_{K_n, v}^{1}}_-{\sim} 
& \Q_p (\chi). 
}
\]
Let $Q_{\fq}$ be the total quotient ring of $\Lambda_{\fq}$. 
Then $Q_{\fq} \otimesL_{\Lambda_{\fq}} \Delta_{K_\infty, v} (1)_{\fq}$ is 
acyclic and $\omega^{-1}$ times the canonical isomorphism 
$Q_{\fq} \otimes_{\Lambda_{\fq}} \Det_{\Lambda_{\fq}}^{-1} (\Delta_{K_\infty, v} (1)_{\fq}) 
\xrightarrow{\sim} Q_{\fq}$ is equal to the one which is induced by the top horizontal map in the above diagram. 
Therefore, we have 
\begin{align*}
f_{\fq}^{1} (\cH_{K_\infty / k, v}^{1}) &=
\chi \parenth{\cfrac{h_{K_\infty / k , v}^{1}}{\omega}} c_{v, \gamma}^{-1} 
=
\cfrac{c_{v, \gamma}^{-1}}{\epsilon_v^1 (K_n/ k)^{\chi}} \cdot \chi 
\parenth{\cfrac{1-\widetilde{\sigma}_{K_\infty / k, v}}{1-\gamma}} \\
&=
\cfrac{c_{v, \gamma}^{-1}}{\epsilon_v^1 (K_n/ k)^{\chi}} \cdot 
\chi \parenth{\cfrac{1-\sigma_{K_\infty / K_n , v}}{(1-\gamma) (1 + \widetilde{\sigma}_{K_\infty / k , v} + 
\cdots + \widetilde{\sigma}_{K_\infty / k , v}^{\# (D_{K_n, v} / I_{K_n, v})-1})}} 
\\
&=
\cfrac{c_{v, \gamma}^{-1}}{\epsilon_v^1 (K_n/ k)^\chi} \cdot \cfrac{1}{\# (D_{K_n, v} / I_{K_n, v})} \cdot
\chi \parenth{\cfrac{1-\gamma^{-c_{v, \gamma}}}{1-\gamma}} \\
&= 
\cfrac{c_{v, \gamma}^{-1}}{\epsilon_v^1 (K_n/ k)^\chi} \cdot \cfrac{1}{\# (D_{K_n, v} / I_{K_n, v})} 
\cdot (- c_{v, \gamma}) \\
&=
\cfrac{-1}{\epsilon_v^1 (K_n/ k)^\chi \cdot \# (D_{K_n, v} / I_{K_n, v})}, 
\end{align*}
where $\widetilde{\sigma}_{K_\infty / k , v} \in \GG_{K_\infty}$ is a lift of 
the arithmetic Frobenius $\sigma_{K_\infty / k , v} \in \GG_{K_\infty} / I_{K_\infty, v} $. 
This completes the proof of Theorem \ref{thm:det non p prime}. \qed

\section{Archimedean considerations}\label{sec:arch}

\subsection{The Gauss sums}

In this subsection, we review the definition and basic properties of Gauss sums (cf.~\cite[Chapter 1, \S 5]{Fro83}).

Let $K / k$ be a finite abelian extension of number fields and $G:= \Gal(K/k)$ its Galois group. Let $v$ be a finite prime of $k$.
We write $D_v \subset G$ for the decomposition group of $v$. 
We write
\[
\rec_{k_v}: k_v^\times \twoheadrightarrow D_v
\]
for the local reciprocity map, which sends a uniformizer to an {\it arithmetic} Frobenius.

Let $\psi_{k_v} : k_v \to \C^{\times}$ be the additive character defined by 
\[
\psi_{k_v} : k_v \xrightarrow{\Tr_{k_v / \Q_l}} \Q_l \twoheadrightarrow \Q_l / \Z_l 
\hookrightarrow \Q/ \Z \xrightarrow{\exp ( 2 \pi \sqrt{-1}(-))} \C^\times , 
\]
where $l$ is the prime number below $v$ and 
the last map sends $x$ to $\exp ( 2 \pi \sqrt{-1}x)$. 
For a character $\chi$ of $G$, we put $\chi_v := \chi|_{D_v} : D_v \to \C^\times$ and
\[
\theta_{\chi_v}: k_v^\times \overset{\rec_{k_v}}{\twoheadrightarrow}
D_v \xrightarrow{\chi_v} \C^\times.
\]

\begin{defn}[Gauss sums]\label{def:Gsum}
Let $\chi$ be a character of $G$.
For any finite prime $v$ of $k$,
we define the local Gauss sum $\tau_{k}(\chi_v) \in \ol{\Q}$ by
\[
\tau_k (\chi_v) :=
\begin{cases}
\sum_{u \in \OO_{k_v}^\times / 1 + \ff(\chi_v)} 
\theta_{\chi_v} (uc^{-1}) \psi_{k_v} (uc^{-1})  & \text{if } \ff(\chi_v) \neq(1), \\ 
\theta_{\chi_v} (c^{-1}) & \text{if } \ff(\chi_v) = (1), 
\end{cases}
\]
where $\ff(\chi_v)$ is the conductor of $\chi_v$ and 
$c \in k_v^\times$ is an element satisfying
\[
(c) = \ff(\chi_v) \cdot \cD_{k_v / \Q_l},  
\]
where $\cD_{k_v / \Q_l}$ is the different ideal of $k_v / \Q_l$.
It is known that $\tau_{k}(\chi_v)$ is independent of the choice of $c$.
We also define the (global) Gauss sum $\tau_k(\chi) \in \ol{\Q}$ by 
\[\tau_k(\chi) := \prod_{v} \tau_k (\chi_v),
\]
where $v$ runs over all finite primes of $k$. This is actually a finite product since $k/\Q$ is unramified and $\ff(\chi_v)=(1)$ for almost all finite primes $v$ of $k$, and such $v$ satisfies $\tau_{k}(\chi_v)=1$.
\end{defn}

\subsection{The functional equations}

We review the functional equations of the Hecke $L$-function $L(\chi, s)$ of $k$ associated with a character $\chi : \Gal(K/k) \to \C^{\times}$. 

\begin{defn}\label{defn:A}
We write $r_\R (k)$ and $r_\C (k)$ for the number of the real places and complex places of $k$ respectively. 
For any character $\chi$ of $G$, 
we also write $r_\R (k)_{\ram}^\chi$ for the number of the real places of $k$ where $\chi$ ramifies.
For $j \in \Z_{\geq 1}$, we then put 
\[A(k, \chi, j):=
\begin{cases}
2^{r_\R (k) + r_\C (k) - r_\R (k)_{\ram}^\chi} \cdot 
(\pi \sqrt{-1})^{ r_\C (k) + r_\R (k)_{\ram}^\chi} 
& \text{if } j \text{ is even}, \\
2^{ r_\C (k) + r_\R (k)_{\ram}^\chi} \cdot
(\pi \sqrt{-1})^{r_\R (k) + r_\C (k) - r_\R (k)_{\ram}^\chi}
& \text{if } j \text{ is odd}.
\end{cases}
\]
\end{defn}

Put $r_k := [k : \Q]$.
We put
\[
\delta_{\chi={\trivial} , j=1} := \begin{cases}
-1 & \text{if }\chi = {\trivial} \text { and } j=1, \\
1 & \text{otherwise.}
\end{cases}
\]
Let $\mathfrak{f}_\chi$ be the conductor ideal of $\chi$.

\begin{thm}\label{thm:func eq}
For any character $\chi : \Gal (K/k) \to \C^{\times}$ and $j \in \Z_{\geq 1}$, we have 
\begin{align*}
\cfrac{L^\ast (\chi^{-1}, 1-j)}{L^\ast (\chi, j)}  
&=
\delta_{\chi={\trivial} , j=1} \cdot  |D_k|^{j-\frac{1}{2}} \cdot  N (\mathfrak{f}_\chi)^{j-1} \cdot \tau_k ({\chi}) \cdot (\sqrt{-1})^{r_\C (k)}
\cdot \cfrac{((j-1)!)^{r_k}}{(2 \pi \sqrt{-1})^{jr_k}} 
\cdot A(k, \chi, j),
\end{align*}
where $L^\ast (\chi, j)$ is the leading coefficient of the Taylor expansion of $L(\chi, s)$ at $s=j$ and $N (\mathfrak{f}_\chi) = \# (\OO_k / \mathfrak{f}_\chi)$.
\end{thm}

\begin{proof}
Let us deduce this formula from a usual formulation of the functional equation (cf.~\cite[Hecke's theorem in p.36]{Fro83}).
It claims
\[
\parenth{\prod_{v \in S_\infty (k)} L_v (\chi, s)} \cdot L (\chi, s) 
= \cfrac{\tau_k ({\chi^{-1}}) \cdot \parenth{|D_k| N (\mathfrak{f}_\chi) }^{\frac{1}{2}-s} }{\sqrt{N (\mathfrak{f}_\chi)}\cdot (\sqrt{-1})^{r_\R (k)_{\ram}^\chi}} \cdot
\parenth{\prod_{v \in S_\infty (k)} L_v (\chi^{-1}, 1-s) } \cdot
L (\chi^{-1}, 1-s)
\]
with
\[
L_v(\chi, s) =
\begin{cases}
\pi^{-\frac{s}{2}} \Gamma (\frac{s}{2}) 
&\text{if } v  \text{ is a real place and unramified at } \chi, \\
\pi^{-\frac{s+1}{2}} \Gamma (\frac{s+1}{2})
&\text{if } v  \text{ is a real place and ramifies at } \chi, \\
2 (2 \pi)^{-s} \Gamma (s)
&\text{if }  v  \text{ is a complex  place}.   
\end{cases}
\]
By properties of Gamma functions, we obtain
\[
\cfrac{L_v(\chi, s)}{L_v (\chi^{-1}, 1-s)} =
\begin{cases}
 2 (2 \pi)^{-s} \Gamma (s) \cdot \cos \frac{\pi}{2}s 
&\text{if } v  \text{ is a real place and unramified at } \chi, \\
 2 (2 \pi)^{-s} \Gamma (s) \cdot \sin \frac{\pi}{2}s 
&\text{if } v  \text{ is a real place and ramifies at } \chi, \\
2 (2 \pi)^{-2s} \Gamma (s)^2 \cdot \sin \pi s
&\text{if }  v  \text{ is a complex  place}. 
\end{cases}
\]
Therefore, we have 
\begin{align*}
&L (\chi^{-1}, 1-s) = 
\parenth{|D_k| N (\mathfrak{f}_\chi) }^{s-\frac{1}{2}}
\cfrac{\sqrt{N (\mathfrak{f}_\chi)} (\sqrt{-1})^{r_\R (k)_{\ram}^\chi}}{\tau_k ({\chi^{-1}} )}
\parenth{ \prod_{v \in S_\infty (k)} \cfrac{L_v (\chi, s)}{L_v (\chi^{-1}, 1-s)} }\cdot L (\chi, s) \\
&= \parenth{|D_k| N (\mathfrak{f}_\chi) }^{s-\frac{1}{2}}
\cfrac{\sqrt{N (\mathfrak{f}_\chi)}}{\tau_k ({\chi^{-1}})}
\cdot \cfrac{\Gamma (s)^{r_k}}{(2 \pi)^{sr_k}} L (\chi, s) \\
& \qquad
\times \left(2\cos \frac{\pi}{2}s\right)^{r_\R (k) - r_\R (k)_{\ram}^\chi} \cdot 
\left(2\sqrt{-1} \sin \frac{\pi}{2}s\right)^{r_\R (k)_{\ram}^\chi} \cdot
(2\sin \pi s)^{r_\C (k)}. 
\end{align*}
For each positive integer $j$, we have
\[
\cos \left(\frac{\pi}{2}s\right)^\ast_{s=j} =
\begin{cases}
(-1)^{\frac{j}{2}} & \text{if } j \text{ is even},  \vspace{2mm}\\
(-1)^{\frac{j+1}{2}} \dfrac{\pi}{2} & \text{if } j \text{ is odd},  
\end{cases} \;\;\;\;\;
\sin \left(\frac{\pi}{2}s\right)^\ast_{s=j} =
\begin{cases}
(-1)^{\frac{j}{2}} \dfrac{\pi}{2} & \text{if } j \text{ is even}, \vspace{2mm}\\
(-1)^{\frac{j-1}{2}}  & \text{if } j \text{ is odd}, 
\end{cases}
\]
and 
$\sin (\pi s)^\ast_{s=j} =
(-1)^j \pi$, 
where $(-)^\ast_{s=j}$ means the leading term of Taylor expansion at $s=j$. 
Let $\lambda := \lambda (\chi^{-1}, 1-j) \in \Z_{\geq0}$ be the order of zero of $L (\chi^{-1} , s)$ at $s=1-j$. 
Then, by the above formulas, we have 
\begin{align*}
& L^\ast (\chi^{-1}, 1-j) = \lim_{s \to 1-j} (s-(1-j))^{-\lambda} \cdot L (\chi^{-1}, s) 
= (-1)^{\lambda} \cdot \lim_{s' \to j}  (s'-j)^{-\lambda} \cdot L (\chi^{-1}, 1-s') \\
& =
(-1)^{\lambda} \cdot \parenth{|D_k| N (\mathfrak{f}_\chi) }^{j-\frac{1}{2}}
\cfrac{\sqrt{N (\mathfrak{f}_\chi)}}{\tau_k ({\chi^{-1}})}
\cdot \cfrac{((j-1)!)^{r_k}}{(2 \pi)^{jr_k}} \cdot L^\ast (\chi, j)\\
&\times 
\begin{cases}
(\sqrt{-1})^{jr_k + r_\R (k)_{\ram}^\chi} \cdot 
2^{r_\R (k) - r_\R (k)_{\ram}^\chi + r_\C (k)} \cdot 
\pi^{r_\R (k)_{\ram}^\chi + r_\C (k)} 
& \text{if } j \text{ is even},  \\
(\sqrt{-1})^{jr_k + r_\R (k)- r_\R (k)_{\ram}^\chi} \cdot
2^{r_\R (k)_{\ram}^\chi + r_\C (k)} \cdot
\pi^{r_\R (k) - r_\R (k)_{\ram}^\chi + r_\C (k)}
& \text{if } j \text{ is odd},  
\end{cases}\\
&=(-1)^{\lambda} \cdot
 \cfrac{(-1)^{jr_k} \cdot |D_k|^{j-\frac{1}{2}} \cdot  N (\mathfrak{f}_\chi)^j}{(\sqrt{-1})^{r_\C (k)}
\cdot  \tau_k ({\chi^{-1}})}\cdot \cfrac{((j-1)!)^{r_k}}{(2 \pi \sqrt{-1})^{jr_k}} 
\cdot A(k, \chi, j) \cdot L^\ast (\chi, j), 
\end{align*}
where we replace $s' = 1-s$ in the second equality. 
Since 
\[
\tau_k (\chi) \cdot \tau_k (\chi^{-1}) = N (\mathfrak{f}_\chi) \cdot \prod_{v: \text{ finite prime of }k}\chi (\rec_{k_v}(-1))  
= N (\mathfrak{f}_\chi) \cdot (-1)^{r_\R (k)^\chi_{\ram}}
\] 
(cf.~\cite[equation (5.7a) in Chapter 1, \S 5]{Fro83}) and 
\[
\lambda(\chi^{-1}, 1-j) = 
\begin{cases}
r_\R (k)_{\ram}^\chi + r_\C (k) & \text{if } j \text{ is even}, 
\\ 
\displaystyle
r_\R(k)  + r_\C (k) -1 & \text{if } j =1 \text { and } \chi = \trivial, 
\\
r_\R(k) -r_\R (k)_{\ram}^\chi + r_\C (k) & \text{otherwise}, 
\end{cases}
\]
we have \[
(-1)^{\lambda(\chi^{-1}, 1-j)} \cdot (-1)^{jr_k} \cdot \frac{ N (\mathfrak{f}_\chi)^j}{\tau_k (\chi^{-1})} 
=\delta_{\chi={\trivial} , j=1} \cdot (-1)^{r_\C (k) } \cdot N (\mathfrak{f}_\chi)^{j-1} \cdot \tau_k (\chi), 
\]
using $r_k = r_\R (k) + 2 r_\C (k)$. 
Thus we obtain Theorem \ref{thm:func eq}. 
\end{proof}

\subsection{The description of $\alpha$}\label{ss:alpha}

Recall that $\Sigma_k := \{\iota : k \hookrightarrow \C \}$ is set of the embeddings. 
We fix a labeling $\Sigma_k = \{\iota_{k, 1} , \dots , \iota_{k, r_k} \}$.
Moreover, for each $1 \leq i \leq r_k $, we choose $\iota_{K, i} \in \Sigma_K$ such that $\iota_{K, i}|_k = \iota_{k, i}$.
\begin{defn}\label{defn:B}
For any $x_1, \dots , x_{r_k} \in K$, we define
\begin{equation}\label{det B}
B_{K/k}(x_1, \dots, x_{r_k}) = \det \parenth{\sum_{\sigma \in G} 
(\iota_{K, i} (x_m^\sigma))\sigma^{-1}}_{1 \leq i,m \leq r_k} \in \C [G]. 
\end{equation}
\end{defn}

For each $j \in \Z_{\geq 1}$, we consider the map
\[
\wedge \alpha_K^j : 
\C \otimes_{\Q}
\bigwedge_{\Q[G]}^{r_k} K    \to
\C  \otimes_{\Q}  \bigwedge_{\Q[G]}^{r_k} H_K (j) 
\]
induced by the period map $\alpha_K^j$ defined in \S \ref{period map}. 
Let us show an explicit description of this map. Recall that we take a $\Q[G]$-basis $\{b_{\iota_{K, i}}^j\}_{1\leq i \leq r_k}$ of $H_K(j)$, where $b_{\iota_{K, i}}^j$ is the element of $H_K(j)$ whose $\iota_{K, i}$-component is $(2\pi\sqrt{-1})^j$ and the other components are zero.

\begin{lem}\label{wedge alpha}
For any $x_1, \dots , x_{r_k} \in K$ and a character $\chi$ of $G$, we have
\[ 
\parenth{ \wedge \alpha_K^j (1 \otimes \wedge_{m=1}^{r_k} x_m )}^\chi 
=A(k, \chi, j) \cdot 
 \cfrac{B_{K/k}(x_1, \dots, x_{r_k})^\chi}
{(2 \pi \sqrt{-1})^{jr_k}}  (\wedge_{i=1}^{r_k} b_{\iota_{K, i}}^j)^\chi.
\]
\end{lem}

\begin{proof}
By permutation of the labels if necessary, we may assume that $\iota_{k, i}$ is real for $1 \leq i \leq r_\R (k)$ and 
$c_{\R} \circ \iota_{k, i} = \iota_{k, i + r_\C (k)}$ for any $r_\R (k) +1 \leq i \leq r_\R (k)  + r_\C (k)$, 
where $c_{\R} \in \Gal(\C / \R)$ is the complex conjugation. 
Also for any $1 \leq i \leq r_k$, 
we fix an embedding $\iota_{K, i} \in \Sigma_K$ such that 
$\iota_{K, i}|_k = \iota_{k, i}$ for any $1 \leq i \leq r_k$ 
and $c_{\R} \circ \iota_{K, i}= \iota_{K, i+r_\C (k)}$ any $r_\R (k) +1 \leq i \leq r_\R (k)  + r_\C (k)$. 

We only show the case where $j$ is even because the other case can be shown in a similar way.
Then, by the formula \eqref{alpha},  we have for any $x \in K$
\begin{align*}
\alpha_K^j (1 \otimes x)
&=
\displaystyle
\sum_{\iota_K \in \Sigma_K}
\parenth{(2 \Real (\iota_K (x)) - \pi \Imag (\iota_K (x)))\cfrac{1}{(2\pi \sqrt{-1})^j}}b_{\iota_K}^j \\
&=
\sum_{i=1}^{r_k} \parenth{ \sum_{\sigma \in G}
(2 \Real (\iota_{K,i} (x^\sigma)) - \pi \Imag (\iota_{K,i} (x^\sigma)))\cfrac{1}{(2\pi \sqrt{-1})^j} \sigma^{-1} } 
b_{\iota_{K,i}}^j .
\end{align*}
For any $x_1, \dots , x_{r_k} \in K$, we consider
\[
B' (x_1, \dots x_{r_k}) = 
 \det \parenth{ \sum_{\sigma \in G}
(2 \Real (\iota_{K,i} (x_m^\sigma)) - \pi \Imag (\iota_{K,i} (x_m^\sigma))) \sigma^{-1} }_{1 \leq i,m \leq r_k }. 
\] 
Then we have 
\begin{equation}\label{A}
\wedge \alpha_K^j (1 \otimes \wedge_{m=1}^{r_k} x_m ) 
= \cfrac{1}{(2 \pi \sqrt{-1})^{jr_k}} B' (x_1, \dots x_{r_k}) 
\wedge_{i=1}^{r_k} b_{\iota_{K, i}}^j . 
\end{equation}
We compute $B' (x_1, \dots x_{r_k})$. 

If $\iota_{K, i}$ is a real embedding, 
we have 
\begin{equation}\label{1}
\sum_{\sigma \in G} 
(2 \Real (\iota_{K, i} (x_m^\sigma)) - \pi \Imag (\iota_{K, i} (x_m^\sigma)))\sigma^{-1} = 
2\sum_{\sigma\in G} 
(\iota_{K, i}(x_m^\sigma))\sigma^{-1}
\end{equation}
for any $1 \leq m \leq r_k$. 

If $\iota_{k, i}$ is real but $\iota_{K, i}$ is not real
(i.e. the place $\iota_{k, i}$ is ramified in $K/k$), 
letting $c_i \in G$ be the generator of the inertia group of $\iota_{K,i}$,
we have 
\begin{align}\label{2}
\parenth{\sum_{\sigma \in G} 
(2\Real (\iota_{K, i} (x_m^\sigma)) - \pi \Imag (\iota_{K, i} (x_m^\sigma)))\sigma^{-1}}^\chi 
=
\begin{cases}\displaystyle
2 \parenth{
\sum_{\sigma \in G} 
(\iota_{K, i} (x_m^\sigma))\sigma^{-1}}^\chi
& \text{ if } \chi(c_i)=1 \\
\displaystyle
\pi \sqrt{-1}
 \parenth{ \sum_{\sigma \in G} 
(\iota_{K, i} (x_m^\sigma))\sigma^{-1}}^\chi
& \text{ if }  \chi(c_i)=-1 
\end{cases} 
\end{align}
for any $1 \leq m \leq r_k$ and any character $\chi$ of $G$. 

If $r_\R (k) +1 \leq i \leq r_\R (k) + r_\C (k)$ (so $\iota_{k, i}$ is not real),
since $c_{\R} \circ \iota_{K,  i}= \iota_{K,  i +r_\C (k)}$, we have 
\[
\Real (\iota_{K,  i + r_\C (k)} (x_m^\sigma)) = \Real (\iota_{K,  i} (x_m^\sigma)), \;\;\;
\Imag (\iota_{K,  i + r_\C (k)} (x_m^\sigma)) = -\Imag (\iota_{K,  i} (x_m^\sigma))
\]
for any $1 \leq m \leq r_k$ and $\sigma \in G$. 

Combining these observations and using elementary row operation, 
we have 
\begin{align*}
B'(x_1, \dots, x_{r_k})^\chi & =
2^{r_\R (k)-r_\R (k)_{\ram}^\chi} \cdot (\pi \sqrt{-1})^{r_\R (k)_{\ram}^\chi} \cdot
(2 \pi \sqrt{-1})^{r_\C (k)} \cdot
B_{K/k}(x_1, \dots, x_{r_k})^\chi \\
&=
A (k, \chi, j) \cdot B_{K/k}(x_1, \dots, x_{r_k})^\chi. 
\end{align*}
for any character $\chi$ of $G$. This formula and \eqref{A} imply the 
formula in the Lemma \ref{wedge alpha} in the case where $j$ is even. 
\end{proof}

\subsection{Choice of a basis of $\Z_p \otimes \OO_K$}\label{ss:choice}

The purpose of this subsection is to find a basis $x_1, \dots , x_{r_k}$ of $\Z_p \otimes \OO_K$ such that we have a nice relationship between $B_{K/k}(x_1, \dots , x_{r_k})$ and Gauss sums. 

Let $K/k$ be a finite abelian extension of number fields such that $p$ is unramified in $K /\Q$ 
and put $G:= \Gal(K/k)$.  
For any non-negative integer $n$, 
we put $K_n := K (\mu_{p^n})$ and $k_n := k (\mu_{p^n})$. 
We write $\GG_{K_n} = \Gal(K_n/k)$ and $\GG_{k_n} = \Gal(k_n/k)$, so we have a natural isomorphism $\GG_{K_n} \simeq G \times \GG_{k_n}$. For any $v \in S_p$, 
we write $\sigma_{K_n, v} \in \GG_{K_n}$ for the arithmetic Frobenius of $v$ in $\Gal(K_n / k_n) (\subset \GG_{K_n})$. 

We consider a character $\chi$ of $\GG_{K_n}$.
Let $n_\chi \leq n$ be the minimum integer such that $\chi$ factors through $\GG_{K_{n_{\chi}}}$.
We also introduce the following notation.

\begin{setting}\label{setting 1}
For a character $\chi$ of $\GG_{K_n}$, 
we write 
\[
\chi = \chi_1 \chi_2
\]
 with characters $\chi_1$ of $G$ and $\chi_2$ of $\GG_{k_n}$.
We then 
define a character $\chi_2'$ of $\Gal (\Q(\mu_{p^n}) / \Q)$ as the composite map 
\[
\chi_2': \Gal (\Q(\mu_{p^n}) / \Q) \xrightarrow{\sim} \GG_{k_n} \xrightarrow{\chi_2} \C^{\times}, 
\]
where the first map is the inverse of the restriction map. 
\end{setting}

We introduce some basic properties of the local Gauss sums.
Claim (iii) below is a generalization of the Davenport--Hasse relation and its proof will be given in \S \ref{sec:DH}.

\begin{lem}\label{lem:Davenport--Hasse}
In Setting \ref{setting 1}, for a finite prime $v$ of $k$, the following statements hold.
\begin{itemize}
\item[(i)]
If $v \nmid p$, we have 
\[
\tau_k (\chi_v) 
=\chi_2 (\rec_{k_v} (\ff(\chi_v) \cdot \cD_{k_v / \Q_l})^{-1}) \tau_k (\chi_{1, v}). 
\]
\item[(ii)]
If $v \mid p$, we have 
\[
\tau_k (\chi_v) = \chi(\sigma_{K_n, v})^{-n_\chi} \cdot \tau_k (\chi_{2, v}).   
\]
\item[(iii)]
If $v \mid p$, we have 
\[
\tau_k (\chi_{2, v}) = (- 1)^{n_\chi(r_{k_v} -1)} \cdot \tau_{\Q} (\chi_{2, p}')^{r_{k_v}},
\]
where we put $r_{k_v} = [k_v : \Q_p]$.  
\end{itemize}
In particular, combining (ii) and (iii), 
we obtain
\[
\prod_{v \in S_p }\tau_k (\chi_v) =
\parenth{\prod_{v \in S_p  } \chi(\sigma_{K_n, v})^{-n_\chi} } \cdot
(- 1)^{n_\chi(r_{k} -\# S_p)} \cdot \tau_{\Q} (\chi_{2, p}')^{r_k}.
\]
\end{lem}

\begin{proof}
(i) 
We take $c \in k_v^\times$ satisfying $(c) = \ff(\chi_v) \cdot \cD_{k_v / \Q_l}$. 
If $\ff(\chi_v) = (1)$, we have 
\[
\tau_k (\chi_v) 
=\theta_{\chi_v} (c^{-1})  =\theta_{\chi_{2, v}} (c^{-1}) \theta_{\chi_{1,v}} (c^{-1})
=\chi_2 (\rec_{k_v} (\cD_{k_v / \Q_l})^{-1}) \tau_k (\chi_{1, v}). 
\]
If $\ff(\chi_v) \neq (1)$, we have 
\begin{align*}
\tau_k (\chi_v) 
&=
\sum_{u \in \OO_{k_v}^\times / 1 + \ff(\chi_v)} 
\theta_{\chi_v} (uc^{-1}) \psi_{k_v} (uc^{-1})  \\
&=
\theta_{\chi_{2, v}} (c^{-1}) \sum_{u \in \OO_{k_v}^\times / 1 + \ff(\chi_v)} 
\theta_{\chi_{1,v }} (uc^{-1}) \psi_{k_v} (uc^{-1})  \\
&=\chi_2 (\rec_{k_v} (\ff(\chi_v) \cdot \cD_{k_v / \Q_l})^{-1}) \tau_k (\chi_{1, v}). 
\end{align*}
(ii)
Since $K/\Q$ is unramified at $p$, 
we have $\cD_{k_v / \Q_p} =1$, $(p^{n_\chi}) = \ff(\chi_v) = \ff(\chi_{2, v}) $, 
and  
\begin{align*}
\theta_{\chi_v} (up^{-n_\chi}) 
&= 
\chi_v (\rec_{k_v}(u)) \cdot \chi_v (\rec_{k_v}(p))^{-n_\chi} \\
&=
\chi_{2, v} (\rec_{k_v}(u))\cdot \chi_{2,v} (\rec_{k_v}(p))^{-n_\chi}\cdot
\chi_{1,v} (\rec_{k_v}(p))^{-n_\chi} 
= 
\theta_{\chi_{2, v}} (up^{-n_\chi}) \cdot \chi (\sigma_{K_n, v} )^{-n_\chi}
\end{align*}
for any $u \in \OO_{k_v}^\times$ 
(note that $\chi_{1,v} (\rec_{k_v}(p)) = \chi (\sigma_{K_n, v} )$). 
Therefore, by the definition of Gauss sums, we have 
\begin{align*}
\tau_k (\chi_v)
&=
\sum_{u \in \OO_{k_v}^\times / 1 + p^{n_\chi} \OO_{k_v}} 
\theta_{\chi_v} (up^{-n_\chi}) \psi_{k_v} (up^{-n_\chi})  \\
&=\chi (\sigma_{K_n, v} )^{-n_\chi} \cdot 
\sum_{u \in \OO_{k_v}^\times / 1 + p^{n_\chi} \OO_{k_v}} 
\theta_{\chi_{2,v}} (up^{-n_\chi}) \psi_{k_v} (up^{-n_\chi}) 
= 
\chi(\sigma_{K_n, v})^{-n_\chi} \cdot \tau_k (\chi_{2, v}). 
\end{align*}
\end{proof}

We take a basis $\xi = (\zeta_{p^n})_n$ of $\Z_p(1) = \varprojlim_n \mu_{p^n}(K_{\infty})$.
Also, we choose $\{\iota_{K_n, 1}, \dots , \iota_{K_n, r_k} \} \subset \Sigma_{K_n}$ such that 
$\iota_{K_n, i} |_K = \iota_{K, i}$  and 
$\iota_{K_n, i} (\zeta_{p^n}) = \exp (\frac{2 \pi \sqrt{-1}}{p^n})$ in $\C$
 for each $1 \leq i \leq r_k$
(this choice is unique and possible since $p$ is unramified in $K/\Q$). 
For $n \geq 0$, 
we define $a_n \in K_n$ by 
\[
a_n := 
\begin{cases}
\frac{1}{p^{n-1}} (\zeta_p + \zeta_{p^2} + \cdots + \zeta_{p^n}) & \text{if } n >0, \\
-1 & \text{if } n =0, .
\end{cases}
\]
Using this, for any $x_1, \dots , x_{r_k} \in K$, we define $B_{K_n / k} (a_n x_1, \dots , a_n x_{r_k}) \in \C[\GG_{K_n}]$ by \eqref{det B}. 

\begin{lem}\label{B Gauss sum}
For any $x_1, \dots , x_{r_k} \in K$ and a character $\chi$ of $\GG_{K_n}$,
we have 
\begin{align*}
&B_{K_n / k} (a_n x_1, \dots , a_n x_{r_k})^\chi &\\
&= \begin{cases}
(-1)^{n_\chi (r_k - \# S_p)} \cdot B_{K/k} (x_1 , \dots , x_{r_k})^{\chi_1} \cdot 
\prod_{v \in S_p } 
\tau_{k} (\chi_v) \cdot \chi(\sigma_{K_n, v})^{n_\chi} \cdot N (v)^{1-n_\chi }
& \text{if } n_\chi >0, \\
(-1)^{r_k} \cdot B_{K/k} (x_1 , \dots , x_{r_k})^{\chi_1} & \text{if } n_\chi =0.
\end{cases}&
\end{align*}
\end{lem}
\begin{proof} 
By the natural isomorphism $\GG_{K_n} \simeq G \times \GG_{k_n}$, 
we regard $G \subset \GG_{K_n}$ and $\C[G] \subset \C[\GG_{K_n}]$. 
We compute $B_{K_n / k} (a_n x_1, \dots , a_n x_{r_k})$ as follows.  
\begin{align*}
B_{K_n / k} (a_n x_1, \dots , a_n x_{r_k}) 
&= 
\det \parenth{\sum_{\sigma \in \GG_{K_n}} 
(\iota_{K_n, i} (a_n^\sigma x_m^\sigma))\sigma^{-1}}_{1 \leq i,m \leq r_k} \\
&=
\det \parenth{\sum_{\sigma_2 \in \GG_{k_n}} \iota_{K_n, i} (a_n^{\sigma_2}) \sigma_2^{-1}
\parenth{\sum_{\sigma_1 \in G} 
(\iota_{K, i} ( x_m^{\sigma_1}))\sigma_1^{-1}}}_{1 \leq i,m \leq r_k}  \\
&=
\parenth{\sum_{\sigma_2 \in \GG_{k_n}} \iota_{K_n, i} (a_n^{\sigma_2}) \sigma_2^{-1}}^{r_k}
\cdot \det 
\parenth{\sum_{\sigma_1 \in G} 
(\iota_{K, i} ( x_m^{\sigma_1}))\sigma_1^{-1}}_{1 \leq i,m \leq r_k} \\
&=
\parenth{\sum_{\sigma_2 \in \GG_{k_n}} \iota_{K_n, i} (a_n^{\sigma_2}) \sigma_2^{-1}}^{r_k}
\cdot B_{K/k} (x_1 , \dots , x_{r_k}). 
\end{align*}
For any character $\chi$ of $\GG_{K_n}$, we have 
\[
\sum_{\sigma_2 \in \GG_{k_n}} a_n^{\sigma_2} \chi_2 (\sigma_2^{-1})
= [k_n : k] e_{\chi_2} a_n 
= 
\begin{cases}
\frac{[k_n : k]}{p^{n-1}} e_{\chi_2} \zeta_{p^{n_\chi}}
= \frac{1}{p^{n-1}}\sum_{\sigma_2 \in \GG_{k_n}} \zeta_{p^{n_\chi}}^{\sigma_2} \chi_2 (\sigma_2^{-1})
& \text{if } n_\chi >0, \\
-1 & \text{if } n_\chi =0, 
\end{cases}
\]
where $e_{\chi_2} \in \C[\GG_{k_n}]$ is the idempotent of $\chi_2$. 
This implies the lemma when $n_\chi = 0$. 

From now on, we assume that $n_\chi >0$. 
We write 
$\rec_{\Q_p} : \Z_p^\times / 1 + p^{n_\chi}\Z_p \xrightarrow{\sim} \Gal (\Q_p (\mu_{p^{n_\chi}}) / \Q_p)$
for the local reciprocity map. 
We note that 
$\zeta_{p^{n_\chi}}^{\rec_{\Q_p} (u)} =  (\zeta_{p^{n_\chi}})^{u^{-1}}$ for any $u \in  \Z_p^\times / 1 + p^{n_\chi}\Z_p$. 
Then, by the above formula, we have 
\begin{align*}
&\sum_{\sigma_2 \in \GG_{k_n}} \iota_{K_n, i} (a_n^{\sigma_2}) \chi (\sigma_2)^{-1}
=
\frac{1}{p^{n-1}} 
\sum_{\sigma_2 \in \GG_{k_n}} \iota_{K_n, i} (\zeta_{p^{n_\chi}}^{\sigma_2}) \chi_2 (\sigma_2)^{-1} \\
&=
\frac{1}{p^{n_\chi-1}} 
\sum_{u \in \Z_p^\times / 1 + p^{n_\chi}\Z_p} 
\chi_{2, p}' (\rec_{\Q_p} (u) )^{-1} \exp\left(\frac{2 \pi \sqrt{-1}}{p^{n_\chi}} u^{-1}\right) \\
&=\frac{1}{p^{n_\chi-1}} \sum_{u \in \Z_p^\times / 1 + p^{n_\chi}\Z_p} 
\theta_{\chi_{2, p}'}(u p^{-n_\chi}) \psi_{\Q_p} (up^{-n_\chi}) 
= \frac{1}{p^{n_\chi-1}}  \tau_{\Q} (\chi_{2, p}'),
\end{align*}
where $\chi_2'$ is the character of $\Gal (\Q (\mu_{p^n}) / \Q)$ defined in Setting \ref{setting 1}.
Therefore, we have
\begin{align*}
&B_{K_n / k} (a_n x_1, \dots , a_n x_{r_k})^\chi
= \frac{\tau_{\Q} (\chi_{2, p}')^{r_k}}{p^{r_k (n_\chi-1)}}   
\cdot B_{K/k} (x_1 , \dots , x_{r_k})^{\chi_1} \\
&= (-1)^{n_\chi (r_k - \# S_p )} \cdot B_{K/k} (x_1 , \dots , x_{r_k})^{\chi_1}
\cdot \prod_{v \in S_p } 
\tau_{k} (\chi_v) \cdot \chi(\sigma_{K_n, v})^{n_\chi} \cdot N (v)^{1-n_\chi }, 
\end{align*}
by using Lemma \ref{lem:Davenport--Hasse}. 
This completes the proof of Lemma \ref{B Gauss sum}. 
\end{proof}

Recall that we identify $\C$ and $\C_p$. 
Using this, we can extend each $\iota_K \in \Sigma_K $ to 
\[
K_p:= \Q_p \otimes_\Q K \to \C_p ; \;\; a \otimes x \mapsto a \cdot \iota_K (x),
\]
which is denoted by $\iota_K$ by abuse of notation.
For any $x_1, \dots , x_{r_k} \in K_p$, 
we then define $B_{K/k}(x_1, \dots, x_{r_k}) \in \C_p [G]$ by the same formula \eqref{det B}. 

For any finite prime $v$ of $k$, 
we define the equivariant local Gauss sum $\tau_{K/k, v}$ by
\[
\tau_{K/k, v} = \sum_{\chi \in \widehat{G}} \tau_k (\chi_v) e_\chi \in \C [G]
\] 
and then the equivariant (global) Gauss sum $\tau_{K/k}$ by 
\[
\tau_{K/k} = |D_k|^{\frac{1}{2}} (\sqrt{-1})^{r_\C (k)}  \prod_{v} \tau_{K/k, v},
\]
where $D_k$ is the absolute discriminant of $k$. 
This equivariant Gauss sum $\tau_{K/k}$ is equal to the one 
which is defined in \cite[page 4]{BlBu02} and \cite[equation (11)]{BlBu03} by \cite[line 3 of p. 553  and equation (14)]{BlBu03}.

The following is the goal of this subsection.
We fix a subset and a labeling 
$\{ \iota_{K, 1}, \dots , \iota_{K, r_k} \} \subset \Sigma_{K}$ as in \S \ref{ss:alpha}. 
 
\begin{prop}\label{prop: det B and Gauss sum}
There is a $\Z_p[G]$-basis $\{x_1, \dots , x_{r_k}\}$ of 
$\Z_p \otimes_\Z \OO_K$ such that 
\[
B_{K/k}(x_1, \dots , x_{r_k})
= \tau_{K/k} 
\prod_{v \in S_{\ram} (K/k)_f} \parenth{ \epsilon_v^0 (K/k)^\# - \cfrac{N_{I_v}}{\# I_v} }.
\]
Moreover, such a basis satisfies
\[
\det \parenth{\sum_{\sigma \in G} 
\Tr_{K/\Q}(x_a^{\sigma{^{-1}}} \cdot  x_b)\sigma
}_{1\leq a,b \leq r_k} 
=  |D_k| \cdot (-1)^{r_\C (k)}
 \prod_{v \in S_{\ram} (K/k)_f} \rec_{k_v} (-1) 
\sum_{\chi \in \widehat{G}} 
N(\ff_\chi) e_\chi. 
\]
\end{prop}

\begin{proof}
To find such a basis, we apply a result of 
Bley and Burns \cite{BlBu02}, \cite{BlBu03}. 
For each finite prime $v$ of $k$, 
we put
\[
f_v :=\epsilon_v^0 (K/k) - \cfrac{N_{I_v}}{\# I_v} =
1-\cfrac{N_{I_v}}{\# I_v} - \cfrac{N_{I_v}}{\# I_v}\sigma_{K,v} \in \Q[G],
\] 
where $\sigma_{K, v}$ is the arithmetic Frobenius in $\Gal(K/k)$ as defined at the beginning of \S \ref{ss:choice}.
A direct computation shows $f_v^{-1} = f_v^\#$. 

We consider the $\C_p[G]$-isomorphism
\[
\pi_{K/k} : 
\C_p \otimes K \xrightarrow{\sim} \bigoplus_{\iota_K \in \Sigma_K} \C_p
\;\; ; a \otimes x \mapsto (a \iota_K(x))_{\iota_K \in \Sigma_K},
\]
where we use the fixed isomorphism $\C \xrightarrow{\sim} \C_p$.
Since $G$ is abelian and $K/k$ is unramified at all $p$-adic primes, 
the conjecture ${\rm \bf{C}}_p (K/k)$ in \cite{BlBu03}, proved in \cite[Corollary 7.6]{BlBu03}, tells us that 
the determinant of the homomorphism $\pi_{K/k}$ with respect to the $\Z_p[G]$-free lattices $\Z_p \otimes \OO_K$ and $\bigoplus_{\iota_K \in \Sigma_K} \Z_p$ is
\[
\tau_{K/k} \cdot \prod_{v \in S_{\ram} (K/\Q)_f } f_v^{-1}
\]
up to $\Z_p [G]^\times$ (see \cite[the proof of Theorem 5.1]{BlBu02}). 
Here we note that we can replace the set $S_{\ram} (K / \Q)_f$ by $S_{\ram} (K/k)_f$ because for $v \in S_{\ram} (K / \Q)_f \setminus S_{\ram} (K/k)_f$, we have $f_v \in \Z_p[G]^\times$.
We take a basis $y_{\iota_{K, 1}}, \dots, y_{\iota_{K, r_k}}$ of $\bigoplus_{\iota_K \in \Sigma_K} \Z_p$, where $y_{\iota_{K, i}}$ is the element whose $\iota_{K, i}$-component is $1$ and the other components are zero.
Then for each basis $x_1, \dots, x_{r_k}$ of $\Z_p \otimes \OO_K$, the determinant of $\pi_{K/k}$ with respect to these bases is $B_{K/k}(x_1, \dots , x_{r_k})$.
Therefore, the desired basis $x_1, \dots, x_{r_k}$ indeed exists.

Let us show the final claim.
We first compute ($(-)^{\mathsf{T}}$ denotes the transpose matrix)
\begin{align*}
&\parenth{\sum_{\sigma \in G} 
(\iota_{K, i} (x_m^\sigma))\sigma^{-1}}_{1 \leq i,m \leq r_k}^{\mathsf{T}}
\cdot
\parenth{\sum_{\tau \in G} 
(\iota_{K, i} (x_m^\tau))\tau}_{1 \leq i,m \leq r_k}\\
& \qquad = \parenth{\sum_{i=1}^{r_k} \sum_{\sigma \in G} 
(\iota_{K, i} (x_a^\sigma))\sigma^{-1}
\cdot
\sum_{\tau \in G} 
(\iota_{K, i} (x_b^\tau))\tau}_{1 \leq a, b \leq r_k}\\
& \qquad = \parenth{\sum_{i=1}^{r_k} \sum_{\sigma \in G} \sum_{\tau \in G}
(\iota_{K, i} (x_a^\sigma x_b^\tau))\sigma^{-1} \tau}_{1 \leq a, b \leq r_k}\\
& \qquad = \parenth{\sum_{\sigma \in G} \sum_{i=1}^{r_k} \sum_{\tau \in G}
(\iota_{K, i} (x_a^{\sigma^{-1} \tau} x_b^\tau))\sigma}_{1 \leq a, b \leq r_k}\\
& \qquad = \parenth{\sum_{\sigma \in G} \Tr_{K/\Q}(x_a^{\sigma^{-1}} x_b)\sigma}_{1 \leq a, b \leq r_k}.
\end{align*}
Therefore, by the definition of $B_{K/k}(x_1, \dots, x_{r_k} )$, the left hand side of the claim is
\begin{align*}
& B_{K/k}(x_1, \dots, x_{r_k} ) \cdot B_{K/k}(x_1, \dots, x_{r_k} )^\# \\
& \qquad = \tau_{K/k} \cdot
\prod_{v \in S_{\ram} (K/k)_f} f_v^\#  \cdot 
\tau_{K/k}^{\#} \cdot
\prod_{v \in S_{\ram} (K/k)_f} f_v \\
& \qquad = \tau_{K/k} \cdot \tau_{K/k}^{\#}.
\end{align*}
By applying $\tau_k (\chi_v) \cdot \tau_k (\chi^{-1}_v) = \theta_{\chi_v} (-1)N(\ff (\chi_v))$ 
for each finite prime $v$ of $k$ at which $\chi$ ramifies (cf.~\cite[equation (5.7a) in Chapter 1, \S 5]{Fro83}), we continue to
\begin{align*}
& \qquad =
 |D_k| \cdot (-1)^{r_\C (k)} 
\sum_{\chi \in \widehat{G}} \left( \prod_{v \in S_{\ram} (K/k)_f} \theta_{\chi_{v}} (-1) N(\ff(\chi_v)) \right) e_\chi  \\
& \qquad =
 |D_k| \cdot (-1)^{r_\C (k)}
 \prod_{v \in S_{\ram} (K/k)_f} \rec_{k_v} (-1) 
\sum_{\chi \in \widehat{G}} 
N(\ff_\chi) e_\chi.
\end{align*}
This completes the proof of Proposition \ref{prop: det B and Gauss sum}. 
\end{proof}

\section{Proof of Theorem \ref{thm:lETNC}}\label{sec:pf_lETNC}

Our purpose in this section is to prove Theorem \ref{thm:lETNC}. 
Namely, we construct a system of bases of the determinant module of the 
local Galois cohomology complex by using the result of the previous sections.

We fix a number field $k$ such that $k / \Q$ is unramified at $p$. 
Let $K/k$ be a finite abelian extension in which all $p$-adic primes are unramified. 
For any $n \geq 0$, 
we put $K_n = K(\mu_{p^n})$ and $\GG_{K_n} = \Gal(K_n/k)$. 
We set $k_\infty := k (\mu_{p^\infty})$, $K_\infty := K (\mu_{p^\infty})$, $\GG_{K_\infty} := \Gal (K_\infty / k)$ and $\Lambda := \Z_p[[\GG_{K_\infty}]]$. 
We fix a basis $\xi = (\zeta_{p^n})_n$ of $\Z_p(1) = \varprojlim_n \mu_{p^n}(K_{\infty})$. 

\subsection{Auxiliary units}\label{ss:aux}

For any finite prime $v$ of $k$, 
let $\rec_{k_v} : k_v^\times \to \Gal (k_v (\mu_{p^\infty}) / k_v) $ be the local reciprocity map. 
By the natural isomorphism $\GG_{K_\infty} \simeq \GG_K \times \Gal (k_\infty / k)$, 
we regard $\rec_{k_v} : k_v^\times \to \Gal (k_\infty / k) \subset \GG_{K_\infty}$.

\begin{defn}\label{defi:D}
For any $j \in \Z$, 
we define 
\[
D_j(K_\infty / k) := |D_k|^j \prod_{v \in S_{\ram} (k/\Q)_f} \rec_{k_v} (\cD_{k_v / \Q_l})^{-1} \in \Lambda^\times. 
\]
We indeed have $D_j(K_\infty / k) \in  \Lambda^\times$ because $|D_k| \in \Z_p^\times$ by the assumption that $p$ is unramified in $k / \Q$.
\end{defn}

\begin{lem}\label{lem:cond unit}
Let $v$ be a non-$p$-adic finite prime of $k$.
Then, for any $j \in \Z$, 
there is a (unique) element $\ff_j (K_\infty / k)_v \in \Lambda^\times$ 
satisfying the following.
For any character $\chi \in \widehat{\GG_{K_\infty}}$,
we write $\chi = \chi_1 \chi_2$ with $\chi_1 \in \widehat{\GG_{K}}$ and 
$\chi_2 \in \widehat{\GG_{k_\infty}} = \widehat{\Gal (k_\infty / k)}$ as in Setting \ref{setting 1}.  
Then we have 
\[
\chi (\ff_j (K_\infty / k)_v) = 
\begin{cases}
N (v)^{j}\cdot \chi_2 (\sigma_{k_\infty, v})^{-1} &\text{if } \chi \text{ is unramified at } v, \\
N (\ff(\chi_v))^{j}\cdot \chi_2 (\sigma_{k_\infty, v})^{-i_\chi} &\text{otherwise}, 
\end{cases}
\]
where $\sigma_{k_\infty, v} \in \Gal (k_\infty / k)$ is the arithmetic Frobenius of $v$, 
$\ff(\chi_v)$ is the $v$-component of the conductor of $\chi$ and 
$i_\chi$ is defined by $N (\ff(\chi_v)) = N (v)^{i_\chi}$. 
\end{lem}

\begin{proof}
Let $I_v \subset \GG_{K_\infty}$ be the inertia group of $v$, which is finite since $v \nmid p$.
We consider the filtration 
\[
I_v^0 = I_v \supset I_v^1 \supset \cdots \supset I_v^r = \{ 1 \}
\]
such that for any character $\chi$ of $\GG_{K_\infty}$, 
$N (\ff(\chi_v)) = N (v)^{i}$ is equivalent to 
$\ker \chi \supset I_v^i $ and $\ker \chi \not\supset I_v^{i-1} $ 
(i.e, $I_v^i $ is the $i$-th ramification group in the upper numbering). 
We put $e_{v, i} := e_{I_v^i} = \frac{1}{\# I_v^i} \sum_{\sigma \in I_v^i} \sigma$. 
For any $j \in \Z$, we define $\ff_j (K_\infty / k)_v$ by 
\begin{align*}
\ff_j (K_\infty / k)_v := N (v)^{j} \cdot \sigma_{k_\infty, v}^{-1} \cdot e_{v, 0}
+ 
\sum_{i=1}^{r} \parenth{N (v)^j \cdot \sigma_{k_\infty, v}^{-1}}^i \cdot (e_{v, i} - e_{v, i-1})
\\
= N (v)^j \cdot \sigma_{k_\infty, v}^{-1} \cdot e_{v, 1} + 
\sum_{i=2}^{r} \parenth{N (v)^j \cdot \sigma_{k_\infty, v}^{-1}}^i \cdot (e_{v, i} - e_{v, i-1}).
\end{align*}
Then the interpolation property immediately follows.
It remains to show that $\ff_j (K_\infty / k)_v$ is in $\Lambda^{\times}$. 
Since $v$ is not a $p$-adic prime, we have $e_{v, i} \in \Lambda$ for any $1 \leq i \leq r$. 
Therefore $\ff_j (K_\infty / k)_v \in \Lambda$.
Similarly, the element 
\[
(\ff_j (K_\infty / k)_v)^{-1} = 
N (v)^{-j} \cdot \sigma_{k_\infty, v} \cdot e_{v, 0}
+ 
\sum_{i=1}^{r} \parenth{N (v)^{-j} \cdot \sigma_{k_\infty, v}}^{i} \cdot (e_{v, i} - e_{v, i-1}) 
\]
is also contained in $\Lambda$. 
Therefore, we have $\ff_j (K_\infty / k)_v \in \Lambda^\times$. 
\end{proof}

\subsection{The construction of the basis}\label{ss:constr_basis}

In this subsection, we construct the desired $\Lambda$-basis $Z_{K_{\infty}/k, S}^{\loc, j}$ of $\Xi^{\loc}_{K_{\infty}/k, S}(j)$.
Its properties will be shown in the next subsection.

\subsubsection{Archimedean places}

Firstly, let us take a basis of $X_{K_{\infty}}(j)$.

As in \S \ref{ss:alpha} and \S \ref{ss:choice}, we fix a labeling $\Sigma_k = \{\iota_{k, 1} , \dots , \iota_{k, r_k} \}$
and then, for each $1 \leq i \leq r_k $, 
we choose $
\iota_{K_{\infty}, i}: K_{\infty} \hookrightarrow \C
$
such that $\iota_{K_{\infty}, i}|_k = \iota_{k, i}$ and
$
\iota_{K_{\infty}, i} (\zeta_{p^n}) = \exp \left(\frac{2 \pi \sqrt{-1}}{p^n} \right)
$
for all $n \geq 0$. 
Then for each $n \geq 0$, we set $\iota_{K_n, i} = \iota_{K_{\infty, i}}|_{K_n}$.

Let $j \in \Z_{\geq 1}$. Recall that we take a basis $\{e_{\iota_{K_n , i}}^j\}_{1\leq i \leq r_k}$ of $X_{K_n} (j)$ in \S \ref{period map}. We define a $\Lambda$-basis 
$\{e_{K_{\infty}, i}^j \mid 1 \leq i \leq r_k\}$ of $X_{K_\infty} (j)$ as $e_{K_{\infty}, i}^j := (e_{\iota_{K_n , i}}^j)_{n \geq 0} \in X_{K_\infty} (j)$. Therefore, we have a basis
\[
(\wedge_{i=1}^{r_k} e_{K_\infty, i}^j )^\ast
\in
\Det_{\Lambda}^{-1} (X_{K_\infty} (j)),  
\]
where $(-)^\ast$ means the dual basis.  

Note that the choice of $\iota_{K, i}$ also determines a $\Q[G]$-basis $\{ b_{\iota_{K, i}}^j \mid 1 \leq i \leq r_k \}$
of the Betti cohomology space $H_K (j)$.

\subsubsection{Non-$p$-adic primes}

For non-$p$-adic finite primes $v$, by 
Theorem \ref{thm:det non p prime}, we have a basis
\[
\cH_{K_{\infty}/k, v}^j \in \Det_{\Lambda}^{-1} \parenth{\Delta_{K_{\infty}, v} (j) }.
\]
In addition, for non-$p$-adic finite prime $v$ which is unramified in $K/k$, 
by Corollary \ref{cor:unr_basis}, we also have a basis
\[
\cE_{K_{\infty}/k, v}^j \in \Det_{\Lambda}^{-1} \parenth{\Delta_{K_{\infty}, v} (j) }.
\] 

\subsubsection{$p$-adic primes}\label{sss:p}

We set
$
\OO_{K_p} := \Z_p \otimes_{\Z} \OO_{K}
\simeq \prod_{v \in S_p} \OO_{K_v}
$
and
\[
R_{K_p} := \left\{ f(T) \in \OO_{K_p} [[T]]  \, \middle| \, \sum_{\zeta \in \mu_p} f (\zeta (1+T)-1) =0 \right\} \subset \OO_{K_p}[[T]].
\]
Then we have $R_{K_p} = \bigoplus_{v \in S_p} R_{K_v}$.

For each $v \in S_p$, we consider the isomorphism 
\[
\Phi_{K_{\infty}, v}^j : 
\Det_{\Lambda}^{-1} \parenth{
\Delta_{K_{\infty}, v} (j) }
\xrightarrow{\sim} 
\Det_{\Lambda} (R_{K_v})
\]
in Theorem \ref{thm:Phi}. 
Note that this depends on the choice of a basis of $\Z_p (1)$, and we use the one at the beginning of this section. 
Put $\Delta_{K_{\infty}, p} (j) := \bigoplus_{v \in S_p} \Delta_{K_{\infty}, v} (j) $. 
Then by taking the tensor product for $v \in S_p$, 
the maps $\Phi_{K_{\infty}, v}^j $ induce an isomorphism  
\[
\Phi_{K_{\infty}, p}^j : 
\Det_{\Lambda}^{-1} \parenth{
\Delta_{K_\infty, p} (j) }
\xrightarrow{\sim} 
\Det_{\Lambda} (R_{K_p}).
\]

We take an ordered $\Z_p[\GG_K]$-basis $x_{K, 1}, \dots , x_{K, r_k}$ of $ \OO_{K_p}$, which is denoted by $\xx$,
as in Proposition \ref{prop: det B and Gauss sum}
for the embeddings $\{ \iota_{K, i} \mid 1 \leq i \leq r_k\} \subset \Sigma_{K}$. 
Then, by \cite[Lemme 1.5]{Per90}, 
we know that $\{ x_{K, i} (1+T) \mid 1 \leq i \leq r_k \}$ is a basis of $R_{K_p}$ as a $\Lambda$-module.
We take  
\[
\wedge_{i=1}^{r_k} x_{K,i} (1+T) \in \Det_{\Lambda} (R_{K_p})
\]
as a $\Lambda$-basis of $\Det_{\Lambda} (R_{K_p})$, and then by using the isomorphism $\Phi_{K_{\infty}, p}^j$, we obtain a $\Lambda$-basis
\[
Z^{p, j}_{K_{\infty}/k, \xx} := \left(\Phi_{K_{\infty}, p}^j\right)^{-1} \parenth{ \wedge_{i=1}^{r_k} x_{K,i} (1+T)}
\]
of $\Det_{\Lambda}^{-1} \parenth{ \Delta_{K_{\infty}, p} (j) }$.

\subsubsection{Putting all together}

Now we are ready to construct the desired basis of $\Xi^{\loc,j}_{K_{\infty}/k, S}$. 
Recall that $S$ is a finite set of places of $k$ such that 
$S \supset S_p \cup S_{\ram} (K/k) \cup S_\infty$. 
We set $S_f := S \setminus S_\infty$ and 
$S_{\ram} (K/k)_f := S_{\ram} (K/k) \cap S_f$
(note that $S_p \cap S_{\ram} (K/k)_f = \emptyset$ since $p$ is unramified in $K/ \Q$).

\begin{defn}\label{def:Z}
For any 
$j \in \Z_{\geq 1}$,
we define a $\Lambda$-basis $Z_{K_\infty /k, S}^{\loc, j}$ of
\[
\Xi_{K_\infty / k, S}^{\loc} (j) 
= \Det_{\Lambda}^{-1} \parenth{\Delta_{K_\infty, p} (j)} 
\otimes
\parenth{\bigotimes_{v \in S_f \setminus S_p} \Det_{\Lambda}^{-1} (\Delta_{K_{\infty}, v} (j)) }
\otimes
\Det_{\Lambda}^{-1} (X_{K_{\infty}} (j)) 
\]
by
\begin{align*}
Z_{K_\infty /k, S}^{\loc, j}
&:= D_{j-1}(K_\infty / k) \prod_{v \in S_{\ram} (K/k)_f} \ff_{j-1} (K_\infty / k)_v  
\\
&\times
  Z_{K_\infty /k, \xx}^{p, j} \otimes \left(\bigotimes_{v \in S_{\ram}(K/k)_f} \cH_{K_{\infty}/k, v}^{ j}
  \otimes \bigotimes_{S_f \setminus (S_{\ram}(K/k)_f \cup S_p)} \cE_{K_{\infty}/k , v}^{ j}\right)
\otimes (\wedge_{i=1}^{r_k} e_{K_\infty, i}^j )^\ast. 
\end{align*}
\end{defn}

\subsection{Proof of Theorem \ref{thm:lETNC}}

We shall check that this basis satisfies the properties in Theorem \ref{thm:lETNC}. 
Thanks to Proposition \ref{prop:localETNC_func}, we only have to consider the intermediate fields $K_n$.
Namely, we only have to compute the image of $Z_{K_\infty /k, S}^{\loc, j}$ under the composite map 
\[
\Xi_{K_\infty / k, S}^{\loc} (j) \twoheadrightarrow \Xi_{K_n/k, S}^{\loc} (j) 
\hookrightarrow \C_p \otimes_{\Z_p} \Xi_{K_n/k, S}^{\loc} (j)
\overset{\vartheta_{K_n / k, S}^{\loc,j}}{\xrightarrow{\sim}} \C_p[\GG_{K_n}] \xrightarrow{\chi} \C_p
\]
for each $j \in \Z_{\geq 1}$, $n \in \Z_{\geq 0}$ and $\chi \in \widehat{\GG_{K_n}}$.

We decompose $\chi$ as $\chi =  \chi_1 \chi_2$ with characters $\chi_1 \in \widehat{\GG_{K}}$ and 
$\chi_2 \in \widehat{\GG_{k_n}} = \widehat{\Gal(k_n / k)}$ as in Setting \ref{setting 1}. 
For any non $p$-adic finite prime $v$ of $k$ lying above a rational prime $l$, 
let $I_{K, v} \subset \GG_K ( \subset \GG_{K_n} )$ be the inertia group of $v$ and 
recall that $\epsilon_v^0 (K/k) \in \Q[\GG_K]$ is defined in \S \ref{local eTNC}. 
Also, $\ff (\chi_v)$ is the $v$-component of the conductor of $\chi$,  
$\cD_{k_v/ \Q_l}$ is the different ideal of $k_v / \Q_l$, and 
$\rec_{k_v} : k_v^\times \twoheadrightarrow \Gal (k_v (\mu_{p^n}) / k_v) \subset \GG_{k_n}$ is the local reciprocity map. 

We first compute the contribution of the $p$-adic and archimedean components. 
Set $K_{n, p} := \Q_p \otimes_{\Q} K_n \simeq \prod_{w_n \mid p} K_{n ,w_n}$. 
We consider the composite map 
\begin{equation}\label{p part descent map}
\Det_{\Lambda}^{-1} \parenth{
\Delta_{K_{\infty}, p} (j) }
\twoheadrightarrow
\Det_{\Z_p [\GG_{K_n}]}^{-1} (\bigoplus_{v \in S_p} \Delta_{K_n, v} (j)) 
\overset{\otimes_v \phi_{K_v}^j}{\hookrightarrow}
\bigwedge_{\Q_p[\GG_{K_n}]}^{r_k} K_{n, p}, 
\end{equation}
where the second map is induced by $\phi_{K_v}^j$, which we introduced in \S \ref{ss:setup}.

\begin{prop}
The composite map 
\begin{align*}
\Det_{\Lambda}^{-1} \parenth{\Delta_{K_{\infty}, p} (j)} 
\otimes
\Det_{\Lambda}^{-1} (X_{K_{\infty}} (j)) 
& \xrightarrow{\eqref{p part descent map}\text{\upshape{ and }} \eqref{Y Betti}}
\bigwedge_{\Q_p[\GG_{K_n}]}^{r_k} K_{n, p}
\otimes
\Det_{\Q_p[\GG_{K_n}]}^{-1} (\Q_p \otimes_\Q H_{K_n} (j)) \\
& \overset{\wedge \alpha_{K_n}^j}{\hookrightarrow} \C_p[\GG_{K_n}]
\overset{\chi}{\to} \C_p
\end{align*}
sends $Z_{K_\infty /k, \xx}^{p, j} \otimes ( \wedge_{i=1}^{r_k} e_{K_\infty, i}^j )^\ast$ to
\begin{align*}
& \delta_{\chi={\trivial} , j=1} \cdot (-1)^{r_k(j-1)} \cdot |D_k|^{1-j} \cdot N (\ff_\chi')^{1-j}  
 \prod_{v \in S_{\ram} (K/\Q)_f} \chi_2 (\rec_{k_v} (\ff(\chi_v) \cD_{k_v/ \Q_l})) \\
&\times
\prod_{v \in S_{\ram} (K/k)_f} 
\parenth{\epsilon_v^{0} (K/k)^\# - e_{I_{K, v}}}^{\chi_1} 
\times
\begin{cases}\displaystyle
\cfrac{L_{S_p}^\ast (\chi^{-1}, 1-j)}{L_{S_p}^\ast (\chi,j)}
& \text{if } j\neq 1,  \\
\displaystyle
\prod_{v \in S_p} 
(\delta_v(K_n/k)^\#)^{\chi} 
\cfrac{L^\ast (\chi^{-1}, 1-j)}{L_{S_p}^\ast (\chi,j)}
 & \text{if }  j=1,     
\end{cases}
\end{align*}
where $\ff_\chi'$ denotes the non-$p$-adic component of the conductor ideal $\ff_\chi$ of $\chi$.  
\end{prop}

\begin{proof}
By Corollary \ref{cor:main thm1}, 
the map $\Det_{\Lambda}^{-1} \parenth{
\Delta_{K_{\infty}, p} (j) } \xrightarrow{\eqref{p part descent map}}
\bigwedge_{\Q_p[\GG_{K_n}]}^{r_k} K_{n, p}
\overset{\chi}{\to} \parenth{ \bigwedge_{\Q_p[\GG_{K_n}]}^{r_k} K_{n, p}}^\chi
$
sends 
$Z_{K_\infty /k, \xx}^{p, j}$
 to 
\begin{align*}
&(-1)^{r_k (j-1)} \cdot  ((j-1)!)^{r_k}  
\\
&\times  
\begin{cases}\displaystyle
(-1)^{n_\chi (r_k - \# S_p )} \cdot \prod_{v \in S_p } N (v)^{j n_\chi -1}
\chi(\sigma_{K_n, v})^{-n_\chi} \cdot 
(\wedge_{i=1}^{r_k} a_n x_{K,i})^\chi & \text{if } n_\chi > 0, \\
\displaystyle
(-1)^{r_k} \cdot \prod_{v \in S_p } 
\parenth{\cfrac{\epsilon_v^{1-j}(K_n/k)^{\#}}{ \epsilon_v^{j}(K_n/k) }}^\chi 
\cdot 
(\wedge_{i=1}^{r_k} a_n x_{K,i})^\chi & \text{if } n_\chi = 0 \text{ and } j\neq 1, \\
\displaystyle
(-1)^{r_k}  \cdot
\prod_{v \in S_p } 
\parenth{\cfrac{\delta_v(K_n/k)^{\#}}{\epsilon_v^{1}(K_n/k)}}^\chi
 \cdot 
(\wedge_{i=1}^{r_k} a_n x_{K,i})^\chi & \text{if }n_\chi = 0 \text{ and }  j=1.
\end{cases}
\end{align*}
Here $n_\chi$ and $\sigma_{K_n, v}$ are defined just before Setting \ref{setting 1};
$\epsilon_v^j(K_n/k)$, $\epsilon_v^{1-j}(K_n/k)$, and  $\delta_v(K_n/k)$ are defined in \S \ref{local eTNC}; and 
$a_n \in K_n (\subset K_{n, p})$ is given by
\[
a_n := 
\begin{cases}
\frac{1}{p^{n-1}} (\zeta_p + \zeta_{p^2} + \cdots + \zeta_{p^n}) 
& \text{if } n \geq 1, \\
-1 & \text{if } n=0.   
\end{cases}
\]
More precisely, Corollary \ref{cor:main thm1} directly implies this formula when the basis $\xx$ of $\OO_{K_p}$ is chosen as the combination of bases of $\OO_{K_v}$, and the general case follows from observing what happens when changing the bases.

Then by combining Lemma \ref{wedge alpha}, we see that the image to be computed is
\begin{align}\label{local period map}
& (-1)^{r_k (j-1)} \cdot   
 B_{K_n/k}(a_n x_{K,1}, \dots, a_n x_{K, r_k})^\chi
 \cdot \cfrac{((j-1)!)^{r_k}}{(2 \pi \sqrt{-1})^{jr_k}} 
 \\
& \times 
A(k, \chi, j)  \cdot 
\begin{cases}\displaystyle
(-1)^{n_\chi (r_k - \# S_p )} \cdot \prod_{v \in S_p } N (v)^{j n_\chi -1}
\chi(\sigma_{K_n, v})^{-n_\chi} & \text{if } n_\chi > 0,  \nonumber \\
\displaystyle
(-1)^{r_k} \cdot \prod_{v \in S_p } 
\cfrac{\epsilon_v^{1-j}(K_n/k)^{\chi^{-1}}}{ \epsilon_v^{j}(K_n/k)^\chi } 
& \text{if } n_\chi = 0 \text{ and } j\neq 1, \nonumber \\
\displaystyle
(-1)^{r_k }  \cdot
\prod_{v \in S_p } 
\cfrac{\delta_v(K_n/k)^{\chi^{-1}}}{ \epsilon_v^{1}(K_n/k)^\chi }
 & \text{if }n_\chi = 0 \text{ and }  j=1. \nonumber 
\end{cases}
\end{align}

We first consider the case where $\chi_2$ is trivial (i.e. $n_\chi  = 0$).  
By Proposition \ref{prop: det B and Gauss sum} and Lemma \ref{B Gauss sum}, 
we have 
\begin{align*}
&B_{K_n/k}(a_n x_{K,1}, \dots, a_n x_{K, r_k})^\chi 
=(-1)^{r_k} \cdot B_{K/k}(x_{K,1}, \dots, x_{K, r_k})^{\chi_1} \\
&\quad=(-1)^{r_k} \cdot (\tau_{K/k})^{\chi_1}
\prod_{v \in S_{\ram} (K/k)_f} 
\parenth{\epsilon_v^0 (K/k)^\# - e_{I_{K, v}}}^{\chi_1} \\
&\quad=(-1)^{r_k} \cdot 
|D_k|^{\frac{1}{2}} \cdot (\sqrt{-1})^{r_\C (k)} \cdot \tau_k (\chi_1)
\prod_{v \in S_{\ram} (K/k)_f} 
\parenth{\epsilon_v^0 (K/k)^\# - e_{I_{K, v}}}^{\chi_1} . 
\end{align*}
Since $\chi_2$ is trivial, 
we know that $\tau_k (\chi_1) = \tau_k (\chi)$ by Lemma \ref{lem:Davenport--Hasse} (i). 
Therefore, we have 
\begin{align*}
\eqref{local period map} 
& =  A(k, \chi, j) \cdot 
 \cfrac{((j-1)!)^{r_k} \cdot (-1)^{jr_k}}
{(2 \pi \sqrt{-1})^{jr_k}} 
\cdot |D_k|^{\frac{1}{2}} \cdot (\sqrt{-1})^{r_\C (k)} \cdot \tau_k (\chi) \\
& \times 
\prod_{v \in S_{\ram} (K/k)_f} \parenth{\epsilon_v^0 (K/k)^\# - e_{I_{K, v}}}^{\chi_1}
\cdot
\begin{cases}\displaystyle
(-1)^{r_k} \cdot \prod_{v \in S_p } 
\cfrac{\epsilon_v^{1-j}(K_n/k)^{\chi^{-1}}}{ \epsilon_v^{j}(K_n/k)^\chi } 
& \text{if } j\neq 1,  \\
\displaystyle
(-1)^{r_k }  \cdot
\prod_{v \in S_p } 
\cfrac{\delta_v(K_n/k)^{\chi^{-1}}}{ \epsilon_v^{1}(K_n/k)^\chi }
 & \text{if } j=1.  
\end{cases}
\end{align*}
In this case we have $\ff_\chi=\ff'_\chi$, and the claim follows from Theorem \ref{thm:func eq}.

Let us consider the case where $\chi_2$ is non-trivial (i.e. $n_\chi  > 0$). 
Then, by Proposition \ref{prop: det B and Gauss sum} and Lemma \ref{B Gauss sum}, 
we have 
\begin{align*}
&B_{K_n/k}(a_n x_{K,1}, \dots, a_n x_{K, r_k})^\chi =
(-1)^{n_\chi (r_k - \# S_p )} \\
& \times 
\parenth{ \prod_{v \in S_p } 
\chi(\sigma_{K_n, v})^{n_\chi} \cdot N (v)^{1-n_\chi} \cdot \tau_{k} (\chi_v) }
\cdot
(\tau_{K/k})^{\chi_1}
\cdot \prod_{v \in S_{\ram} (K/k)_f} \parenth{\epsilon_v^0 (K/k)^\# - e_{I_{K, v}}}^{\chi_1}. 
\end{align*}
Since 
\[
\tau_k(\chi_1)
\prod_{v \in S_p } \tau_{k} (\chi_v) 
= \tau_k (\chi) \prod_{v \in S_{\ram} (K/\Q)_f} \chi_2 (\rec_{k_v} (\ff(\chi_v)  \cD_{k_v/ \Q_l})) 
\]
by Lemma \ref{lem:Davenport--Hasse} (i), we have 
\begin{align*}
&B_{K_n/k}(a_n x_{K,1}, \dots, a_n x_{K, r_k})^\chi =
(-1)^{n_\chi (r_k - \# S_p )} \cdot |D_k|^{^\frac{1}{2}} \cdot (\sqrt{-1})^{r_\C (k)} \cdot \tau_k (\chi)
\\
& \times 
 \prod_{v \in S_p } 
\chi(\sigma_{K_n, v})^{n_\chi} \cdot N (v)^{1-n_\chi}
\prod_{v \in S_{\ram} (K/\Q)_f} \chi_2 (\rec_{k_v} (\ff(\chi_v)  \cD_{k_v/ \Q_l})) 
 \prod_{v \in S_{\ram} (K/k)_f} \parenth{\epsilon_v^0 (K/k)^\# - e_{I_{K, v}}}^{\chi_1}. 
\end{align*}
We then obtain
\begin{align*}
\eqref{local period map} 
& =  (-1)^{r_k (j-1)} \cdot  A(k, \chi, j) \cdot 
 \cfrac{((j-1)!)^{r_k} }
{(2 \pi \sqrt{-1})^{jr_k}} 
\cdot |D_k|^{\frac{1}{2}} \cdot (\sqrt{-1})^{r_\C (k)} \cdot \tau_k (\chi) \\
& \times 
 \prod_{v \in S_p } N (v)^{(j-1) n_\chi} 
\cdot
\prod_{v \in S_{\ram} (K/\Q)_f} \chi_2 (\rec_{k_v} (\ff(\chi_v)  \cD_{k_v/ \Q_l})) 
\cdot \prod_{v \in S_{\ram} (K/k)_f} \parenth{\epsilon_v^0 (K/k)^\# - e_{I_{K, v}}}^{\chi_1}. 
\end{align*}
Also, we have $\prod_{v \in S_p }N (v)^{ (j-1)n_\chi} =  N (\ff_\chi')^{1-j} \cdot  N (\ff_\chi)^{j-1}$. 
Now the claim follows from Theorem \ref{thm:func eq}.
\end{proof}

Now we consider the contributions of the non-$p$-adic primes and the auxiliary factors.
By Lemma \ref{lem:cond unit} and the definition of $D_{j-1}(K_\infty / k)$, 
we have 
\begin{align*}
&\parenth{ D_{j-1}(K_\infty / k) \cdot \prod_{v \in S_{\ram} (K/k)_f} \ff_{j-1} (K_\infty / k)_v}^\chi 
\\
&=|D_k|^{j-1} \cdot N (\ff_\chi ')^{j-1}
\prod_{v \in S_{\ram} (K/\Q)_f} \chi_2 (\rec_{k_v} (\ff(\chi_v) \cD_{k_v / \Q_l}))^{-1} \\
&\times \prod_{v \in S_{\ram} (K/k)_f} (1-e_{I_{K, v}} +e_{I_{K, v}} N(v)^{j-1} \sigma_{k_\infty, v}^{-1} )^{\chi}.
\end{align*}
The contribution of the non-$p$-adic primes is determined by Theorem \ref{thm:det non p prime} 
and Corollary \ref{cor:unr_basis}.
For each $v \in S_{\ram} (K/k)_f$, we have
\[
\parenth{\epsilon_v^0 (K/k)^\# - e_{I_{K, v}}}^{\chi_1}
\parenth{1-e_{I_{K, v}} +e_{I_{K, v}} N(v)^{j-1} \sigma_{k_\infty, v}^{-1} }^\chi 
=
\parenth{\epsilon_v^{1-j} (K_n/k)^\# - e_{I_{K, v}}}^{\chi}.
\]
Moreover, we have for $j \geq 2$
\[
\Big{(}\epsilon_v^{1-j} (K_n/k)^\# - e_{I_{K, v}}\Big{)}
 \epsilon^{j-1}_v (K_n / k) = \epsilon^{1-j}_v (K_n / k)^\#
\]
and for $j = 1$
\[
\Big{(}\epsilon_v^{0} (K_n/k)^\# - e_{I_{K, v}} \Big{)} 
\cdot 
\Big{(}\epsilon^{0}_v (K_n / k) - e_{\GG_{K_n, v}} \cfrac{1}{\# (\GG_{K_n, v} / I_{K_n, v})}\Big{)}
= \delta_v(K_n/k)^\#. 
\]
By using these formulas, we obtain the formula in Theorem \ref{thm:lETNC}. 
\qed

\begin{rem}
We have the compatibility when $K$ and $S$ vary as follows.

Let $K'$ be an intermediate field of $K/k$. 
Then the image of the basis $Z_{K_\infty / k, S}^{\loc, j}$ by the natural map 
\[
\Xi_{K_\infty / k, S}^{\loc} (j) \twoheadrightarrow \Z_p[[\Gal (K_\infty' / k)]] \otimes_{\Lambda} \Xi_{K_\infty / k, S}^{\loc} (j) \xrightarrow{\sim} \Xi_{K'_\infty / k, S}^{\loc} (j)
\]
coincides with $Z_{K'_\infty / k, S}^{\loc, j}$. 
This immediately follows from Proposition \ref{prop:localETNC_func}. 

Let $S'$ be a finite set of places of $k$ which contains $S$. 
Then we have 
\[
Z_{K_\infty / k, S'}^{\loc, j} = Z_{K_\infty / k, S}^{\loc, j} \otimes \bigotimes_{v \in S' \setminus S } \cE_{K_\infty, v}^j
\in \Xi_{K_\infty / k, S'}^{\loc} (j) = \Xi_{K_\infty / k, S}^{\loc} (j) \otimes_{\Lambda} 
\bigotimes_{v \in S' \setminus S } \Det_{\Lambda}^{-1} (\Delta_{K_\infty, v} (j)).
\] 
 This immediately follows from a direct computation using Corollary \ref{cor:unr_basis}. 
\end{rem}

\section{Twists}\label{App:twist}
 
\subsection{Definition of the twists in non-archimedean cases}

We review the twist maps between the Galois cohomology complexes. 
Although this kind of material is known to experts, we have not found comprehensive references, and we summarize here for the convenience of readers. 

We fix an odd prime number $p$. 
Let $K/k$ be a finite abelian extension of number fields. We consider the Iwasawa algebra $\Lambda = \Z_p[[\Gal(K_{\infty}/k)]]$ as in \S \ref{ss:Iw_limit}.
For any integer $m$, we write 
\[
\twist_m : \Lambda \xrightarrow{\sim} \Lambda ; 
\quad
\sigma \mapsto \chi_{\cyc }^m (\sigma)\sigma
\]
 for the ring automorphism induced by the 
$m$-th power of the cyclotomic character. 

\begin{defn}
Let $\xi = (\zeta_{p^n})_n$ be a basis of $\Z_p(1) = \varprojlim_n \mu_{p^n}(K_{\infty})$. 
For any $j, j' \in \Z$ and any finite prime $v$ of $k$, we define the $(j'-j)$-th twist map
\[
\Twist_{\Delta_{K_{\infty}, v}, j, j'}^{\xi}: \Delta_{K_\infty , v} (j) \xrightarrow{\sim} \Delta_{K_\infty, v} (j')
\]
as follows.
For each $n \geq 0$ and each prime $w_n$ of $K_n$ lying above $v$, 
we have an isomorphism
\[
\RG (K_{n, w_n} , \Z / p^n \Z (j))
\overset{\cup \zeta_{p^n}^{\otimes j'-j}}{\xrightarrow{\sim}}
\RG (K_{n, w_n} , \Z / p^n \Z (j'))
\]
defined by the cup product with $\zeta_{p^n}^{\otimes j'-j} \in \mu_{p^n} (K_{n, w_n})^{\otimes j'-j}$, 
where we regard $\zeta_{p^n} \in K_{n, w_n}$ by the natural inclusion $K_n \hookrightarrow K_{n, w_n}$. 
Since we have an isomorphism
\[
\Delta_{K_\infty , v} (j) \simeq \varprojlim_n \left(\bigoplus_{w_n \mid v} \RG (K_{n, w_n} , \Z / p^n \Z (j)) \right),
\]
we can define the map $\Twist_{\Delta_{K_{\infty}, v}, j, j'}^{\xi}$ by taking the projective limit of $\cup \zeta_{p^n}^{\otimes j'-j}$ with respect to $n$.
We can check that this is a $\twist_{j-j'}$-semilinear isomorphism over $\Lambda$.
\end{defn}

To induce twist maps between the determinant modules, let us establish an elementary lemma.

 \begin{lem}\label{lem:det_twist}
Let $P_1, P_2$ be finitely generated projective $\Lambda$-modules and
$f : P_1 \xrightarrow{\sim} P_2 $ a $\twist_m$-semilinear isomorphism for some $m \in \Z$. 
Then $f$ induces $\twist_m$-semilinear isomorphisms
\[
\Det_{\Lambda} (P_1) \xrightarrow{\sim} \Det_{\Lambda} (P_2),
\quad
\Det_{\Lambda}^{-1} (P_1) \xrightarrow{\sim} \Det_{\Lambda}^{-1} (P_2).
\] 
Therefore, a $\twist_m$-semilinear quasi-isomorphism between perfect complexes over $\Lambda$ induces a $\twist_m$-semilinear isomorphism between the determinant modules.
\end{lem}

\begin{proof}
The first one is simply the determinant of $f$. 
For the second one, we first consider a $\twist_{-m}$-semilinear isomorphism 
\[
\Hom_{\Lambda} (P_2, \Lambda) \xrightarrow{\sim} 
\Hom_{\Lambda} (P_1, \Lambda), \;\; g \mapsto \twist_{-m} \circ  g \circ f. 
\]
The inverse of this map is a $\twist_{m}$-semilinear isomorphism, the determinant of which induces the desired map.
\end{proof}

By applying Lemma \ref{lem:det_twist} to $\Twist_{\Delta_{K_{\infty}, v}, j, j'}^{\xi}$,
we obtain a $\twist_{j-j'}$-semilinear isomorphism
\begin{equation}\label{Det_Twist_map}
\Twist_{v, j, j'}^{\xi} : \Det_{\Lambda} (\Delta_{K_\infty, v} (j)) \xrightarrow{\sim} 
\Det_{\Lambda} (\Delta_{K_\infty, v} (j')).
\end{equation}
This map depends on the choice of the basis $\xi$ of $\Z_p(1)$, and the dependency is described by the next lemma.
We write $\ch (\Delta_{K_\infty, v})$ for the Euler characteristic of $\Delta_{K_\infty, v} (j)$ as a perfect complex over $\Lambda$, which is indeed independent of the integer $j$. 

\begin{lem}\label{twist_depend_xi}
Let $\xi' = (\zeta'_{p^n})_n$ be another basis of $\Z_p (1)$ and consider the map 
$\Twist_{v, j, j'}^{\xi'}$. We take  $a \in \Z_p^\times$ such that $\xi' = \xi^a$.  
Then we have
\[
\Twist_{v, j, j'}^{\xi'} = a^{\ch (\Delta_{K_\infty, v}) \cdot (j' - j)} \cdot \Twist_{v, j, j'}^{\xi}.
\]
In particular, if $\ch (\Delta_{K_\infty, v}) = 0$, 
then the map $\Twist_{v, j, j'}^{\xi}$ does not depend on the choice of $\xi$. 
\end{lem}

\begin{proof}
By the choice of $a$ and the construction show
\[
\Twist_{\Delta_{K_{\infty}, v}, j, j'}^{\xi'} = a^{j'-j} \cdot \Twist_{\Delta_{K_{\infty}, v}, j, j'}^{\xi}
\]
as maps between complexes.
Then we obtain the lemma, taking the determinants.
\end{proof}

\subsection{Definition of the twists in archimedean cases}

\begin{defn}
Let $\xi = (\zeta_{p^n})_n$ be a basis of $\Z_p(1) = \varprojlim_n \mu_{p^n}(K_{\infty})$.
For any $j, j' \in \Z$, we define the twist map
\[
\Twist_{X_{K_{\infty}} , j , j'}^{\xi} : X_{K_\infty} (j)
\xrightarrow{\sim} 
X_{K_\infty} (j')
\]
as follows.
For each $n \geq 0$, we have an isomorphism
\[
\bigoplus_{\iota : K_n \hookrightarrow \C} \Z / p^n \Z (j)
\xrightarrow{\sim} 
\bigoplus_{\iota : K_n \hookrightarrow \C} \Z / p^n \Z (j')
\]
defined by
\[
((a_{n, \iota})_{\iota : K_n \hookrightarrow \C} )_n \mapsto
 ((a_{n, \iota} \otimes \iota (\zeta_{p^n})^{\otimes j'-j})_{\iota : K_n \hookrightarrow \C} )_n.
 \]
 We can define the map $\Twist_{X_{K_{\infty}} , j , j'}^{\xi}$ by taking the projective limit of this isomorphism for $n$, since we have an isomorphism
\[
X_{K_{\infty}}(j) \simeq \varprojlim_n \left( \bigoplus_{\iota : K_n \hookrightarrow \C} \Z / p^n \Z (j) \right).
\]
 \end{defn}
 
Let us check that this is $\twist_{j-j'}$-semilinear:
For any $\sigma \in \Gal (K_\infty / k)$, we have  
 \begin{align*}
&\Twist_{X_{K_{\infty}} , j , j'}^{\xi} \parenth{ \sigma ((a_{n, \iota})_{\iota : K_n \hookrightarrow \C} )_n}
 = 
\Twist_{X_{K_{\infty}} , j , j'}^{\xi} \parenth{((a_{n, \iota})_{\iota \circ \sigma^{-1} \mid_{K_n} : K_n \hookrightarrow \C} )_n} \\
&\quad= 
\parenth{(a_{n, \iota} \otimes \iota (\sigma^{-1} (\zeta_{p^n}))^{\otimes j'-j} )_{\iota \circ \sigma^{-1} \mid_{K_n} : K_n \hookrightarrow \C} }_n 
=\chi_{\cyc}^{j-j'}(\sigma) \sigma  \parenth{ (( a_{n, \iota} \otimes \iota (\zeta_{p^n})^{\otimes j'-j} )_{\iota : K_n \hookrightarrow \C} )_n}.  
 \end{align*}
Therefore, by Lemma \ref{lem:det_twist}, this induces a $\twist_{j-j'}$-semilinear isomorphism over $\Lambda$ between the determinant modules 
\begin{equation}\label{shift X}
\Det_{\Lambda}^{-1} (X_{K_\infty} (j))
\xrightarrow{\sim}
\Det_{\Lambda}^{-1} (X_{K_\infty} (j')). 
\end{equation}

We can also describe the dependency of this isomorphism on the choice of $\xi$ in a similar way as in Lemma \ref{twist_depend_xi}. 
In fact, let $\xi'$ be another basis of $\Z_p(1)$ and take $a\in \Z_p^{\times}$ such that $\xi'=\xi^a$.
 Then, since $X_{K_\infty} (j)$ is a free $\Lambda$-module of rank $r_k := [k : \Q]$, the above isomorphism between the determinant modules for $\xi'$ changes by $a^{-r_k}$ from that for $\xi$.

\subsection{Twists of the bases}\label{ss:twist_basis}

Let $S$ be a finite set of places of $k$ which contains $S_\infty \cup S_p$.  
Then, for any $j \text{ and } j' \in \Z_{\geq 1}$,
we define a $\twist_{j-j'}$-semilinear isomorphism 
\[
\Twist_{j, j'}^{\loc} :
\Xi_{K_\infty / k, S}^{\loc} (j) \xrightarrow{\sim}
\Xi_{K_\infty / k, S}^{\loc} (j')
\]
as the tensor product of \eqref{shift X} and 
$\Twist_{v, j, j'}^{\xi}$
for each finite prime $v$ of $k$. 
This $\Twist_{j, j'}^{\loc}$ is independent of the choice of $\xi$ since the Euler characteristics of $\oplus_{v \in S_f} \Delta_{K_\infty , v} (j)$ and $X_{K_{\infty}}(j)[0]$ are both equal to $-r_k$, hence the dependencies cancel.

\begin{prop}\label{prop:extra_p}
Assume that $K /\Q$ is unramified at $p$. 
Let $v$ be a finite prime of $k$ lying above $p$. 
Then for any $j, j' \in \Z$, we have a commutative diagram
\[
\xymatrix{
\Det_{\Lambda}^{-1} (\Delta_{K_\infty, v} (j))  
\ar[r]^-{\Phi_{K_\infty, v}^j}_-{\simeq} \ar[d]^-{\Twist_{v, j, j'}^{\xi}}_-{\simeq}
&\Det_{\Lambda} (R_{K_v})   \ar[d]^-{\wedge D^{j-j'}}_-{\simeq} \\
\Det_{\Lambda}^{-1} (\Delta_{K_\infty, v} (j'))  
\ar[r]_-{\Phi_{K_\infty, v}^{j'}}^-{\simeq} 
&\Det_{\Lambda} (R_{K_v}).
}
\]
Here, $\xi$ is the basis of $\Z_p(1)$ which we fixed at the beginning of \S \ref{State_2} to construct the map $\Phi_{K_{\infty}, v}^j$. 
Note that both vertical maps are $\twist_{j-j'}$-semilinear. 
\end{prop}

\begin{proof}
Clear from the construction.
\end{proof}

\begin{prop}\label{prop:extra_l}
Let $v$ be a finite prime of $k$ not lying above $p$. 
Then for any $j, j' \in \Z_{\geq 1}$, we have
\[
\Twist_{v, j, j'} (\cH_{K_\infty /k, v}^j) = \cH_{K_\infty /k, v}^{j'}
\]
and, if $K/k$ is unramified at $v$, 
\[
\Twist_{v, j, j'} (\cE_{K_\infty /k, v}^j) = \cE_{K_\infty /k, v}^{j'}.
\] 
Here, we note that the map $\Twist_{v, j, j'} = \Twist_{v, j, j'}^\xi$ is independent of the choice of a basis $\xi$ of $\Z_p(1)$ as the Euler characteristic $\ch (\Delta_{K_\infty, v})$ is zero.
\end{prop}

\begin{proof}
The first equality follows by the property $\twist_{j-j'} (h_{K_{\infty} /k, v}^j) = h_{K_{\infty}/k, v}^{j'}$. 
Then the second also follows since 
$\twist_{j-j'} (- N(v)^{j-1} \sigma_{K_\infty / k , v}^{-1}) = - N(v)^{j'-1} \sigma_{K_\infty / k, v}^{-1}$.
\end{proof}

Now we are ready to show the final formula.

\begin{prop}\label{prop:extra_Z1}
For any $j , j' \in \Z_{\geq 1}$, 
 we have 
\[
\Twist_{j, j'}^{\loc}  (Z_{K_\infty / k, S}^{\loc,j}) =  Z_{K_\infty / k, S}^{\loc , j'}.
\] 
\end{prop}

\begin{proof}
We fix a basis $\xi = (\zeta_{p^n})_n \in \Z_p (1) = \varprojlim_n \mu_{p^n} (K_\infty)$ and 
\[
\iota_{K_\infty, i} : K_\infty \hookrightarrow \C
\] 
for each $1 \leq i \leq r_k$ as in \S \ref{sec:pf_lETNC}. 
We check the compatibility of each local component of $Z^{\loc, j}_{K_{\infty}/k, S}$ in Definition \ref{def:Z} 
for the maps $\Twist_{v, j, j'}^\xi$ and \eqref{shift X}.
The compatibility for non-archimedean primes follows from Propositions \ref{prop:extra_p} and \ref{prop:extra_l}, combined with
\[
\wedge D^{m} \parenth{\wedge_{i=1}^{r_k} x_{K, i} (1+T) }=\wedge_{i=1}^{r_k} x_{K, i} (1+T)
\] 
for $m \in \Z$. 
Also by the choice of $\iota_{K_\infty, i}$ and the definition of $e_{\iota_{K_\infty}, i}^j \in X_{K_\infty} (j)$, 
we have 
$\Twist_{X_{K_\infty}, j, j'}^\xi (e_{\iota_{K_\infty}, i}^j) = e_{\iota_{K_\infty}, i}^{j'}$ for any $1 \leq i \leq r_k$. 
This implies the compatibility of archimedean components.
The compatibility for the elements $D_j(K_\infty / k)$ and $\ff_j (K_\infty / k)_v$, i.e.,
\[
\twist_{j - j'}(D_{j-1} (K_\infty / k)) = D_{j'-1} (K_\infty / k),
\qquad
\twist_{j - j'}(\ff_{j-1} (K_\infty / k)_v) = \ff_{j'-1} (K_\infty / k)_v
\]
follows from their construction (Definition \ref{defi:D} and Lemma \ref{lem:cond unit}).
\end{proof}

\section{The case $j \leq 0$}
\label{section:j <1}

In this section, we prove a variant of Theorem \ref{thm:lETNC} for $j \leq 0$ by
using the local duality. Namely, we construct a $\Lambda$-basis $Z_{K_\infty / k, S}^{\loc, j}$ of $\Xi^{\loc}_{K_{\infty}/k, S}(j)$ for $j\leq 0$ satisfying similar properties as in the case $j \geq 1$. We also prove some extra properties of the constructed basis $Z_{K_\infty / k, S}^{\loc, j}$. 

\subsection{Definition of the map $\vartheta_{K/k, S}^{\loc, j}$ when $j \leq0$}\label{vartheta<0}

The aim of this subsection is to define a ``duality'' isomorphism
\[
\LT^{\Xi (j)}_{K/k, S}: \Xi_{K/k, S}^{\loc} (j)  \xrightarrow{\sim} \Xi_{K/k, S}^{\loc} (1-j)^{-1, \#}
\]
and, using this, to define an isomorphism
\[
\vartheta_{K/k, S}^{\loc, j} : \C_p \otimes_{\Z_p} \Xi_{K/k, S}^{\loc} (j) \xrightarrow{\sim} \C_p [G] 
\]
for $j\leq 0$, as a variant of Definition \ref{vartheta_j}.

We first introduce some notations. 
For a commutative ring $R$, a finite abelian group $G$, and an $R[G]$-module $M$, 
we have a natural isomorphism as $R[G]$-modules
\[
\Hom_R (M, R) 
\stackrel{\sim}{\rightarrow} 
\Hom_{R[G]} (M, R[G])^\# \; ;
f \mapsto 
\sum_{\sigma \in G} f(\sigma(\cdot))\sigma^{-1}, 
\]
where $G$ acts on $\Hom_R (M, R)$ by 
\[
(\sigma \cdot f) (a) := f (\sigma^{-1} a) ,\;\;\;\; (\sigma \in G, f \in \Hom_R (M, R), a \in M). 
\]
Using this, for a finitely generated projective $R[G]$-module $P$ (resp.~a perfect complex $C$ of $R[G]$-modules),  
we have a canonical isomorphism between determinant modules 
\begin{equation}\label{Det_Hom}
\Det_{R[G]} (P) \simeq \Det_{R[G]} (\Hom_{R} (P, R))^{\#, \vee} 
\parenth{\text{resp. }  \Det_{R[G]} (C) \simeq \Det_{R[G]} (\RHom_{R} (C, R))^{\#, \vee}}, 
\end{equation}
where 
$
(-)^\vee : \mathcal{P}_{R[G]} \xrightarrow{\sim} \mathcal{P}_{R[G]} ; (L, r) \mapsto (\Hom_{R[G]} (L, R[G]) , r)
$ 
is the anti equivalence between the category of the graded invertible $R[G]$-modules $\mathcal{P}_{R[G]}$ (see \S \ref{App:det}). 

Now let $K/k$ be a finite abelian extension of number fields and put $G := \Gal (K/k)$. 
We keep the same notation in \S \ref{sec:form_lETNC}. 
For any $j \in \Z_{\leq 0}$, we consider a $\Z_p[G]$-isomorphism 
\[
X_K (j) \xrightarrow{\sim} \Hom_{\Z_p} (X_K (1-j) , \Z_p) ; e_{\iota}^{j} \mapsto  e_{\iota}^{1-j, \ast}, 
\]
where $e_{\iota}^{1-j, \ast}$ is the dual basis of $e_{\iota}^{1-j}$. 
By using \eqref{Det_Hom}, this induces
\begin{equation}\label{Xduality}
\Det_{\Z_p[G]}^{-1} (X_K (j)) \simeq 
\Det_{\Z_p[G]}^{-1} (\Hom_{\Z_p} (X_K (1-j) , \Z_p))
\simeq 
\Det_{\Z_p[G]}^{-1} (X_K (1-j))^{\#, \vee} .
\end{equation}

Let $v$ be a finite prime of $k$.
Recall the semi-local Galois cohomology complex
\[
\Delta_{K_v}(j):=\RG(K_v, \Z_p(j)):=\bigoplus_{w\mid v}\RG(K_w, \Z_p(j))
\]
introduced in \S \ref{ss:setup}.
We use the local duality of the Galois cohomology,
following Nekov\'{a}\v{r} \cite{Nek06}; we reformulate \cite[Proposition 5.2.4]{Nek06} in our semi-local setting.
For each complex $C = [ \cdots \to C^i \xrightarrow{d^i} C^{i+1} \to \cdots ]$ and an integer $a \in \Z$, 
we write $\tau_{\geq a}C$ for the truncated complex defined by 
\[
\tau_{\geq a}C := [ \cdots \to 0 \to 0 \to \Coker d^{a-1} \xrightarrow{d^a} C^{a+1} 
\xrightarrow{d^{a+1}} C^{a+2} \to \cdots ]. 
\]
Now consider a morphism of complexes
\[
\RG (K_v , \Z_p (1)) \to \tau_{\geq 2} \RG (K_v , \Z_p (1)) \xrightarrow{\inv} \bigoplus_{w \mid v} \Z_p [-2]
\xrightarrow{\sum} \Z_p [-2], 
\]
where the second arrow is the one induced by the invariant maps. 
Combining this with the cup product, 
for any $j \in \Z$, we obtain a morphism of complexes 
\[
\RG (K_v , \Z_p (j)) \otimes_{\Z_p} \RG (K_v , \Z_p (1-j)) \xrightarrow{\cup}  
\RG (K_v, \Z_p (1)) \to \Z_p [-2], 
\]
where $\otimes_{\Z_p}$ denotes the tensor product of complexes.
This gives us a quasi-isomorphism
\begin{align}\label{LTD_alpha}
\alpha_{K_v, \Z_p (j)} : 
\RG (K_v , \Z_p (j)) \xrightarrow{\sim} 
 \RHom_{\Z_p} (\RG (K_v , \Z_p (1-j)) , \Z_p)[-2] ,
\end{align}
which is the local duality.
Here, we note that some authors use different conventions concerning the order of cup products.

Now, by using \eqref{Det_Hom} again, the local duality $\alpha_{K_v, \Z_p (j)}$ induces
\begin{equation}\label{local_duality_det}
\Det_{\Z_p[G]}^{-1} (\Delta_{K_v}(j))
\simeq 
\Det_{\Z_p[G]}^{-1} (\Delta_{K_v}(1-j))^{\# , \vee}
\end{equation}
for any $j \in \Z$. 

Putting \eqref{Xduality} and \eqref{local_duality_det} together, we define an isomorphism
\[
\LT^{\Xi (j)}_{K/k, S}: \Xi_{K/k, S}^{\loc} (j)  \xrightarrow{\sim} \Xi_{K/k, S}^{\loc} (1-j)^{-1, \#}
\]
by
\begin{align*}
\Xi_{K/k, S}^{\loc} (j) 
&=
\bigotimes_{v \in S_f} \Det_{\Z_p[G]}^{-1} (\Delta_{K_v}(j))
\otimes_{\Z_p[G]}
\Det_{\Z_p[G]}^{-1} (X_K (j))
\\
&\overset{\eqref{local_duality_det} \text{ and } \eqref{Xduality}}{\simeq}
\Big{(}
\bigotimes_{v \in S_f} \Det_{\Z_p[G]}^{-1} (\Delta_{K_v}(1-j))
\Big{)}^{\# , \vee}
\otimes_{\Z_p[G]}
\Det_{\Z_p[G]}^{-1} (X_K (1-j))^{\#, \vee}
\\
&\overset{\psi}{\simeq} 
\Det_{\Z_p[G]}^{-1} (X_K (1-j))^{\#, \vee}
\otimes_{\Z_p[G]}
\Big{(}\bigotimes_{v \in S_f} \Det_{\Z_p[G]}^{-1} (\Delta_{K_v}(1-j))\Big{)}^{\# , \vee}
\\
&\simeq
\parenth{
\bigotimes_{v \in S_f} \Det_{\Z_p[G]}^{-1} (\Delta_{K_v}(1-j))
\otimes_{\Z_p[G]}
\Det_{\Z_p[G]}^{-1} (X_K (1-j))
}^{\# , \vee}
\\
&=
\Xi_{K/k, S}^{\loc} (1-j)^{\#, \vee} 
=
\Xi_{K/k, S}^{\loc} (1-j)^{-1, \#}
\end{align*}
for any $j \in \Z_{\leq 0}$. 
Here the third isomorphism follows from Lemma \ref{lem:prodinv} and the final equality holds since the grade of $\Xi_{K/k, S}^{\loc} (1-j)$ is zero by Remark \ref{rem:grade_zero}.

\begin{defn}\label{map<1}
We define an isomorphism 
\[
\vartheta_{K/k, S}^{\loc, j} : \C_p \otimes_{\Z_p} \Xi_{K/k, S}^{\loc} (j) \xrightarrow{\sim} \C_p [G] 
\]
for any $j \in \Z_{\leq 0}$ by the composite map 
\[
\C_p \otimes_{\Z_p} \Xi_{K/k, S}^{\loc} (j) \overset{\LT^{\Xi (j)}_{K/k, S}}{\simeq} 
\C_p \otimes_{\Z_p} \Xi_{K/k, S}^{\loc} (1-j)^{-1, \#} 
{\xrightarrow{\sim}} \C_p[G], 
\]
where the second isomorphism sends $\vartheta_{K/k, S}^{\loc, 1-j}$ to $1$. 
\end{defn}

This definition of $\vartheta_{K/k, S}^{\loc, j}$ for $j \in \Z_{\leq 0}$ depends on the local duality.
An explicit description will be given in \S \ref{ss:exp_theta}.

\subsection{Local ETNC for non-positive Tate twists}\label{ss:j<1}

Now we extend Theorem \ref{thm:lETNC} to the case $j \leq 0$. 
We keep the same notation in \S \ref{sec:pf_lETNC}. 

Taking the projective limit of the isomorphism $\LT^{\Xi (j)}_{K/k, S}$ defined in the previous subsection, 
we have a canonical isomorphism 
\[
\LT_{K_\infty / k, S}^{\Xi (j)} : \Xi_{K_\infty / k, S}^{\loc} (j) \xrightarrow{\sim} 
\Xi_{K_\infty / k, S}^{\loc} (1-j)^{\#, \vee} 
=
\Xi_{K_\infty / k, S}^{\loc} (1-j)^{-1, \#} 
\]
for any $j \in \Z_{\leq 0}$.  
For each archimedean place $v$ of $k$, 
we write $c_{K_\infty /k}^v \in \GG_{K_\infty}$ for the generator of the inertia group of $v$ (whose order is either $1$ or $2$) and put 
\[
c^\infty_{K_\infty / k} : = \prod_{v \in S_\infty }c_{K_\infty /k}^v \in \GG_{K_\infty}. 
\]
\begin{defn}\label{Z<1}
For $j \in \Z_{\leq 0}$, 
we define $Z_{K_\infty / k, S}^{\loc , j} \in \Xi_{K_\infty / k, S}^{\loc} (j)$ by 
\[
Z_{K_\infty / k, S}^{\loc , j} := (-1)^{j r_\R (k) + r_\C (k)} \cdot c^\infty_{K_\infty / k} \cdot
 (\LT_{K_\infty / k, S}^{\Xi (j)})^{-1} \Big{(} (Z_{K_\infty / k, S}^{\loc , 1-j})^{\ast} \Big{)}, 
\]
where $(-)^\ast$ means the dual basis.  
This $Z_{K_\infty / k, S}^{\loc , j}$ is a $\Lambda = \Z_p[[\GG_{K_\infty}]]$-basis of $\Xi_{K_\infty / k, S}^{\loc} (j)$. 
\end{defn}

The reason for the extra factors in the definition will become clear in Proposition \ref{prop:extra_Z} below. 

\begin{thm}\label{thm:j<1}
For $j \in \Z_{\leq 0}$ and an intermediate number field $M$ of $K_\infty/k$,
the composite map 
\[
\Xi_{K_\infty / k, S}^{\loc} (j) \twoheadrightarrow 
\Xi_{M / k, S}^{\loc} (j) \hookrightarrow \C_p \otimes_{\Z_p} \Xi_{M / k, S}^{\loc} (j) 
\overset{\vartheta_{M/k, S}^{\loc, j}}{\xrightarrow{\sim}} \C_p [\Gal(M/k)]
\]
sends $Z_{K_\infty / k, S}^{\loc , j}$ to 
\[
(-1)^{  r_\C (k)} \cdot c^\infty_{M / k} \times 
\begin{cases}\displaystyle
\cfrac{\Theta_{M/k, S}^\ast (1-j)^\#}{\Theta_{M/k, S}^\ast (j)}
& \text{if } j\neq 0,  \\
\displaystyle
\delta_{M/k, \trivial} \cdot 
\left(\prod_{v \in S_f} 
\delta_v(M /k) \right)^{-1}
\cfrac{\Theta_{M/k, S}^\ast (1)^\#}{\Theta_{M/k}^\ast (0)}
 & \text{if }  j=0, 
\end{cases}
\]
where $c_{M / k}^\infty$ is the image of $c_{K_\infty / k}^\infty$ by the 
natural restriction map $\Lambda \to \Z_p [\Gal(M/k)]$. 
\end{thm}
\begin{proof}
Let $z_{M /k, S}^{\loc, 1-j} \in \Xi_{M / k, S}^{\loc} (1-j)$ denote the image of $Z_{K_\infty / k, S}^{\loc , 1-j}$ 
by the natural map $\Xi_{K_\infty / k, S}^{\loc} (1-j)\twoheadrightarrow \Xi_{M / k, S}^{\loc} (1-j)$. 
By Definitions \ref{map<1} and \ref{Z<1}, 
 the image of $Z_{K_\infty / k, S}^{\loc , j}$ to be computed is equal to
\[(-1)^{j r_\R (k) + r_\C (k)} \cdot c^\infty_{M / k} 
\parenth{\vartheta_{M/k, S}^{\loc, 1-j} (z_{M /k, S}^{\loc, 1-j})}^{\#, -1}.
\]
Therefore, Theorem \ref{thm:lETNC} shows the theorem.  
\end{proof}

\subsection{Explicit description of $\vartheta_{K/k, S}^{\loc, j}$ for $j \in \Z_{\leq 0}$}\label{ss:exp_theta}

In this subsection, we give an explicit description of $\vartheta_{K/k, S}^{\loc, j}$ for $j \in \Z_{\leq 0}$. 

For a $\Q_p[G]$-module $M$, we set $M^\ast := \Hom_{\Q_p} (M, \Q_p) \simeq \Hom_{\Q_p[G]} (M, \Q_p[G])^\#$. 
For a $\Q_p[G]$-homomorphism $f : M \to N$, we write $f^\ast : N^\ast \to M^\ast$ for the $\Q_p[G]$-homomorphism induced by $f$. 

For any $p$-adic prime $v$ of $k$ and $j \leq 0$,
we consider the Bloch-Kato dual exponential map 
\[
\exp_{\Q_p (j)}^\ast : H^1 (K_v, \Q_p (j)) \overset{\alpha_{K_v,  \Z_p(j)}}{\xrightarrow{\sim}}
H^1 (K_v, \Q_p (1-j))^\ast 
\xrightarrow{(\exp_{\Q_p (1-j)})^\ast} 
K_v^\ast,
\]
where the first isomorphism is induced by the local duality $\alpha_{K_v,  \Z_p(j)}$
in \S \ref{vartheta<0} and 
the second arrow is the dual of the exponential map $\exp_{\Q_p (1-j)}$. 

As in \S \ref{ss:setup}, let us define an isomorphism
\[
\begin{cases}
\vartheta_{K_v}^{j} : 
\Q_p \otimes_{\Z_p} \Det_{\Z_p[G]}^{-1} (\Delta_{K_v}(j))
\xrightarrow{\sim} 
\Q_p [G]
&\text{if } v \nmid p, \\
\phi_{K_v}^j: \Q_p \otimes_{\Z_p} \Det^{-1}_{\Z_p[G]} (\Delta_{K_v}(j))
\xrightarrow{\sim} 
\Det_{\Q_p[G]}(K_v^\ast)
&\text{if }v \mid p.
\end{cases}
\]
for a finite prime $v$ of $k$ and $j \in \Z_{\leq 0}$.
The construction again depends on whether $j = 0$ or $j < 0$.

Suppose $j<0$ and $v \nmid p$.
Then the complex 
$\Q_p \otimesL_{\Z_p} \Delta_{K_v}(j)$ is acyclic and the isomorphism $\vartheta_{K_v}^{j}$ is defined by Proposition \ref{prop:det_ses3}. 

Suppose $j<0$ and $v \mid p$.
Then the complex $\Q_p \otimesL_{\Z_p} \Delta_{K_v}(j)$ is acyclic outside degree one and the dual exponential map $\exp_{\Q_p (j)}^\ast$ is an isomorphism. Then the canonical map in Proposition \ref{prop:det_ses4} and $\exp_{\Q_p (j)}^\ast$ induce the isomorphism $\phi_{K_v}^j$.

Suppose $j = 0$.
We consider the composite map 
\begin{equation}\label{rec_v}
 H^0 (K_v, \Q_p)
 =
\bigoplus_{w \mid v} \Q_p  
\xrightarrow{(\ord_{K_v})^\ast}
(\Q_p \otimes_{\Z_p} \bigoplus_{w \mid v}  \widehat{K_w^\times})^\ast 
\simeq 
\Hom ( \bigoplus_{w \mid v} G_{K_w}^{\ab} , \Q_p)
=
H^1 (K_v, \Q_p),  
\end{equation}
where $G_{K_w}^{\ab}$ is the Galois group of the maximal abelian extension over $K_w$ and the isomorphism is induced by the reciprocity map of local class field theory.  

Now suppose $j =0$ and $v \nmid p$.
Then the complex 
$\Q_p \otimesL_{\Z_p} \Delta_{K_v}(0)$ is acyclic outside degree zero and one
and the map \eqref{rec_v} is an isomorphism. 
Using Proposition \ref{prop:det_ses4}, we define $\vartheta_{K_v}^{0}$ as
\begin{align*}
\Q_p\otimes_{\Z_p}\Det_{\Z_p[G]}^{-1}(\Delta_{K_v}(0))
& \simeq \Det_{\Q_p[G]}^{-1}(\RG(K_v, \Q_p))\\
& \simeq  \Det^{-1}_{\Q_p[G]}(H^0(K_v, \Q_p)) \otimes_{\Q_p[G]} \Det_{\Q_p[G]}(H^1(K_v, \Q_p))\\
&\overset{(\ref{rec_v})}{\simeq} \Det^{-1}_{\Q_p[G]}(H^1(K_v, \Q_p)) \otimes_{\Q_p[G]} \Det_{\Q_p[G]}(H^1(K_v, \Q_p))\\
&\overset{\ev}{\simeq} \Q_p[G],
\end{align*}
where the isomorphism $\ev$ is the one introduced in \S \ref{App:det_defn}.

Finally, we suppose $j = 0$ and $v \mid p$.
Then the complex $\Q_p \otimesL_{\Z_p} \Delta_{K_v}(0)$ is acyclic outside degree zero and one again. 
In this case, we consider the following commutative diagram; 
\[
\xymatrix{
0 \ar[r]
&
\bigoplus_{w \mid v} \Q_p 
\ar[r]^-{(\ord_{K_v})^\ast} \ar@{=}[d]
&
(\Q_p \otimes_{\Z_p} \bigoplus_{w \mid v} \widehat{K_w^\times} )^\ast
\ar[r]
\ar[d]^-{\sim}
&
(\Q_p \otimes_{\Z_p} \bigoplus_{w \mid v} U_{K_w})^\ast
\ar[r] \ar[d]^-{\sim}_-{(\oplus_{w \mid v}\exp_{p})^\ast}
&0
\\
0 \ar[r]
&
H^0 (K_v, \Q_p)
\ar[r]^-{\eqref{rec_v}}
&
H^1 (K_v, \Q_p)
\ar[r]^-{\exp_{\Q_p}^\ast}
&
K_v^\ast 
\ar[r]
&
0, 
} 
\]
where the middle vertical arrow is induced by the reciprocity map and 
$\exp_p : \Q_p \otimes U_{K_w} \xrightarrow{\sim} K_w$ is the $p$-adic exponential map. 
Since the upper sequence is clearly exact, so is the bottom sequence.
We then obtain a canonical isomorphism
\[
\Det_{\Q_p[G]}(H^1 (K_v, \Q_p)) \simeq \Det_{\Q_p[G]}(H^0 (K_v, \Q_p)) \otimes_{\Q_p[G]} \Det_{\Q_p[G]}(K_v^\ast )
\]
by Proposition \ref{prop:det_ses2}.
Using this isomorphism and Proposition \ref{prop:det_ses4}, we define $\phi_{K_v}^0$ as 
\begin{align*}
& \Q_p\otimes_{\Z_p}\Det_{\Z_p[G]}^{-1}(\Delta_{K_v}(0))\\
& \quad \simeq \Det_{\Q_p[G]}^{-1}(\RG(K_v, \Q_p))\\
& \quad  \simeq \Det^{-1}_{\Q_p[G]}(H^0(K_v, \Q_p)) \otimes_{\Q_p[G]} \Det_{\Q_p[G]}(H^1(K_v, \Q_p))\\
& \quad \simeq \Det^{-1}_{\Q_p[G]}(H^0(K_v, \Q_p)) \otimes_{\Q_p[G]} \Det_{\Q_p[G]}(H^0 (K_v, \Q_p)) \otimes_{\Q_p[G]} \Det_{\Q_p[G]}(K_v^\ast )\\
& \quad \overset{\ev}{\simeq} \Q_p[G] \otimes_{\Q_p[G]} \Det_{\Q_p[G]}(K_v^\ast )\simeq \Det_{\Q_p[G]}(K_v^\ast ).
\end{align*}

This completes the construction of $\vartheta_{K_v}^j$ and $\phi_{K_v}^j$ for $j \leq 0$.
The relation with the ones for $j \geq 1$ in \S \ref{ss:setup} are described as follows.

\begin{lem}\label{lem:commutativity_phi_vartheta}
Let $v$ be a finite prime of $k$.
For any $j \in \Z_{\leq 0}$, the following are commutative:
\[
\xymatrix@C=40pt{
\Q_p \otimes_{\Z_p} \Det_{\Z_p[G]}^{-1} (\Delta_{K_v}(j))
\ar[r]^-{\alpha_{K_v, \Z_p(j)}}_-{\sim}
\ar[d]^-{\vartheta_{K_v}^j}_-{\sim}
&
\Q_p \otimes_{\Z_p} \Det_{\Z_p[G]}^{-1} (\Delta_{K_v}(1-j))^{\#, \vee}
\ar[d]^-{(\vartheta_{K_v}^{1-j})^{\vee, -1}}_-{\sim}
\\
\Q_p[G]
\ar[r]_-{\sim}^-{\# : \sigma \mapsto \sigma^{-1}}
& 
\Q_p [G]^{\#}
}
\]
when $v \nmid p$ and 
\[
\xymatrix@C=40pt{
\Q_p \otimes_{\Z_p} \Det_{\Z_p[G]}^{-1} (\Delta_{K_v}(j))
\ar[r]^-{\alpha_{K_v, \Z_p(j)}}_-{\sim}
\ar[d]^-{\phi_{K_v}^j}_-{\sim}
&
\Q_p \otimes_{\Z_p} \Det_{\Z_p[G]}^{-1} (\Delta_{K_v}(1-j))^{\#, \vee}
\ar[d]^-{(\phi_{K_v}^{1-j})^{\vee, -1}}_-{\sim}
\\
\Det_{\Q_p[G]} (K_v^\ast)
\ar[r]_-{\sim}^-{\eqref{Det_Hom}}
& 
\Det_{\Q_p[G]} (K_v)^{\#, \vee}
}
\]
when $v \mid p$.
\end{lem}
\begin{proof}
When $j<0$ and $v \nmid p$, the commutativity holds
since the isomorphisms $\vartheta_{K_v}^j$ and $\vartheta_{K_v}^{1-j}$ are defined as the canonical ones defined for acyclic complexes.

When $j<0$ and $v \mid p$, the commutativity is clear by 
the constcution.

When $j=0$ and $v \nmid p$, the commutativity follows from the commutativity of
\[
\xymatrix@C=50pt{
H^0 (K_v, \Q_p) 
\ar[r]^-{\alpha_{K_v, \Z_p }}_-{\sim}
\ar[d]_-{\sim}^-{\eqref{rec_v}}
&
H^2 (K_v, \Q_p (1))^\ast
\ar[d]_-{\sim}^-{(\inv_{K_v}^{-1} \circ \ord_{K_v})^{\ast}}
\\
H^1 (K_v, \Q_p) 
\ar[r]^-{\alpha_{K_v, \Z_p }}_-{\sim}
&
H^1 (K_v, \Q_p (1))^\ast 
}
\]
(see \cite[Corollary 7.2.13]{NSW08}).

Finally suppose $j=0$ and $v \mid p$.
We similarly have a commutative diagram 
\[
\xymatrix@C=55pt{
0 \ar[r]
&
H^0 (K_v, \Q_p)
\ar[r]^-{\eqref{rec_v}}
\ar[d]^-{\sim}_-{\alpha_{K_v, \Z_p}}
&
H^1 (K_v, \Q_p)
\ar[r]^-{\exp_{\Q_p}^\ast}
\ar[d]^-{\sim}_-{\alpha_{K_v, \Z_p}}
&
\displaystyle
K_v^\ast 
\ar[r]
\ar@{=}[d]
&
0
\\
0 \ar[r]
&
H^2 (K_v, \Q_p (1))^\ast 
\ar[r]^-{(\inv_{K_v}^{-1} \circ \ord_{K_v})^\ast} 
&
H^1 (K_v, \Q_p (1))^\ast
\ar[r]^-{(\exp_{\Q_p (1)})^\ast}
&
\displaystyle
K_v^\ast 
\ar[r] 
&0. 
\\
} 
\]
From this, we deduce that the isomorphism $\phi_{K_v}^0$ coincides with the composite map
\begin{align*}
& \Q_p\otimes_{\Z_p}\Det_{\Z_p[G]}^{-1}(\Delta_{K_v}(0))
 \overset{\alpha_{K_v,\Z_p}}{\simeq} \Det_{\Q_p[G]}^{-1}(\RG(K_v, \Q_p(1)))^{\#, \vee}\\
& \quad \simeq \parenth{ \Det_{\Q_p[G]}(H^1(K_v, \Q_p(1))) \otimes_{\Q_p[G]} \Det_{\Q_p[G]}^{-1}(H^2(K_v, \Q_p(1))) }^{\#, \vee}\\
& \quad \overset{\eqref{Kummer seq}}{\simeq} \Det_{\Q_p[G]}(H^1_f(K_v, \Q_p(1)))^{\#, \vee} \\
& \quad \overset{\log_{\Q_p(1)}}{\simeq} \Det_{\Q_p[G]}(K_v)^{\#, \vee} \overset{\eqref{Det_Hom}}{\simeq} \Det_{\Q_p[G]}(K_v^\ast ), 
\end{align*}
where the second isomorphism comes from Proposition \ref{prop:det_ses4}.
By the definition of $\phi_{K_v}^1$, the composite map 
from the second isomorphism to the fourth is equal to $(\phi_{K_v}^1)^{\vee, -1}$. 
This implies the claim.
\end{proof}

For each $j \leq 0$, as in Definition \ref{vartheta_j}, combining $\vartheta_{K_v}^j$ and $\phi_{K_v}^j$ for each $v \in S_f$ yields an isomorphism 
\begin{equation}\label{dual_version}
\Q_p \otimes_{\Z_p} \bigotimes_{v \in S_f} \Det_{\Z_p[G]}^{-1} (\Delta_{K_v}(j)) \simeq 
\Det_{\Q_p[G]} ((\Q_p \otimes_{\Q} K )^\ast).
\end{equation}
Using this, for $j \leq 0$, we construct a composite map 
\begin{align}\label{second_map}
&  \C_p \otimes_{\Z_p} \Xi_{K/k, S}^{\loc} (j) 
= \C_p \otimes_{\Z_p} \left(\bigotimes_{v \in S_f} \Det_{\Z_p[G]}^{-1} (\Delta_{K_v}(j))
\otimes \Det_{\Z_p[G]}^{-1} (X_K (j))\right)
 \\& \quad
\overset{\eqref{dual_version} \text{ and } \eqref{Xduality} }{\xrightarrow{\sim}} 
\C_p \otimes_{\Q_p} \parenth {\Det_{\Q_p[G]} ((\Q_p \otimes K )^\ast) \otimes_{\Q_p[G]} \Det_{\Q_p[G]}^{-1} ((\Q_p \otimes X_K (1-j))^\ast)} 
\\ 
&\quad
\overset{\wedge (\alpha_K^{1-j})^{\ast} \text{ and } \eqref{Y Betti}^{\ast}}{\xrightarrow{\sim}} 
\C_p \otimes_{\Q_p} 
\parenth{
\Det_{\Q_p[G]}( (\Q_p \otimes H_K (1-j) )^\ast) \otimes_{\Q_p[G]} \Det_{\Q_p[G]}^{-1}
( ( \Q_p \otimes H_K (1-j)  )^\ast) 
}
\\
&\quad
\overset{\ev}{\simeq}
\C_p[G]
\end{align}
($\alpha_K^{1-j}$ is defined in \S \ref{period map}).

\begin{prop}
For any $j \in \Z_{\leq 0}$, 
this \eqref{second_map} coincides with $(-1)^{r_k}$ times the homomorphism $\vartheta_{K / k, S}^{\loc, j}$ in Definition \ref{map<1}. 
\end{prop}
\begin{proof}
To ease the notation, we write $[-] := \C_p \otimes_{\Q_p} \Det_{\Q_p[G]} (-)$ and so $[-]^{-1} = \C_p \otimes_{\Q_p} \Det_{\Q_p[G]}^{-1} (-)$.
The following is commutative.
{\footnotesize
\[
\xymatrix@C=30pt{
\C_p \otimes_{\Z_p} \Xi_{K/k , S}^{\loc} (j)
\ar[r]^-{\LT_{K/k, S}^{\Xi(j)}}_-{\sim}
\ar[d]^-{\eqref{dual_version}\text{ and }\eqref{Xduality}}_-{\sim}
&
\C_p \otimes_{\Z_p} \Xi_{K/k , S}^{\loc} (1-j)^{-1, \#}
= \C_p \otimes_{\Z_p} \Xi_{K/k , S}^{\loc} (1-j)^{\#, \vee}
\ar[d]_-{\sim}^-{\eqref{use_later}^{\vee, -1}}
\\
 [(\Q_p \otimes K )^\ast]\otimes[(\Q_p \otimes X_K (1-j)   )^\ast]^{-1} 
\ar[r]^-{\eqref{Det_Hom} \text{ and }\psi}_-{\sim}
\ar[d]^-{\wedge (\alpha_K^{1-j})^{\ast, -1}\text{ and }\eqref{Y Betti}}_-{\sim}
&
[\Q_p \otimes X_K (1-j)]^{-1, \#, \vee}
 \otimes
[\Q_p \otimes K ]^{\#, \vee} 
\ar[d]^-{\eqref{Y Betti} \text{ and }  (\wedge \alpha_K^{1-j})^{\vee, -1}}_-{\sim}
\\
 [(\Q_p \otimes H_K (1-j))^\ast]\otimes [(\Q_p \otimes H_K (1-j) )^\ast]^{-1}
 \ar[r]^-{\eqref{Det_Hom}\text{ and }\psi}_-{\sim}
\ar[d]^-{\ev}_-{\sim}
&
 [\Q_p \otimes H_K (1-j)]^{-1, \#, \vee} \otimes
 [\Q_p \otimes H_K (1-j)]^{\#, \vee}
\ar[d]^-{\ev}_-{\sim}
\\
\C_p[G]
\ar[r]_-{\sim}^-{\times (-1)^{r_k}}
&
\C_p[G]. 
}
\]
}
Indeed, the commutativity of the upper square follows from Lemma \ref{lem:commutativity_phi_vartheta} 
and the definition of $\LT_{K_\infty / k, S}^{\Xi (j)}$. 
The commutativity of the middle square and the bottom square is obvious (see the notation about the evaluation map 
in \S \ref{App.ev}). 

The composite map of the left vertical arrows is exactly the isomorphism \eqref{second_map}. 
By Definition \ref{vartheta_j}, the composite map of the right vertical arrows is the isomorphism 
$\C_p \otimes_{\Z_p} \Xi_{K/k, S}^{\loc} (1-j)^{-1, \#} \xrightarrow{\sim} \C_p [G] ;  \vartheta_{K/k, S}^{\loc, 1-j} 
\mapsto 1$. 
Now the proposition follows from this commutative diagram and Definition \ref{map<1}. 
\end{proof}

\subsection{Twists}\label{ss:twist}

The aim of this subsection is to prove the following.

\begin{prop}\label{prop:extra_Z}
Proposition \ref{prop:extra_Z1} is valid for $j, j' \in \Z$ instead of $\Z_{\geq 1}$.
\end{prop}

To prove this, it is enough to show the following three claims:
\begin{itemize}
 \item[(i)] For any $j , j' \in \Z_{\geq 1}$, we have 
$\Twist_{j, j'}^{\loc}  (Z_{K_\infty / k, S}^{\loc,j}) =  Z_{K_\infty / k, S}^{\loc , j'}$. 
 \item[(ii)] For any $j , j' \in \Z_{\leq 0}$, we have 
$\Twist_{j, j'}^{\loc}  (Z_{K_\infty / k, S}^{\loc,j}) =  Z_{K_\infty / k, S}^{\loc , j'}$. 
 \item[(iii)]
  For any $j  \in \Z_{\geq 1}$, we have 
 $\Twist_{j, 1-j}^{\loc}  (Z_{K_\infty / k, S}^{\loc,j}) =  Z_{K_\infty / k, S}^{\loc , 1-j}. $
 \end{itemize}
Claim (i) is already established.
Claim (ii) is established as follows:
By the compatibility of the cup product, 
we have a commutative diagram 
\[
\xymatrix{
\Xi_{K_\infty / k, S}^{\loc} (j) \ar[rr]^-{\Twist_{j, j'}^{\loc}}_-{\simeq} \ar[d]^-{\LT_{K_\infty / k, S}^{\Xi (j)}}_-{\simeq}
&& 
\Xi_{K_\infty / k, S}^{\loc} (j') \ar[d]^-{\LT_{K_\infty / k, S}^{\Xi (j')}}_-{\simeq}
\\
\Xi_{K_\infty / k, S}^{\loc} (1-j)^{-1, \#} 
\ar[rr]^-{\Twist_{1-j', 1-j}^{\loc , \ast}}_-{\simeq} 
&&
\Xi_{K_\infty / k, S}^{\loc} (1-j')^{-1, \#} , 
}
\]
where the horizontal arrows are $\twist_{j-j'}$-semilinear and  
 the bottom arrow is defined by  
\[
\Twist_{j', j}^{\loc, \ast} : 
\Hom_{\Lambda} (\Xi_{K_\infty / k, S}^{\loc} (j), \Lambda)^{\#} \xrightarrow{\sim} 
\Hom_{\Lambda} (\Xi_{K_\infty / k, S}^{\loc} (j'), \Lambda )^{\#} ; 
f \mapsto  \twist_{j'-j} \circ f \circ \Twist_{j', j}^{\loc}. 
\] 
Now claim (ii) follows from claim (i), this diagram, and
\[
\twist_{j-j'} \parenth{(-1)^{j r_\R (k) + r_\C (k)} \cdot c^\infty_{K_\infty / k}}
= \parenth{(-1)^{j' r_\R (k) + r_\C (k)} \cdot c^\infty_{K_\infty / k}}. 
\] 

It remains to show claim (iii). 
For this, we need detailed study of the twists.

\subsubsection{$p$-adic primes}

We keep the notation in \S \ref{sec:pf_lETNC}. 
For an integer $j \geq 1$, 
we write 
\[
\LT_{K_\infty / k, p}^{\Delta (1-j)} :   \Det_{\Lambda}^{-1} (\Delta_{K_\infty, p} (1-j))  \xrightarrow{\sim} 
  \Det_{\Lambda}^{-1} (\Delta_{K_\infty, p} (j))^{\#, \vee} 
\]
for the $p$-adic component of the isomorphism $\LT_{K_\infty / k, S}^{\Xi (1-j)}$. 

Let us discuss the explicit reciprocity law of the local duality by Coleman maps. 
We define an $\OO_{K_p} [[\Gal (K_\infty / K)]] $-bilinear form
\[
\langle \;, \;\rangle: 
R_{K_p} \times R_{K_p} \to \OO_{K_p} [[\Gal (K_\infty / K)]]
\]
by
\[
\langle a(1+T), b(1+T) \rangle = ab ,
\;\; (a, b \in \OO_{K_p} [[\Gal (K_\infty / K)]]).  
\]
This is well-defined since $R_{K_p}$ is a free $\OO_{K_p} [[\Gal (K_\infty / K)]] $-module of rank one with basis $1+T$ 
by \cite[Lemme 1.5]{Per90}.
We also define an involution $\iota_{R_{K_p}}: R_{K_p} \to R_{K_p}$ by
\[
\iota_{R_{K_p}}(a (1+T)) = a^\# \cdot  (1+T)^{-1} = 
\sigma_{-1}a^\# \cdot  (1+T) , \;\; (a \in \OO_{K_p} [[\Gal(K_\infty / K)]]),  
\]
where $(-)^\#$ denotes the involution $\sigma \mapsto \sigma^{-1}$ for $\sigma \in \Gal (K_\infty / K))$ and
$\sigma_{-1} \in \Gal (K_\infty / K)$ is the unique element satisfying $\chi_{\cyc} (\sigma_{-1}) = -1$. 
Using these, we define a pairing 
\[
\langle \;, \;\rangle_{R_{K_p}} : 
R_{K_p} \times R_{K_p} \to \Lambda
\]
by
\[
\langle f, g \rangle_{R_{K_p}} = \sum_{\sigma \in \Gal (K_\infty / k (\mu_{p^\infty}))} 
\Tr_{K / \Q} \parenth{ \langle \sigma^{-1} f, \iota_{R_{K_p}} (g) \rangle } \sigma , 
\]
where the map $\Tr_{K / \Q} : \OO_{K_p} [[\Gal (K_\infty / K)]] \to \Z_p [[\Gal (K_\infty / K)]] \hookrightarrow \Lambda$ is induced by the trace map. 
We can check
\[
t\langle f, g\rangle_{R_{K_p}} = \langle t f, g\rangle_{R_{K_p}} = \langle f, t^\# g\rangle_{R_{K_p}}
\]
for $f, g \in R_{K_p}$ and $t \in \Lambda$. 
In other words,
$\langle \;, \;\rangle_{R_{K_p}} : R_{K_p} \times R_{K_p}^\# \to \Lambda$ 
is a $\Lambda$-bilinear form, which is moreover a perfect pairing since $K/\Q$ is unramified at $p$. 
Therefore, this and \eqref{Det_Hom} induce an isomorphism 
\begin{align}\label{RLpairing}
\Det_{\Lambda} (R_{K_p}) &\xrightarrow{\sim} \Det_{\Lambda} (R_{K_p})^{\# , \vee}
= \Hom_{\Lambda} (\bigwedge_{\Lambda}^{r_k} R_{K_p}, \Lambda)^\# ; \\
\wedge_{a=1}^{r_k} f_a &\mapsto \parenth{ \wedge_{b=1}^{r_k} g_b \mapsto \det (\langle f_a, g_b\rangle_{R_{K_p}})_{1\leq a, b \leq r_k}}. 
\end{align}

\begin{prop}\label{prop:ERL}
For $j \geq 1$, the diagram
\[\xymatrix@C=50pt{
 \Det_{\Lambda}^{-1} (\Delta_{K_\infty, p} (1-j)) \ar[r]^-{\LT_{K_\infty / k, p}^{\Delta(1-j)}}_-{\sim} \ar[d]^-{\simeq}_-{\Phi_{K_\infty, p}^{1-j}}
&
\Det_{\Lambda}^{-1} (\Delta_{K_\infty, p} (j))^{\#, \vee} 
 \\
 \Det_{\Lambda} (R_{K_p}) \ar[r]^-{\eqref{RLpairing}}_-{\sim}
 &
 \Det_{\Lambda} (R_{K_p})^{\#, \vee} \ar[u]_-{\simeq}^-{(\Phi_{K_\infty, p}^{j})^{\ast}}
}\]
is commutative up to sign $(-1)^{ r_k j}$, 
where the right vertical arrow is induced by $\Phi_{K_\infty,p}^{j}$. 
\end{prop}
\begin{proof}
For each $n \geq 0$, 
we consider a perfect pairing induced by the cup product and the invariant maps
\begin{align*}
\langle\; , \; \rangle_n 
& :
\bigoplus_{w_n \mid p} H^1 (K_{n, w_n} , \Z_p (1-j)) \times 
\bigoplus_{w_n \mid p} H^1 (K_{n, w_n} , \Z_p (j)) 
\\
&\xrightarrow{\cup}
\bigoplus_{w_n \mid p}  H^2 (K_{n, w_n} , \Z_p (1)) \xrightarrow{\sum_{w_n \mid p} \inv_{K_{n, w_n}}} \Z_p. 
\end{align*}
This induces a pairing
\[
\langle \; , \; \rangle_{\LT} : H^1 (\Delta_{K_\infty, p} (1-j)) \times H^1 (\Delta_{K_\infty, p} (j))  \to \Lambda
\]
defined by
\[
\langle (x_n)_n ,(y_n)_n \rangle_{\LT} = 
 \parenth{ \sum_{\sigma \in \GG_{K_n}} \langle x_n^{\sigma^{-1}} , y_n \rangle_{n} \sigma}_n,  
\]
where we regard $H^1 (\Delta_{K_\infty, p} (j)) = \varprojlim_n 
\parenth{\bigoplus_{w_n \mid p} H^1 (K_{n, w_n} , \Z_p (j))}$ and 
$\Lambda = \varprojlim_n \Z_p [\GG_{K_n}]$. 
We can check
that $\langle \; , \; \rangle_{\LT} : H^1 (\Delta_{K_\infty, p} (1-j)) \times H^1 (\Delta_{K_\infty, p} (j))^\#  \to \Lambda$ 
is a $\Lambda$-bilinear form. 

Let $Q$ be the total fractional ring of $\Lambda$.  
By the local duality theorem, 
this induces an isomorphism 
\begin{align}\label{LTH1}
\bigwedge_{Q}^{r_k}
Q \otimes_{\Lambda} H^1 (\Delta_{K_\infty, p} (1-j)) 
&\simeq \Hom_{Q} ( \bigwedge_{Q}^{r_k}
Q \otimes_{\Lambda} H^1 (\Delta_{K_\infty, p} (j)), Q)^\#, \\
\wedge_{a=1}^{r_k} x_a &\mapsto \parenth{\wedge_{b=1}^{r_k}y_b 
\mapsto \det (\langle x_a , y_b \rangle_{\LT})_{1 \leq a, b \leq r_k}}. 
\end{align}

Now the map 
$\LT_{K_\infty / k, p}^{\Delta(1-j)}: Q\otimes_{\Lambda} \Det_{\Lambda}^{-1} (\Delta_{K_\infty, p} (1-j)) 
\simeq Q \otimes_{\Lambda} \Det_{\Lambda}^{-1} (\Delta_{K_\infty, p} (1-j))^{\#, \vee} $ is equal to the composite map 
\begin{align*}
Q \otimes_{\Lambda} \Det_{\Lambda}^{-1} (\Delta_{K_\infty, p} (1-j)) 
\overset{\eqref{Iwasawa cohomology}}{\simeq} 
\Det_{Q} (Q \otimes_{\Lambda} H^1 (\Delta_{K_\infty, p} (1-j))) \\
\overset{\eqref{LTH1}}{\simeq}
\Det_{Q} (Q \otimes_{\Lambda} H^1 (\Delta_{K_\infty, p} (j)))^{\#, \vee} 
\overset{\eqref{Iwasawa cohomology}}{\simeq}
Q \otimes_{\Lambda} \Det_{\Lambda}^{-1} (\Delta_{K_\infty, p} (j))^{\#, \vee} .
\end{align*}
On the other hand, by 
\cite[Theorem II.16]{Ber03}, 
we have a commutative diagram 
\[\xymatrix@C=36pt{
\left(Q \otimes H^1 (\Delta_{K_\infty, p} (1-j))\right) \times 
\left(Q \otimes H^1 (\Delta_{K_\infty, p} (j))^\# \right)
\ar[d]_-{\eqref{Iwasawa cohomology} \text{ and }\eqref{Kummer seq of infty}}^-{\simeq}
\ar[r]^-{\langle\; , \; \rangle_{\LT}} 
& Q \ar@{=}[dd]\\
\left(Q \otimes \bigoplus_{v \in S_p} U_{K_\infty, v} (-j)\right) \times 
\left(Q \otimes \bigoplus_{v \in S_p} U_{K_\infty, v} (j-1)^\# \right)
 \ar[d]_-{\left(\oplus_{v \in S_p}\Colman_{K_\infty, v}^{1-j}\right) \times \left(\oplus_{v \in S_p} \Colman_{K_\infty, v}^{j}\right)} ^-{\simeq}
& \\    
\left(Q \otimes_{\Lambda} R_{K_p}\right) \times 
\left(Q \otimes_{\Lambda} R_{K_p}^\#\right) \ar[r]^-{(-1)^{j} \cdot\langle\; , \; \rangle_{R_{K_p}}} 
& Q. 
}
\]
Here we note that the map $\Colman_{K_\infty, v}^j$ (resp. $ \Colman_{K_\infty, v}^{1-j}$ ) coincides with the inverse of the isomorphism $\Omega_{\Q_p (j), j}$ (resp. $\Omega_{\Q_p (1-j), 1-j}$) in \cite{Ber03} in our semi-local setting.
The proposition follows from this diagram. 
\end{proof}

\begin{rem}
Berger \cite[Remark II.17]{Ber03} pointed out that the sign in the explicit reciprocity formulas in \cite[Th\'eor\`eme 4.2.3]{PR99} seems to be incorrect.
This is why we referred \cite[Theorem II.16]{Ber03} in this proof.
\end{rem}

At the beginning of \S \ref{sec:pf_lETNC}, 
we fixed a basis $\xi = (\zeta_{p^n})_n$ of $\Z_p (1) = \varprojlim_n \mu_{p^n} (K_\infty)$. 
For this $\xi$, we consider the twist map 
\[
 \Twist_{p, j, 1-j}^\xi := \otimes_{v \in S_p}  \Twist_{v, j, 1-j}^\xi 
 : 
  \Det_{\Lambda}^{-1} (\Delta_{K_\infty, p} (j)) 
  \xrightarrow{\sim}  \Det_{\Lambda}^{-1} (\Delta_{K_\infty, p} (1-j)). 
\]
We regard $\Z_p[\GG_K] \subset \Lambda$ by using the natural isomorphism
$\GG_K \overset{\sim}{\leftarrow} \Gal (K_\infty / k (\mu_{p^\infty})) \subset \GG_{K_\infty}$.
 
Using Proposition \ref{prop:ERL}, we show the following lemma.

\begin{lem}\label{lem:compute}
For $j \geq 1$, we have 
 \begin{align*}
 \Twist_{p, j, 1-j}^\xi (Z^{p,j}_{K_\infty / k, \xx}) = 
 & (-1)^{r_kj} \cdot (\sigma_{-1})^{r_k} \cdot
 |D_k| \cdot (-1)^{r_\C (k)} 
 \prod_{v \in S_{\ram} (K/k)_f} \rec_{k_v} (-1) (\sum_{\chi \in \widehat{\GG_K}} N (\ff_\chi) e_\chi )
 \\
& \times 
( \LT_{K_\infty / k, p}^{\Delta (1-j)} )^{-1} \parenth{(Z^{p,j}_{K_\infty / k, \xx})^\ast }
\end{align*}
in $ \Det_{\Lambda}^{-1} (\Delta_{K_\infty, p} (1-j)) $, 
where
$(-)^\ast$ means the $\Lambda$-dual basis and 
$\rec_{k_v} : k_v^\times \to \GG_{K_\infty}$ is defined in Definition \ref{defi:D}. 
  \end{lem}
 
 \begin{proof}
 By Proposition \ref{prop:extra_p}, 
 we have a commutative diagram
\[\xymatrix@C=36pt{
 \Det_{\Lambda}^{-1} (\Delta_{K_\infty, p} (j)) \ar[r]^-{\Twist_{p, j, 1-j}^{\xi}}_-{\sim} \ar[d]^-{\simeq}_-{\Phi_{K_\infty, p}^j}
&
\Det_{\Lambda}^{-1} (\Delta_{K_\infty, p} (1-j))  
\ar[d]^-{\simeq}_-{\Phi_{K_\infty, p}^{1-j}} \\
 \Det_{\Lambda} (R_{K_p}) \ar[r]^-{\wedge D^{2j-1}}_-{\sim}
 &
 \Det_{\Lambda} (R_{K_p}).
 }\]
Since $D (x (1+T)) = x (1+T)$ for any $x \in \OO_{K_p}$, we have 
\begin{equation}\label{d}
\Twist_{p, j, 1-j}^\xi (Z^{p,j}_{K_\infty / k, \xx}) = (\Phi_{K_\infty, p}^{1-j})^{-1} (\wedge_{i=1}^{r_k} x_{K, i} (1+T))
\end{equation}
in $\Det_{\Lambda}^{-1} (\Delta_{K_\infty, p} (1-j))$ by the definition of $Z^{p,j}_{K_\infty / k, \xx}$. 

Let us compute the image of $\wedge_{i=1}^{r_k} x_{K, i} (1+T)$ by the map \eqref{RLpairing}.
We have
\begin{align*}
& \det \parenth{\langle x_{K, a}(1+T), x_{K, b} (1+T) \rangle_{R_{K_p}}}_{1\leq a, b \leq r_k} \\
& \quad =
\det \parenth{
\sum_{\tau \in \Gal (K_\infty / k (\mu_{p^\infty}))} 
\Tr_{K / \Q} \parenth{ \langle x_{K, a}^{\tau^{-1}} (1+T) ,  x_{K, b} \sigma_{-1}(1+T) \rangle } \tau
}_{1\leq a, b \leq r_F} \\
&\quad =
(\sigma_{-1})^{r_k} A_{\xx},
\end{align*}
where we put 
\[
A_{\xx} := \det \parenth{ \sum_{\tau \in \Gal (K_\infty / k(\mu_{p^\infty}))} \Tr_{K/ \Q} ( x_{K, a}^{\tau^{-1}} x_{K,b})\tau}_{1 \leq a, b \leq r_k} \in \Lambda^\times.
\]
Explicitly, by Proposition \ref{prop: det B and Gauss sum}, we have 
 \[
 A_{\xx} =
 |D_k| \cdot (-1)^{r_\C (k)}
 \prod_{v \in S_{\ram} (K/k)_f} \rec_{k_v} (-1) 
\sum_{\chi \in \widehat{\GG_K}} 
N(\ff_\chi) e_\chi.
\]
Therefore, the map \eqref{RLpairing} sends $\wedge_{i=1}^{r_k} x_{K, i} (1+T)$ to
\[
(\sigma_{-1})^{r_k} A_{\xx}^\# \parenth{\wedge_{i=1}^{r_k} x_{K, i} (1+T)}^\ast
\in \Det_{\Lambda} (R_{K_p})^{\#, \vee}.
\]

Now Proposition \ref{prop:ERL} implies that 
\begin{align*}
\LT_{K_\infty / k, p}^{\Delta (1-j)} \circ \Twist_{p, j, 1-j}^\xi (Z^{p,j}_{K_\infty / k, \xx}) 
&\overset{\eqref{d}}{=} 
\LT_{K_\infty / k, p}^{\Delta (1-j)} \circ (\Phi_{K_\infty, p}^{1-j})^{-1} (\wedge_{i=1}^{r_k} x_{K, i} (1+T))
\\
 &=
(-1)^{r_kj} (\sigma_{-1})^{r_k} A_{\xx}^\# \cdot (\Phi_{K_\infty, p}^{j})^\ast \parenth{ (\wedge_{i=1}^{r_k} x_{K, i} (1+T))^\ast} 
\\
 &= (-1)^{r_kj}(\sigma_{-1})^{r_k} A_{\xx}^\# \cdot \parenth{ (\Phi_{K_\infty, p}^{j})^{-1} (\wedge_{i=1}^{r_k} x_{K,i} (1+T))}^\ast
 \\
 &=  (-1)^{r_kj}(\sigma_{-1})^{r_k} A_{\xx}^\# \cdot   (Z^{p,j}_{K_\infty / k, \xx} )^\ast
\end{align*}
in $\Det_{\Lambda}^{-1} (\Delta_{K_\infty, p} (j))^{\#, \vee}$. 
Noting that $A_{\xx} = A_{\xx}^\#$, we obtain the lemma.
\end{proof}

\subsubsection{Non-$p$-adic primes}

Let $v$ be a non-$p$-adic finite prime of $k$ in $S$.
We write 
\[
\LT_{K_\infty / k, v}^{\Delta (1-j)} : 
\Det_{\Lambda}^{-1} (\Delta_{K_\infty, v} (1-j)) \simeq \Det_{\Lambda}^{-1} (\Delta_{L_\infty} (j))^{\#, \vee}
= \Det_{\Lambda} (\Delta_{L_\infty} (j))^{\#}
\]
for the $v$-component of the isomorphism $\LT_{K_\infty / k, S}^{\Xi (1-j)}$. 
We use the same notations as in \S \ref{sec:non p adic prime}. 
\begin{lem}\label{lem:LT_l-adic} 
For $j \geq 1$, we have 
\[
\Twist_{v, j, 1-j} (\cH_{K_\infty /k, v}^j) =\parenth{1-e_{I_{K_\infty}, v} + e_{I_{K_\infty}, v} \cdot N(v)} \cdot 
 (\LT_{K_\infty / k, v}^{\Delta (1-j)})^{-1} \parenth{ (\cH_{K_\infty /k, v}^j)^\ast }
\]
in $\Det_{\Lambda}^{-1} (\Delta_{K_\infty, v} (1-j))$, 
where $(-)^\ast$ means the $\Lambda$-dual basis and $I_{K_\infty, v} \subset \GG_{K_\infty}$ is the inertia group of $v$. 
\end{lem}
\begin{proof}
We consider the commutative diagram
\[
\xymatrix@C=50pt{
Q \otimes \Det_{\Lambda} (\Delta_{K_\infty, v} (j))^\#  
\ar[r]^-{\eqref{canonical det}}_-{\simeq} \ar[d]_-{\simeq}^-{(\LT_{K_\infty / k, v}^{\Delta (1-j)})^{-1}}
&Q^\# \ar[r]^-{\#}
&Q \ar@{=}[d]
\\
Q \otimes \Det_{\Lambda}^{-1} (\Delta_{K_\infty, v} (1-j))
\ar[rr]^-{\eqref{canonical det}}_-{\simeq} 
& 
&Q 
\\
Q \otimes \Det_{\Lambda}^{-1} (\Delta_{K_\infty, v} (j))  \ar@{-->}[rru]
\ar[rr]^-{\eqref{canonical det}}_-{\simeq} \ar[u]^-{\simeq}_-{\Twist_{v, j, 1-j}}
& 
&Q \ar[u]_-{\twist_{2j-1}}^-{\simeq}, 
}
\]
where the lower vertical arrows are both $\twist_{2j-1}$-semilinear. 
By definition, the upper horizontal composite arrow sends $(\cH_{K_\infty /k, v}^j)^\ast$ to $(h_{K_\infty / k, v}^{j, \#})^{-1}$ and the dotted arrow sends $\cH_{K_\infty /k, v}^j$ to 
$\twist_{2j-1} ((h_{K_\infty / k, v}^{j}))= h_{K_\infty / k, v}^{1-j}$.  
Therefore,
\begin{align*}
\Twist_{v, j, 1-j} (\cH_{K_\infty /k, v}^j) &=  (h_{K_\infty / k, v}^j)^{\#} \cdot h_{K_\infty / k, v}^{1-j}
\cdot (\LT_{K_\infty / k, v}^{\Delta (1-j)})^{-1} \parenth{ (\cH_{K_\infty /k, v}^j)^\ast } \\
&= 
\parenth{1-e_{I_{K_\infty}, v} + e_{I_{K_\infty}, v} \cdot N(v)} \cdot 
 (\LT_{K_\infty / k, v}^{\Delta (1-j)})^{-1} \parenth{ (\cH_{K_\infty /k, v}^j)^\ast }
\end{align*}
in $\Det^{-1}_{\Lambda} (\Delta_{K_\infty, v} (1-j))$. 
\end{proof}

\begin{lem}\label{lem:unr_basis_2}
Assume that $K/k$ is unramified at $v$. 
Then, for $j \geq 1$,
we have 
\[\Twist_{v, j, 1-j} (\cE_{K_\infty /k, v}^j) =(\LT_{K_\infty / k, v}^{\Delta (1-j)})^{-1} \parenth{ (\cE_{K_\infty /k, v}^j)^\ast }
\] 
in  
$\Det_{\Lambda}^{-1} (\Delta_{K_\infty, v} (1-j))$. 
\end{lem}

\begin{proof}
By definition (see the proof of Corollary \ref{cor:unr_basis}), we have 
\begin{align*}
\Twist_{v, j, 1-j} (\cE_{K_\infty /k, v}^j) 
&=
 \Twist_{v, j, 1-j} (- N(v)^{j-1} \sigma_{K_\infty / k, v}^{-1} \cdot \cH_{K_\infty /k, v}^j)   
\\
&= - N(v)^{-j} \sigma_{K_\infty / k, v}^{-1} \cdot  \Twist_{v, j, 1-j} (\cH_{K_\infty /k, v}^j)  
\\
&= - N(v)^{-j+1} \sigma_{K_\infty / k, v}^{-1} 
(\LT_{K_\infty / k, v}^{\Delta (1-j)})^{-1} \parenth{ (\cH_{K_\infty /k, v}^j)^\ast }
\\
&= (\LT_{K_\infty / k, v}^{\Delta (1-j)})^{-1} \parenth{ (\cE_{K_\infty /k, v}^j)^\ast }
\end{align*}
where the third equality follows from Lemma \ref{lem:LT_l-adic}. 
\end{proof}

\subsubsection{Putting all together}\label{ss:extra}

Now we complete the proof of Proposition \ref{prop:extra_Z}.
As already discussed, it remains to show claim (iii). 

For $j \in \Z_{\geq 1}$, we decompose the element $Z_{K_\infty/k , S}^{\loc, j} \in \Xi_{K_\infty / k, S}^{\loc} (j)$ in Definition \ref{def:Z} as $Z_{K_\infty/k , S}^{\loc, j} = A_j \cdot (Z^{\fin}_j \otimes E_j)$, where
\[
A_j :=  D_{j-1}(K_\infty / k) \prod_{v \in S_{\ram} (K/k)_f} \ff_{j-1} (K_\infty / k)_v \in \Lambda^\times,
\]
\[
Z_j^{\fin} := Z_{K_\infty /k, \xx}^{p, j} \otimes \bigotimes_{v \in S_{\ram} (K/k)_f} \cH_{K_{\infty}/k, v}^{ j}
\otimes \bigotimes_{v \in S_f \setminus (S_{\ram} (K/k)_f \cup S_p) } \cE_{K_{\infty}/k, v}^{ j} 
 \in \bigotimes_{v \in S_f} \Det_{\Lambda}^{-1} (\Delta_{K_\infty, v} (j)),
\]
and
\[
E_j := (\wedge_{i=1}^{r_k} e_{K_\infty, i}^j )^\ast \in \Det_{\Lambda}^{-1} (X_{K_\infty} (j)).
\]

By Definition \ref{Z<1}, claim (iii) can be restated as
\[
\Twist_{j, 1-j}^{\loc} (A_j \cdot (Z^{\fin}_j \otimes E_j)) 
=  (-1)^{(j-1)r_\R (k) + r_\C (k)} \cdot c_{K_\infty / k}^\infty  \cdot 
(\LT_{K_\infty / k, S}^{\Xi(1-j)})^{-1} \parenth{ (A_j \cdot (Z^{\fin}_j \otimes E_j) )^\ast}. 
\]
We can check this as follows. 
We write 
$\Twist_{j, 1-j}^{\loc} = \Twist_{\fin, j, 1-j}^{\xi} \otimes \Twist_{X_{K_\infty},j, 1-j}^{\xi}$, 
where $\Twist_{\fin, j, 1-j}^{\xi} $ denotes the finite prime components and 
$\Twist_{X_{K_\infty},j, 1-j}^{\xi}$ the $X_{K_\infty} (j)$-component. 
We also write 
\[
\LT_{K_\infty / k}^{\fin} : 
\bigotimes_{v \in S_f} \Det_{\Lambda}^{-1} (\Delta_{K_\infty, v} (1-j)) 
\xrightarrow{\sim}
\bigotimes_{v \in S_f} \Det_{\Lambda}^{-1} (\Delta_{K_\infty, v} (j))^{\#, \vee}
\]
for the finite prime components of the isomorphism $\LT_{K_\infty / k, S}^{\Xi(1-j)}$. 
Then, by Lemmas \ref{lem:compute}, \ref{lem:LT_l-adic} and \ref{lem:unr_basis_2}, 
we have
\begin{align*}\label{finite_part}
\Twist_{\fin, j, 1-j}^{\xi}  (Z_j^{\fin}) 
&=
(-1)^{jr_k} \cdot (\sigma_{-1})^{r_k} \cdot
|D_k| \cdot (-1)^{r_\C (k)}
 \prod_{v \in S_{\ram} (K/k)_f} \rec_{k_v} (-1) 
\sum_{\chi \in \widehat{\GG_K}} 
N(\ff_\chi) e_\chi
\\
&\times \prod_{ v \in S_{\ram} (K/k)_f} \parenth{1-e_{I_{K_\infty, v}} + e_{I_{K_\infty ,v}} N(v) } 
(\LT_{K_\infty / k}^{\fin})^{-1} ((Z_j^{\fin})^\ast). 
\end{align*}
Furthermore, we have $\Twist_{X_{K_\infty},j, 1-j}^{\xi} (E_j) = E_{1-j}$
 and 
\begin{align*}
\twist_{2j-1} (A_j) \cdot A_j^\# = |D_k|^{-1} \cdot \prod_{ v \in S_{\ram} (K/k)_f} \parenth{1-e_{I_{K_\infty, v}} + 
e_{I_{K_\infty,v}} N(v)^{-1} } 
\sum_{\chi \in \widehat{\GG_K}}N(\ff_\chi)^{-1} e_\chi. 
\end{align*}
Combining these, we obtain
\begin{align*}
&\Twist_{j, 1-j}^{\loc} (A_j \cdot (Z^{\fin}_j \otimes E_j))
= \twist_{2j-1} (A_j) \cdot \Twist_{j, 1-j}^{\loc} ( Z^{\fin}_j \otimes E_j) 
\\
&= \parenth{ \twist_{2j-1} (A_j) \cdot A_j^\#}  A_j^{\#, -1} \cdot 
\parenth{\Twist_{\fin, j, 1-j}^{\xi}  (Z_j^{\fin}) \otimes \Twist_{X_{K_\infty},j, 1-j}^{\xi} (E_j) }
\\
&=  (-1)^{j r_k + r_\C (k)} (\sigma_{-1})^{r_k}  \prod_{v \in S_{\ram} (K/k)_f} \rec_{k_v} (-1) \cdot 
 A_j^{\#, -1} \cdot 
\parenth{
(\LT_{K_\infty / k}^{\fin})^{-1} ((Z_j^{\fin})^\ast) \otimes E_{1-j} }. 
\end{align*}
On the other hand, we have
\begin{align*}
& (\LT_{K_\infty / k, S}^{\Xi(1-j)} )^{-1} \parenth{ (A_j \cdot (Z^{\fin}_j \otimes E_j))^\ast} 
= 
A_j^{\#, -1} \cdot  (\LT_{K_\infty / k, S}^{\Xi(1-j)} )^{-1} \parenth{(Z^{\fin}_j \otimes E_j)^\ast}  
\\
&\quad= 
A_j^{\#, -1} \cdot  (\LT_{K_\infty / k, S}^{\Xi(1-j)} )^{-1} \parenth{ E_j^\ast \otimes (Z^{\fin}_j)^\ast}  
= 
A_j^{\#, -1} \cdot \parenth{ E_{1-j} \otimes (\LT_{K_\infty / k}^{\fin})^{-1} ((Z_j^{\fin})^\ast) }
\\
&\quad= 
(-1)^{r_k} \cdot A_j^{\#, -1} \cdot \parenth{ (\LT_{K_\infty / k}^{\fin})^{-1} ((Z_j^{\fin})^\ast)  \otimes  E_{1-j}}.
\end{align*}
Here, the sign $(-1)^{r_k}$ appears 
since the grade of $\otimes_{v \in S_f} \Det_{\Lambda}^{-1} (\Delta_{K_\infty, v} (1-j))$ is $r_k$ and that of $\Det_{\Lambda}^{-1} (X_{K_\infty} (1-j))$ is $-r_k$.

It remains to check that
$(\sigma_{-1})^{r_k}  \prod_{v \in S_{\ram} (K/k)_f} \rec_{k_v} (-1) = c_{K_\infty / k}^\infty. $
We consider the commutative diagram 
\[\xymatrix{
k_v^\times \ar[r]^-{\rec_{k_v}} \ar[d]_-{N_{k_v / \Q_p}} 
&
\Gal (K_w (\mu_{p^\infty}) / k_v) \ar@{^{(}->}[r] \ar[d]^-{\res}
&
\Gal (K_\infty / k) 
\\
\Q_p^\times \ar[r]^-{\rec_{\Q_p}} 
&
\Gal (\Q_p (\mu_{p^\infty}) / \Q_p) \ar[r]_-{\sim}^-{f}
& 
\Gal (K_\infty / K) \ar@{^{(}->}[u]
}
\]
for $v \in S_p $ and a prime $w$ of $K$ lying above $v$. 
Here, $\rec_{\Q_p}$ and $\rec_{k_v}$ denote the reciprocity maps, $N_{k_v / \Q_p}$ the norm map, and $f$ the inverse of $\Gal (K_\infty / K) \simeq \Gal (K_w (\mu_{p^\infty})/ K_w) 
\overset{\res}{\xrightarrow{\sim}} \Gal (\Q_p (\mu_{p^\infty}) / \Q_p) $. 
The composite map of the bottom arrows sends $-1$ to $\sigma_{-1}$. 
Therefore, we have
\[
(\sigma_{-1})^{r_k} =  \prod_{v \in S_p } f \circ \rec_{\Q_p} (N_{k_v / \Q_p} (-1))
= \prod_{v \in S_p} \rec_{k_v} (-1). 
\]
From this and the reciprocity law of the global class field theory, we deduce that 
\[
(\sigma_{-1})^{r_k}  \prod_{v \in S_{\ram} (K/k)_f} \rec_{k_v} (-1) =
\prod_{v \in S_{\ram} (K_\infty /k)_f} \rec_{k_v} (-1) = c_{K_\infty / k}^\infty. 
\]
This completes the proof of Proposition \ref{prop:extra_Z}. 
\qed

\appendix

\section{Determinant modules}\label{App:det}

In this section, we review the theory of the determinant modules, following Knudsen--Mumford \cite{KM76}.

\subsection{Definitions}\label{App:det_defn}

\subsubsection{Category of graded invertible modules}\label{App.ev}

(See \cite[page 20]{KM76}.)
Let $R$ be a commutative ring and $\mathcal{P}_R$ be the category of graded invertible $R$-modules.
Objects in $\mathcal{P}_R$ are pairs $(L, r)$ where $L$ is an invertible $R$-module and $r : \Spec(R) \to \Z$ is a
locally constant function.
For two objects $(L,r)$ and $(M,s)$ in $\mathcal{P}_R$, the morphisms from $(L,r)$ to $(M,s)$ are 
isomorphisms $L \xrightarrow{\sim} M$ as $R$-module if $r = s$, and there are no morphisms otherwise. 

The category $\mathcal{P}_R$ is equipped with the tensor product defined by $( L , r ) \otimes_R (M,s) := (L \otimes_R M, r+ s)$.
The associativity holds naturally, and the unit object is ${\trivial}_R := (R, 0)$.
The commutativity holds by the isomorphism
\begin{equation}\label{switch}
\psi: (L,r) \otimes_{R} (M,s) \xrightarrow{\sim} (M,s) \otimes_{R} (L,r),
\quad
l \otimes m \mapsto (-1)^{rs} m \otimes l
\end{equation}
(the last formula holds Zariski locally). For any objects $( L , r )$ and $(M,s)$ in $\mathcal{P}_R$, we always identify $(L,r) \otimes_{R} (M,s)$ with $(M,s) \otimes_{R} (L,r)$ via the above isomorphism $\psi$.

For each object $(L,r)$ of $\mathcal{P}_R$, we define
 $(L,r)^{-1} :=(L^\ast,-r)$ with $L^\ast := \Hom_R(L,R)$. We can regard the object $(L,r)^{-1}$ as a right inverse of $(L,r)$ by the evaluation map
 \[
 \ev_{(L, r)} : (L,r) \otimes_R (L,r)^{-1} \overset{\sim}{\to} (R, 0) = {\trivial}_R,
 \quad
 x \otimes f \mapsto f(x). 
 \]
The object $(L,r)^{-1}$ also can be regarded as a left inverse of $(L,r)$ by the isomorphism
 \[
(L, r)^{-1} \otimes_R (L, r) 
\overset{\psi}{\to} 
(L, r) \otimes_R (L, r)^{-1}
\overset{\ev_{(L, r)}}{\to} {\trivial}_R.
 \]
 Note that this left inverse differs from 
\[
\ev_{(L, r)^{-1}} : (L,r)^{-1} \otimes_R (L,r) \overset{\sim}{\to} {\trivial}_R,
 \quad
 f \otimes x \mapsto f(x)
\]
by the sign $(-1)^r$ when we identify $((L, r)^{-1})^{-1} \simeq (L, r)$. The construction of the determinant functors below depends on this choice of the inverse (\cite[page 30]{KM76}).

Throughout this paper, we often write $\ev$ for the evaluation map $ \ev_{(L, r)}$ 
for simplicity. 
When we change the order (i.e., use the isomorphism $\psi$) before applying the evaluation map, 
we always write it explicitly.

\begin{lem}[{cf.~\cite[page 31]{KM76}}]\label{lem:prodinv}
Let $(L, r)$ and $(M, s)$ be objects of $\PP_R$.
The natural isomorphism $(L \otimes M)^* \simeq M^* \otimes L^*$ yields an isomorphism
\[
\theta: ((L, r) \otimes (M, s))^{-1} \simeq (M, s)^{-1} \otimes (L, r)^{-1},
\]
which gives us a commutative diagram
\[
\xymatrix{
(L, r) \otimes (M, s) \otimes (M, s)^{-1} \otimes (L, r)^{-1} \ar[r]^-{\id \otimes \theta}_-{\sim} \ar[d]^-{\sim}_-{\ev_{(M, s)}}
 & ((L, r) \otimes (M, s)) \otimes ((L, r) \otimes (M, s))^{-1}
 \ar[d]_-{\sim}^-{\ev_{(L, r) \otimes (M, s)}}
\\
(L, r) \otimes (L, r)^{-1} \ar[r]^-{\sim}_-{\ev_{(L, r)}}
&
{\trivial}_R. 
}
\]
\end{lem}

\subsubsection{Determinants of modules}

(See \cite[pages 20--21]{KM76}.)
For a finitely generated projective $R$-module $P$, writing $\rank_R(P)$ for its (locally constant) rank, we define its determinant module by
\[
\Det_R(P) := \left( \bigwedge^{\rank_R(P)}_R P , \rank_R(P) \right) \in \mathcal{P}_R.
\]
We write $\Det_R^{-1}(P) := \Det_R(P)^{-1}$ for the inverse of $\Det_R(P)$ in $\mathcal{P}_R$.  

\begin{prop}\label{prop:det_ses}
Let $0 \to P_1 \xrightarrow{f} P_2 \xrightarrow{g} P_3 \to 0$ be an exact sequence of finitely generated projective $R$-modules.
Set $r_i = \rank_R(P_i)$ for $i = 1, 2, 3$.
Then we have a canonical isomorphism
\[
 \Det_R(P_1) \otimes_R \Det_R(P_3)  \simeq \Det_R(P_2) 
 \]
defined locally by
 \begin{align*}
 x_1 \wedge \cdots \wedge x_{r_1} \otimes 
 z_{1} \wedge \cdots \wedge z_{r_3}  
\mapsto 
f(x_1) \wedge \cdots \wedge f(x_{r_1}) \wedge  
 \widetilde{z_{1}} \wedge \cdots \wedge \widetilde{z_{r_3}},
\end{align*}
where $ x_1 , \dots , x_{r_1} \in P_1$, $ z_{1} , \dots , z_{ r_3} \in P_3$, and $\wtil{z_1}, \dots, \wtil{z_{r_3}} \in P_2$ are lifts of $z_1, \dots, z_{r_3}$ with respect to $g$.
\end{prop}
 
\subsubsection{Determinants of complexes}

(See \cite[pages 31--35]{KM76}.)
Let $C$ be a perfect complex of $R$-modules.
This means that $C$ is a bounded complex
\[
C= [ \cdots \to C^i \to C^{i+1} \to \cdots]
\]
such that each $C^i$ is a finitely generated projective $R$-module.
Then the determinant module of $C$ is defined by 
\[
 \Det_R(C) := \bigotimes_{i \in \Z} \Det_R^{(-1)^i} (C^i)  \in \mathcal{P}_R.
\]

Proposition \ref{prop:det_ses} implies the following.
 
 \begin{prop}\label{prop:det_ses2}
Let $0 \to C_1 \to C_2 \to C_3 \to 0$ be an exact sequence of perfect complexes over $R$.
Then we have a canonical isomorphism
\[
 \Det_R(C_1) \otimes_R \Det_R(C_3)  \simeq \Det_R(C_2).
 \]
\end{prop}

We have canonical trivializations of the determinants of acyclic complexes.

\begin{prop}[{cf.~\cite[page 32--37]{KM76}}]\label{prop:det_ses3}
Let $C$ be an acyclic complex of $R$-modules. Then we have a canonical isomorphism
\[
\Det_{R} (C) \simeq {\trivial}_R. 
\]
\end{prop}

\begin{prop}[{cf.~\cite[page 43]{KM76}}]\label{prop:det_ses4}
Let $C$ be a perfect complex over $R$.
We assume that each $H^i(C)$ is also finitely generated projective over $R$.
Then we have a canonical isomorphism
\[
\Det_{R} (C) \simeq \bigotimes_{i  \in \Z} \Det_R^{(-1)^i} (H^i (C)).
\]
\end{prop}

\subsubsection{Determinants in the derived category}

(See \cite[pages 37--43]{KM76}.)
Let $D^{\perf} (R)$ be the derived category of perfect complexes of $R$-modules.
The morphisms are defined so that any quasi-isomorphisms between complexes are actually isomorphisms.

We write $D^{\perf} (R)_{\isom}$ for the category whose objects are the same as $D^{\perf}(R)$ and the morphisms are restricted to isomorphisms.
Then the determinant can be extend to a functor 
\[ 
D^{\perf} (R)_{\isom} \to \mathcal{P}_R,
\quad
C \mapsto \Det_R(C)
\]
which is also called the determinant functor (\cite[page 42, Theorem 2]{KM76}). 

\subsubsection{Base changes}
 
Let $f : R \to R'$ be a ring homomorphism.
We have the base change functor 
\[
R' \otimes_R (-) : \mathcal{P}_{R} \to \mathcal{P}_{R'}, 
\quad
(L,r) \mapsto (R' \otimes_R L, r \circ f^* ),
\]
where $f^* : \Spec(R') \to \Spec(R)$ is induced by $f$.
For a finitely generated projective $R$-module $P$ 
(resp. a perfect complex $C$ of $R$-modules), 
we have a canonical isomorphism 
\[
R' \otimes_R \Det_{R} (P) \simeq \Det_{R'} (R' \otimes_R P) 
\quad
(\text{resp. }  
R' \otimes_R \Det_{R} (C) \simeq \Det_{R'} (R' \otimesL_R C) 
).
\]

\subsubsection{Linear duals}

For $(L, r) \in \mathcal{P}_R$, we define $(L, r)^\vee := (L^\ast , r)$ (recall $(L, r)^{-1} = (L^*, -r)$).
Then, for any finitely generated projective $R$-module $P$, 
we have an isomorphism
\[
\Det_R (\Hom_R(P, R)) \xrightarrow{\sim} \Det_R (P)^\vee
\]
defined by 
\[
f_1 \wedge \cdots \wedge f_r \mapsto \parenth{
x_1 \wedge \cdots \wedge x_r \mapsto \det (f_i(x_j))_{i, j}
},
\]
where $\det$ denotes the determinants of matrices. 
Similarly, for any $C \in D^{\perf} (R)$, we also have an isomorphism
\[
\Det_R (\RHom_{R} (C, R)) \simeq \Det_R (C)^\vee .
\]

\subsection{Propositions}\label{App:det_prop}

In this subsection, we introduce some propositions on the determinant modules, which are used in \S \ref{chi = 1 and j=1}, \S \ref{ss:interp_l}. For integers $a \leq b$, we write $D^{[a, b]}(R)$ for the subcategory of $D^{\perf} (R)$ that consists of complexes that are quasi-isomorphic to $[0 \to C^a \to C^{a+1} \to \dots \to C^b \to 0]$ with $C^a, C^{a+1}, \dots, C^b$ finitely generated projective.

We begin with an easy lemma. Let $k$ be a field. Suppose we have an exact sequence of $k$-vector spaces
\[
0 \to A \overset{f}{\to} A_1 \overset{g}{\to} A_2 \overset{h}{\to} A \to 0.
\]
Then we have a canonical isomorphism
\[
\Det_k(A_1) \simeq \Det_k(A_2),
\]
which is induced by Proposition \ref{prop:det_ses3}. 
More precisely, we have short exact sequences $0 \to A \overset{f}{\to} A_1 \overset{g}{\to} B \to 0$ and $0 \to B \to A_2 \overset{h}{\to} A \to 0$ with $B:=\mathrm{Im}(g)$. These induce the above isomorphism as 
\[
\Det_k(A_1) 
\simeq \Det_k(A) \otimes \Det_k(B) 
\overset{\psi}{\simeq} \Det_k(B) \otimes \Det_k(A) 
\simeq \Det_k(A_2).
\]
This isomorphism can be described explicitly as follows. 
\begin{lem}\label{lem:Det01}
In the above situation, set $r = \dim_k A$ and $r + s = \dim_k A_1 = \dim_k A_2$.
Let us take a basis $a_1, \dots, a_r, \wtil{b_1}, \dots, \wtil{b_s}$ of $A_1$ and a basis $\wtil{a_1}, \dots, \wtil{a_r}, b_1, \dots, b_s$  of $A_2$ such that $f \circ h (\wtil{a_i}) = a_i$ for $1 \leq i \leq r$ and $g(\wtil{b_j}) = b_j$ for $1 \leq j \leq s$.
Then the above isomorphism
\[
\Det_k(A_1) \simeq \Det_k(A_2)
\]
sends the basis $a_1 \wedge \dots \wedge a_r \wedge \wtil{b_1} \wedge \dots \wedge \wtil{b_s}$ to the basis $\wtil{a_1} \wedge \dots \wedge \wtil{a_r} \wedge b_1 \wedge \dots \wedge b_s$.
\end{lem}

\begin{proof}
By the assumption, we see that $h(\wtil{a_1}), \dots, h(\wtil{a_r})$ is a basis of $A$ and $b_1, \dots, b_s$ is a basis of $B=\mathrm{Im}(g)$.
With respect to these basis, the image of $a_1 \wedge \dots \wedge a_r \wedge \wtil{b_1} \wedge \dots \wedge \wtil{b_s} \in \Det_k(A)$ under the isomorphism $\Det_k(A_1) \simeq \Det_k(A_2)$ are computed as
\begin{align*}
a_1 \wedge \dots \wedge a_r \wedge \wtil{b_1} \wedge \dots \wedge \wtil{b_s}
& \mapsto
(h(\wtil{a_1}) \wedge \dots \wedge h(\wtil{a_r})) \otimes (b_1 \wedge \dots \wedge b_s)\\
& \overset{\psi}{\mapsto}
(-1)^{rs} (b_1 \wedge \dots \wedge b_s) \otimes (h(\wtil{a_1}) \wedge \dots \wedge h(\wtil{a_r}))\\
& \mapsto 
(-1)^{rs} b_1 \wedge \dots \wedge b_s \wedge \wtil{a_1} \wedge \dots \wedge \wtil{a_r}\\
& = \wtil{a_1} \wedge \dots \wedge \wtil{a_r} \wedge b_1 \wedge \dots \wedge b_s.
\end{align*}
\end{proof}

To state Propositions \ref{prop:DetA} and \ref{prop:Det02} below, we now focus on a specific setting.
Let $R$ be a DVR with a uniformizer $\omega$ and let $k = R/(\omega)$ be its residue field.
For a finitely generated free $R$-module $M$, we write $\ol{M} = k \otimes_R M$.
Similarly, for a complex $C \in D^{\perf}(R)$, we write $\ol{C} = k \otimesL_R C \in D^{\perf}(k)$.

\begin{defn}\label{defn:Bock}
Let $C \in D^{[1, 2]}(R)$ be a complex with $H^{2}(C)$ annihilated by $\omega$, that is, 
$H^{2}(C)$ is a $k$-vector space.
Let us define the Bockstein map 
\[
\beta_{\omega}: \Det^{-1}_k (\ol{C}) \simeq \Det_k (\ol{H^1(C)})
\]
as follows.
By the exact triangle $C \overset{\omega \times}{\to} C \to \ol{C}$, we have an exact sequence
\begin{equation}\label{eq:Bock_seq}
0 \to \ol{H^1(C)} \to H^1(\ol{C}) \to H^{2}(C)[\omega] \to 0.
\end{equation}
Since we have $H^{2}(C)[\omega] = H^{2}(C) = H^{2}(\ol{C})$ by the assumption, this induces the map $\beta_{\omega}$ as follows:
\begin{align*}
\Det^{-1}_k (\ol{C}) 
& \simeq \Det_k(H^1(\ol{C})) \otimes \Det_k^{-1}(H^2(\ol{C}))\\
& \overset{\eqref{eq:Bock_seq}}{\simeq} \Det_k(\ol{H^1(C)}) \otimes \Det_k(H^2(C)[\omega]) \otimes \Det_k^{-1}(H^2(C)[\omega])\\
& \overset{\text{ev}}{\simeq} \Det_k (\ol{H^1(C)}).
\end{align*}
\end{defn}

\begin{prop}\label{prop:DetA}
Let $C$ be as in Definition \ref{defn:Bock}. Putting $s = \dim_k H^{2}(C)$, we have a commutative diagram
\[
\xymatrix{
	\Det_R^{-1} (C) \ar@{->>}[d] \ar[r]_-{\omega^{- s} \times}^-{\simeq}
	& \Det_R(H^1(C)) \ar@{->>}[d]\\
	\Det_k^{-1} (\ol{C}) \ar[r]^-{\simeq}_-{\beta_{\omega}}
	& \Det_k (\ol{H^1(C)}).
}
\]
Both the vertical arrows are the natural base-change maps and the upper isomorphism is obtained by identifying $\Det_R^{-1} (C) = \omega^{s} \cdot \Det_R (H^1(C))$, 
which results from identification of $\Det_R^{-1}(H^2(C))$ with the Fitting ideal $\Fitt_R(H^2(C)) = \omega^s R$ of $H^2(C)$.
\end{prop}

\begin{proof}
Since $C \in D^{[1, 2]}(R)$, we can take a quasi-isomorphism $C \simeq [C^1 \overset{g}{\to} C^{2}]$ with $C^1$, $C^2$ finitely generated free over $R$. 
Then we also have $\ol{C} \simeq [\ol{C^1} \overset{\ol{g}}{\to} \ol{C^2}]$ and we obtain exact sequences
\[
0 \to H^1(C) \to C^1 \overset{g}{\to} C^{2} \to H^{2}(C) \to 0
\]
and
\[
0 \to H^1(\ol{C}) \to \ol{C^1} \overset{\ol{g}}{\to} \ol{C^{2}} \to H^{2}(\ol{C}) \to 0.
\]

Putting $r = \rank_R H^1(C)$ and $m = \rank_R C^{2}$, we have $\rank_R C^1= r + m$.
Let $X$ be the image of $g: C^1 \to C^2$. We take
\begin{itemize}
\item
a basis $\{x_1, \dots, x_r\}$ of $H^1(C)$ over $R$,
\item
a basis $y_1, \dots, y_m$ of $C^{2}$ over $R$ such that $\omega y_1, \dots, \omega y_s, y_{s+1}, \dots, y_m$ is a basis of $X$, and
\item
$\wtil{\omega y_1}, \dots, \wtil{\omega y_s}, \wtil{y_{s+1}}, \dots, \wtil{y_m} \in C^1$ as lifts of the elements without tildes.
\end{itemize}
Then 
\[
x_1, \dots, x_r, \wtil{\omega y_1}, \dots, \wtil{\omega y_s}, \wtil{y_{s+1}}, \dots, \wtil{y_m}
\]
forms a basis of $C^1$ over $R$.
By construction, $H^1(\ol{C})$ has a $R$-basis
\[
\ol{x_1}, \dots, \ol{x_r}, \ol{\wtil{\omega y_1}}, \dots, \ol{\wtil{\omega y_s}} \in \ol{C^1}.
\]
Also, the classes of $\ol{y_1}, \dots, \ol{y_s} \in \ol{C_2}$ form a basis of $H^2(\ol{C})$.

The module $\Det^{-1}_R (C) = \Det_R(C^1) \otimes \Det_R^{-1}(C^{2})$ has a basis
\[
\xi := (x_1 \wedge \dots \wedge x_r \wedge \wtil{\omega y_1} \wedge \dots \wedge \wtil{\omega y_s} \wedge \wtil{y_{s+1}} \wedge \dots \wedge \wtil{y_m})
\otimes (y_1 \wedge \dots \wedge y_m)^{\ast}.
\]
Since $\xi$ is identified with
\[
\omega^{s } \cdot (x_1 \wedge \dots \wedge x_r)
\]
in $\omega^{s} \Det_R (H^1(C))$, the upper horizontal arrow sends $\xi$ to
\[
x_1 \wedge \dots \wedge x_r
\]
in $\Det_R H^1(C)$.

On the other hand, by the map $H^1(\ol{C}) \to H^{2}(C)[\omega] = H^{2}(\ol{C})$ induced by $C \overset{\omega}{\to} C \to \ol{C}$, the element $\ol{\wtil{\omega y_j}}$ is sent to the class of $\ol{y_j}$ in $H^{2}(\ol{C})$.
Therefore, by the definition of $\beta_{\omega}$, the element
\[
\ol{\xi} = (\ol{x_1} \wedge \dots \wedge \ol{x_r} \wedge \ol{\wtil{\omega y_1}} \wedge \dots \wedge \ol{\wtil{\omega y_s}} \wedge \ol{\wtil{y_{s+1}}} \wedge \dots \wedge \ol{\wtil{y_m}})
\otimes (\ol{y_1} \wedge \dots \wedge \ol{y_m})^\ast
\]
is sent by $\beta_{\omega}$ to 
$ \ol{x_1} \wedge \cdots \wedge \ol{x_r}$.
This shows the proposition.
\end{proof}

Based on the situation in \S \ref{chi = 1 and j=1}, we consider the following setting.

\begin{setting}\label{setting:A1}
Let $C$ be as in Definition \ref{defn:Bock}.
Suppose that we are given a surjective homomorphism $\ol{\varphi}: H^1(\ol{C}) \to H^2(\ol{C})$ such that the composite map
\[
\varphi: H^1(C) \to H^1(\ol{C}) \overset{\ol{\varphi}}{\to} H^2(\ol{C}) \simeq H^2(C)
\]
is still surjective.

Let $H^1_f(C)$ and $H^1_f(\ol{C})$ be the kernel of $\varphi$ and $\ol{\varphi}$ respectively, so we have a commutative diagram
\begin{equation}\label{eq:Det02_1}
\xymatrix{
	0 \ar[r]
	& H^1_f(C) \ar[r] \ar[d]
	& H^1(C) \ar[r]^{\varphi} \ar[d]
	& H^2(C) \ar[r] \ar[d]^{\simeq}
	& 0\\
	0 \ar[r]
	& H^1_f(\ol{C}) \ar[r]
	& H^1(\ol{C}) \ar[r]_{\ol{\varphi}}
	& H^2(\ol{C}) \ar[r]
	& 0.
}
\end{equation}
By the bottom exact sequence, together with Propositions \ref{prop:det_ses2} and \ref{prop:det_ses4}, we have 
\begin{align*}
\Det_R^{-1}(C) 
& \simeq  \Det_R(H^1(C))\otimes_{R}\Det_R^{-1}(H^2(C))\\
& \simeq \Det_R(H^1_f(C))\otimes_{R}\Det_R(H^2(C))\otimes_{R}\Det_R^{-1}(H^2(C))\\
 & \overset{\text{ev}}{\simeq}  \Det_R (H^1_f(C))
\end{align*}
and $\Det_k^{-1}(\ol{C}) \simeq \Det_k (H^1_f(\ol{C}))$ in a similar manner.

By applying the snake lemma to the multiplication-by-$\omega$ map to the upper sequence of \eqref{eq:Det02_1}, we obtain a diagram with exact rows
\begin{equation}\label{eq:Det02_2}
\xymatrix{
	0 \ar[r]
	& H^2(C) \ar[r]
	& \ol{H^1_f(C)} \ar[r] \ar[d]
	& \ol{H^1(C)} \ar[r]^{\ol{\varphi}} \ar[d]
	& H^2(C) \ar[r] \ar[d]^{\simeq}
	& 0\\
	& 0 \ar[r]
	& H^1_f(\ol{C}) \ar[r]
	& H^1(\ol{C}) \ar[r]_{\ol{\varphi}}
	& H^2(\ol{C}) \ar[r]
	& 0.
}
\end{equation}
Taking \eqref{eq:Bock_seq} into account, applying the snake lemma to the above diagram again, we see that the cokernel of the left vertical arrow is isomorphic to $H^2(C)$. We then obtain an exact sequence
\begin{equation}\label{eq:Det02_3}
0 \to H^2(C) \to \ol{H^1_f(C)} \to H^1_f(\ol{C}) \to H^2(C) \to 0.
\end{equation}
\end{setting}

\begin{prop}\label{prop:Det02}
In Setting \ref{setting:A1}, put $s = \dim_k H^2(C)$.
Then the diagram
\[
\xymatrix{
\Det_R^{-1}(C) \ar[r]^-{\simeq} \ar@{->>}[dd]&\Det_R (H^1_f(C))\ar@{->>}[d]\\
 &  \Det_k (\ol{H^1_f(C)}) \ar[d]_{\simeq}^{\alpha}\\
\Det_k^{-1}(\ol{C})\ar[r]_-{\simeq}&\Det_k (H^1_f(\ol{C}))
}
\]
is commutative up to sign $(-1)^s$.
Here, the horizontal isomorphisms are the ones constructed following \eqref{eq:Det02_1}, the isomorphism $\alpha$ is induced by \eqref{eq:Det02_3} and Lemma \ref{lem:Det01}, and the other vertical arrows are the natural ones.
\end{prop}

\begin{proof}
As in the proof of Proposition \ref{prop:DetA}, we take $C \simeq [C^1 \overset{g}{\to} C^2]$, write $X$ for the image of $g$, and set $r = \rank_R H^1(C)$ and $m = \rank_R C^2$. We take
\begin{itemize}
\item
$x_1, \dots, x_r$ as an $R$-basis of $H^1(C)$ such that $\omega x_1, \dots, \omega x_s, x_{s+1}, \dots, x_r$ is an $R$-basis of $H^1_f(C)$ and 
\item
$y_1, \dots, y_m$ as an $R$-basis of $C^2$ such that $\omega y_1, \dots, \omega y_s, y_{s+1}, \dots, y_m$ is an $R$-basis of $X$.
\end{itemize}
Then $\varphi(x_1), \dots, \varphi(x_s)$ is a $k$-basis of $H^2(C)$.
If we write $[y_i] \in H^2(C)$ for the class of $y_i$, then $[y_1], \dots, [y_s]$ is also a $k$-basis of $H^2(C)$. By an appropriate linear transformation, we may assume that $\varphi(x_i) = [y_i]$ in $H^2(C)$ for $1 \leq i \leq s$.

Let $\wtil{\omega y_1}, \dots, \wtil{\omega y_s}, \wtil{y_{s+1}}, \dots, \wtil{y_m} \in C^1$ be lifts of the elements without tildes. We note that $\ol{\wtil{\omega y_i}} \in H^1 (\ol{C})$ for $1 \leq i \leq s$. We may take the lifts $\wtil{\omega y_i}$ so that $\ol{\wtil{\omega y_i}} \in H^1_f (\ol{C})$. In fact, since $\varphi(x_1), \dots, \varphi(x_s)$ is a basis of $H^2(C)$,  
there are elements $c_{ij} \in R$ (unique up to the maximal ideal) such that
\[
\ol{\varphi}(\ol{\wtil{\omega y_i}}) = \sum_{j=1}^s \ol{c_{ij}} \varphi(x_j) 
\]
in $H^2 (\ol{C}) \simeq H^2 (C)$. The desired lift of $\omega y_i$ is $\wtil{\omega y_i} - \sum_{j=1}^s c_{ij} x_j$. 

The elements $x_1, \dots, x_r, \wtil{\omega y_1}, \dots, \wtil{\omega y_s}, \wtil{y_{s+1}}, \dots, \wtil{y_m}$ form an $R$-basis of $C^1$.
Moreover, 
\[
\ol{x_1}, \dots, \ol{x_r}, \ol{\wtil{\omega y_1}}, \dots, \ol{\wtil{\omega y_s}}
\]
is a basis of $H^1(\ol{C})$ and $\ol{\wtil{\omega y_i}}$ is in $H^1_f(\ol{C})$ for $1 \leq i \leq s$.
We can take an $R$-basis of $\Det_R^{-1}(C) = \Det_R(C^1) \otimes \Det_R^{-1}(C^2)$ as
\[
\xi := (x_1 \wedge \dots \wedge x_r \wedge \wtil{\omega y_1} \wedge \dots \wedge \wtil{\omega y_s} \wedge \wtil{y_{s+1}} \wedge \dots \wedge \wtil{y_m})
\otimes (y_1 \wedge \dots \wedge y_m)^*.
\]
Under the isomorphism $\Det_R^{-1}(C) \simeq \Det_R(H^1_f(C))$, $\xi$ goes to
\[
\omega x_1 \wedge \dots \wedge \omega x_s \wedge x_{s + 1} \wedge \dots \wedge x_r
\]
in $\Det_R H^1_f(C)$, and further goes to
\[
\ol{\omega x_1} \wedge \dots \wedge \ol{\omega x_s} \wedge \ol{x_{s + 1}} \wedge \dots \wedge \ol{x_r}
\]
in $\Det_k \ol{H^1_f(C)}$ under the natural projection.

We compute the image of 
\[
\ol{\xi} = 
(\ol{x_1} \wedge \dots \wedge \ol{x_r} \wedge \ol{\wtil{\omega y_1}} \wedge \dots \wedge \ol{\wtil{\omega y_s}} \wedge \ol{\wtil{y_{s+1}}} \wedge \dots \wedge \ol{\wtil{y_m}})
\otimes (\ol{y_1} \wedge \dots \wedge \ol{y_m})^*
\]
under the lower horizontal arrow. 
The canonical isomorphism $\Det_k^{-1}(\ol{C}) \simeq \Det_k(H^1(\ol{C})) \otimes \Det_k^{-1}(H^2(\ol{C}))$ sends 
$\ol{\xi}$ to
\begin{align}
& (\ol{x_1} \wedge \dots \wedge \ol{x_r} \wedge \ol{\wtil{\omega y_1}} \wedge \dots \wedge \ol{\wtil{\omega y_s}})
\otimes ([y_1] \wedge \dots \wedge [y_s])^*\\
& \quad =
(-1)^{s r} (\ol{x_{s+1}} \cdot \wedge \dots \wedge \ol{x_r} \wedge \ol{\wtil{\omega y_1}} \wedge \dots \wedge \ol{\wtil{\omega y_s}} \wedge \ol{x_1} \wedge \dots \wedge \ol{x_s})
\otimes ([y_1] \wedge \dots \wedge [y_s])^*.
\end{align}
Since $\varphi(x_i) = [y_i]$ ($1 \leq i \leq s$), under the isomorphism $\Det_k(H^1(\ol{C})) \otimes \Det_k^{-1}(H^2(\ol{C}))
\simeq \Det_k H^1_f(\ol{C})$, this element goes to 
\begin{align}
& (-1)^{s r} \cdot \ol{x_{s+1}} \wedge \dots \wedge \ol{x_r} \wedge \ol{\wtil{\omega y_1}} \wedge \dots \wedge \ol{\wtil{\omega y_s}}\\
& \quad =
(-1)^{s r} \cdot (-1)^{(r - s) s} \cdot \ol{\wtil{\omega y_1}} \wedge \dots \wedge \ol{\wtil{\omega y_s}} \wedge \ol{x_{s+1}} \wedge \dots \wedge \ol{x_r}\\
& \quad =
(-1)^s \cdot \ol{\wtil{\omega y_1}} \wedge \dots \wedge \ol{\wtil{\omega y_s}} \wedge \ol{x_{s+1}} \wedge \dots \wedge \ol{x_r}.
\end{align}

To conclude the proof, we only have to show that 
\[
\alpha(\ol{\omega x_1} \wedge \dots \wedge \ol{\omega x_s} \wedge \ol{x_{s + 1}} \wedge \dots \wedge \ol{x_r})=\ol{\wtil{\omega y_1}} \wedge \dots \wedge \ol{\wtil{\omega y_s}} \wedge \ol{x_{s+1}} \wedge \dots \wedge \ol{x_r}.
\]
Recall that $\alpha$ is induced by \eqref{eq:Det02_3}.
By observing the construction, the map $H^1_f(\ol{C}) \to H^2(C)$ sends $\ol{\wtil{\omega y_i}}$ to $[y_i]$ and the map $H^2(C) \to \ol{H^1_f(C)}$ sends $[y_i]$ to $\ol{\omega x_i}$, since $\varphi(x_i) = [y_i]$.
Now Lemma \ref{lem:Det01} shows the claim.
\end{proof}

\section{A proof of the generalized Davenport--Hasse relation}\label{sec:DH}

In this section we prove Lemma \ref{lem:Davenport--Hasse}(iii).
This proof was provided to the authors by Daichi Takeuchi.

We first briefly recall the classical Davenport--Hasse relation for Gauss sums over finite fields. We fix a prime number $p$. Let $\F$ be a finite field whose characteristic is $p$. 
We fix a primitive $p$-th root of unity $\zeta_p\in\C$ and define an additive character
\[
\psi' : \F \to \C^{\times} ; a \mapsto \zeta_p^{\Tr_{\F/\F_p}(a)}.
\]
Let $\chi' : \F^{\times} \to \C^{\times}$ be a (multiplicative) character.
Then the classical Gauss sum $g(\chi')$ over $\F$ is defined by
\[
g(\chi'):= \sum_{a\in\F^{\times}} \chi'(a)\psi'(a).
\]

For a positive integer $s$, let $\F_s$ be the unique extension of $\F$ such that $[\F_s : \F]=s$. 
Associated to $\chi'$, we have a character $\chi : \F_s^{\times} \to \C^{\times}$ defined by
\[
\chi:=\chi' \circ \NN_{\F_s/\F} : \F^{\times}_s \to \C^{\times}.
\]
Similarly, we have an additive character $\psi$ over $\F_s$ which is defined as $\psi:=\psi'\circ\Tr_{\F_s/\F}$ and use it to define the Gauss sum $g(\chi)$. 
The classical Davenport--Hasse relation describes a relationship between these two Gauss sums $g(\chi)$ and $g(\chi')$.

\begin{thm}[Davenport--Hasse relation, {\cite[Chapter 11, Section 3, Theorem 1]{IrRo82}}]\label{classical HD}
We have $g(\chi) = (-1)^{s-1} g(\chi')^s$.
\end{thm}

Now we recall the notation in \S \ref{ss:choice}.
Let $k$ be a finite extension of $\Q$ such that $p$ is unramified in $k/\Q$ and $k_n:=k(\mu_{p^n})$. We consider a character $\chi_2 : \Gal(k_n/k) \to \C^{\times}$, and define $n_{\chi}$ as the minimum integer such that $\chi_2$ factors through $\Gal(k_{n_{\chi}}/k)$.

For a $p$-adic place $v$ of $k$, we write $\chi_{2, v}$ for the restriction of $\chi$ on the decomposition group of $v$. By combining $\chi_{2, v}$ with the local reciprocity map, we have a character of $k^{\times}_v$ denoted by $\theta_{\chi_{2, v}} : k^{\times}_v \to \C^{\times}$. 

Through the isomorphism $\Gal(k_n/k)\simeq \Gal(\Q(\mu_{p^n})/\Q)$, we can also regard $\chi_2$ as a character $\chi'_2: \Gal(\Q(\mu_{p^n})/\Q) \to \C^{\times}$.
We then similarly write $\chi'_{2, p}$ for the restriction of $\chi'_2$ on the decomposition group of $p$
and, combining with the local reciprocity map, have a character $\theta_{\chi'_{2, p}} : \Q_p^{\times} \to \C^{\times}$.
By construction, we have $\theta_{\chi_{2, v}} = \theta_{\chi'_{2, p}} \circ \NN_{k_v/\Q_p}$.

For these characters $\theta_{\chi_{2, v}}, \theta_{\chi'_{2, p}}$, we have introduced the local Gauss sums $\tau_k(\chi_{2, v}), \tau_{\Q}(\chi'_{2, p})$ in Definition \ref{def:Gsum}. 
These local Gauss sums are related to each other as follows.

\begin{thm}[Lemma \ref{lem:Davenport--Hasse}(iii)]\label{app:HD}
We have 
\[
\tau_k(\chi_{2, v}) = (-1)^{n_{\chi}(r_{k_v}-1)}\cdot \tau_{\Q}(\chi'_{2, p})^{r_{k_v}}
\]
with $r_{k_v}:=[k_v : \Q_p]$.
\end{thm}

\begin{rem}
This theorem generalizes Theorem \ref{classical HD} to wildly ramified characters.
In fact, let us show that the statement of Theorem \ref{app:HD} for $n_{\chi} = 1$ is equivalent to Theorem \ref{classical HD}.
When $n_{\chi}=1$ (i.e., $\chi_{2, v}$ is tamely ramified), we can use a uniformizer $\pi$ of $k_v$ as the element $c \in k_v^{\times}$ in Definition \ref{def:Gsum} since $\ff(\chi_v)=(\pi)$.
We also have $\theta_{\chi_{2, v}}(\pi)=1$ since $k_v/\Q_p$ is unramified.
Then the map
\[
\OO_{k_v}^{\times} \to \C^{\times}\; ;\; u\mapsto \theta_{\chi_{2,v}}\left(u\pi^{-1}\right)
\]
defines a character by the condition $\theta_{\chi_{2, v}}(\pi)=1$, and factors through $\F_v^{\times}$, where $\F_v$ is the residue field of $k_v$. In addition, since the additive character $\psi_{k_v}$ in the definition of local Gauss sums factors through $\OO_{k_v}$, the additive character 
\[
\OO_{k_v} \to \C^{\times}\; ;\; u \mapsto \psi_{k_v}\left(u\pi^{-1}\right)
\]
factors through $\F_v$. Therefore, when $n_{\chi}=1$, the local Gauss sums $\tau_k(\chi_{2, v})$ and $\tau_{\Q}(\chi'_{2, p})$ are the same as the classical Gauss sums in Theorem \ref{classical HD} defined for characters over finite fields. 
\end{rem}

\begin{proof}[Proof of Theorem \ref{app:HD}]
When $n_{\chi}=0$ (i.e., $\chi_2$ is trivial), both sides of the equality in the statement are $1$ by definition.

When $n_{\chi}=1$, the statement follows from Theorem \ref{classical HD} as already discussed.

In the rest of the proof, we consider the case where $n_{\chi}>1$. 
We often identify elements of subquotients of $\OO_{k_v}$ and $\Z_p$ with their arbitrary lifts in $\OO_{k_v}$ and $\Z_p$ respectively, and assign the same symbols for them if no confusion occurs. 

We first consider the case where $n_{\chi}$ is odd. Put $n_{\chi}=2m+1$ ($m\in\Z_{\geq 1}$). With the additive character $\psi_{k_v}$, we can define a perfect pairing 
\begin{equation}\label{trpairing}
\OO_{k_v}/p^{n_{\chi}}\OO_{k_v} \times \OO_{k_v}/p^{n_{\chi}}\OO_{k_v} \rightarrow \C^{\times}\ ;\ (x, y)\mapsto \psi_{k_v}(xyp^{-n_{\chi}}),
\end{equation}
where the value $\psi_{k_v}(xyp^{-n_{\chi}})$ is well-defined since $\psi_{k_v}(\OO_{k_v})=0$. 
Since the exact annihilator of $p^m\OO_{k_v}/p^{n_{\chi}} \OO_{k_v}$, this pairing induces
\begin{equation}\label{pfpair}
p^m\OO_{k_v}/p^{n_{\chi}} \OO_{k_v} \times \OO_{k_v}/p^{m+1}\OO_{k_v} \rightarrow \C^{\times}, 
\end{equation}
which is also perfect.
We define a group homomorphism (a ``truncated'' exponential map)
\[
E_1 : p^m\OO_{k_v}/p^{n_{\chi}} \OO_{k_v} \rightarrow 1+p^m\OO_{k_v}/ 1+ p^{n_{\chi}}\OO_{k_v}\ ;\ \ z \mapsto 1+z+\frac{z^2}{2}.
\]
We can see that this is an isomorphism. 
Then because of the perfect pairing \eqref{pfpair}, we can associate to the homomorphism $\theta_{\chi_{2, v}}\circ E_1 : p^m\OO_{k_v}/p^{n_{\chi}} \OO_{k_v} \to \C^{\times}$ an element $e_{k_v} \in \OO_{k_v}/p^{m+1}\OO_{k_v}$ such that 
\begin{align}\label{thetavspsi}
\theta_{\chi_{2, v}}(E_1 (z)) = \psi_{k_v}(e_{k_v}zp^{-n_{\chi}})
\end{align}
holds for all $z \in p^m\OO_{k_v}/p^{n_{\chi}} \OO_{k_v}$.
Since $\theta_{\chi_{2, v}}\circ E_1$ is injective, we have $e_{k_v} \in (\OO_{k_v}/p^{m+1}\OO_{k_v})^{\times}$.
In the same way, we can take an element $e_{\Q_p} \in (\Z_p/p^{m+1}\Z_p)^{\times}$ such that 
\[
\theta_{\chi_{2, p}^{\prime}}(E_1 (z)) = \psi_{\Q_p}(e_{\Q_p}zp^{-n_{\chi}})
\]
holds for all  $z \in p^m\Z_p/p^{n_{\chi}} \Z_p$.
Here, we restrict $E_1$ on $p^m\Z_p/p^{n_{\chi}} \Z_p  \subset p^m\OO_{k_v}/p^{n_{\chi}}\OO_{k_v}$. 

Let us show $e_{k_v}=e_{\Q_p}$ in $\Z_p/p^{m+1}\Z_p \subset \OO_{k_v}/p^{m+1}\OO_{k_v}$. 
Since $\theta_{\chi_{2, v}} = \theta_{\chi_{2, p}^{\prime}} \circ \NN_{k_v/\Q_p}$ and 
$\NN_{k_v/\Q_p}\circ E_1 = E_1 \circ \Tr_{k_v / \Q_p}$, 
we have 
\begin{align*}
\theta_{\chi_{2,v}} (E_1(z)) = \theta_{\chi^{\prime}_{2,p}}(\NN_{k_v/\Q_p}(E_1(z))) = \theta_{\chi^{\prime}_{2,p}}(E_1(\Tr_{k_v/\Q_p}(z)))
 = \psi_{\Q_p}(\Tr_{k_v/\Q_p}(e_{\Q_p}zp^{-n_{\chi}}))
\end{align*}
for $z \in p^m\OO_{k_v}/p^{n_{\chi}} \OO_{k_v}$. 
On the other hand, since $\psi_{k_v} = \psi_{\Q_p}\circ\Tr_{k_v/\Q_p}$, 
we have
\[
\theta_{\chi_{2,v}} (E_1(z)) = \psi_{k_v}(e_{k_v}zp^{-n_{\chi}}) = \psi_{\Q_p}(\Tr_{k_v/\Q_p}(e_{k_v}zp^{-n_{\chi}}))
\]
for $z \in p^m\OO_{k_v}/p^{n_{\chi}} \OO_{k_v}$. Since $\Ker \psi_{\Q_p} = \Z_p$, we have $\Tr_{k_v/\Q_p}((e_{k_v}-e_{\Q_p})zp^{-n_{\chi}}) \in \Z_p$ for arbitrary lifts of $e_{k_v}, e_{\Q_p}$ and all $z \in p^m\OO_{k_v}$. This implies
\[
e_{k_v}-e_{\Q_p} \in p^{n_{\chi}-m}\cD_{k_v / \Q_p} ^{-1} = p^{m+1}\OO_{k_v},
\]
where we use the fact that $k_v / \Q_p$ is unramified. Therefore, we obtain $e_{k_v}=e_{\Q_p}$ in $\OO_{k_v}/p^{m+1}\OO_{k_v}$, as claimed.

We put $e:=e_{k_v}=e_{\Q_p} \in (\Z_p/p^{m+1}\Z_p)^{\times}$. In the following computation of the Gauss sum $\tau_k(\chi_{2, v})$, we decompose $u\in \OO_{k_v}^{\times}/1+p^{n_{\chi}} \OO_{k_v}$ as a product $u = xy$ with $x \in \OO_{k_v}^{\times}/1+p^m \OO_{k_v}$ (lifted to $\OO_{k_v}^{\times}/1+p^{n_{\chi}} \OO_{k_v}$) and $y \in 1+p^m\OO_{k_v}/1+p^{n_{\chi}} \OO_{k_v}$. 
We have  
\begin{align*}
\tau_k(\chi_{2, v}) & = \sum_{u \in \OO_{k_v}^{\times}/1+p^{n_{\chi}} \OO_{k_v}} \theta_{\chi_{2, v}}(up^{-n_{\chi}})\psi_{k_v}(up^{-n_{\chi}})\\
&=\sum_{x \in \OO_{k_v}^{\times}/1+p^m \OO_{k_v}}\sum_{y \in 1+p^m\OO_{k_v}/1+p^{n_{\chi}}\OO_{k_v}} \theta_{\chi_{2, v}}\left(xyp^{-n_{\chi}}\right)\psi_{k_v}\left(xyp^{-n_{\chi}}\right)\\
&=\sum_{x \in \OO_{k_v}^{\times}/1+p^m \OO_{k_v}}\theta_{\chi_{2, v}}(x p^{-n_{\chi}})
\sum_{z\in p^m\OO_{k_v}/p^{n_{\chi}}\OO_{k_v}} 
\theta_{\chi_{2, v}}(E_1 (z))\psi_{k_v}\left(x E_1 (z) p^{-n_{\chi}}\right)\\
&=\sum_{x \in \OO_{k_v}^{\times}/1+p^m \OO_{k_v}}\theta_{\chi_{2, v}}(x p^{-n_{\chi}})
\sum_{z\in p^m\OO_{k_v}/p^{n_{\chi}}\OO_{k_v}} 
\psi_{k_v}(e z p^{-n_{\chi}})\psi_{k_v}\left(x\left(1+z+\frac{z^2}{2}\right)p^{-n_{\chi}}\right)\\
&=\sum_{x \in \OO_{k_v}^{\times}/1+p^m \OO_{k_v}}\theta_{\chi_{2, v}}(xp^{-n_{\chi}})\psi_{k_v}(xp^{-n_{\chi}})\sum_{z\in p^m\OO_{k_v}/p^{n_{\chi}}\OO_{k_v}} \psi_{k_v}\left(\left(ez+xz+x\frac{z^2}{2}\right)p^{-n_{\chi}}\right).
\end{align*}
For the sum with respect to $z$, we have 
\begin{align*}
&\sum_{z\in p^m\OO_{k_v}/p^{n_{\chi}}\OO_{k_v}} \psi_{k_v}\left(\left(ez+xz+x\frac{z^2}{2}\right)p^{-n_{\chi}}\right)\\
= &\sum_{\substack{z_1\in p^m\OO_{k_v}/p^{m+1}\OO_{k_v} \\ z_2\in p^{m+1}\OO_{k_v}/p^{n_{\chi}}\OO_{k_v}}} \psi_{k_v}\left(\left(e(z_1+z_2)+x(z_1+z_2)+x\frac{(z_1+z_2)^2}{2}\right)p^{-n_{\chi}}\right)\\
= & \sum_{z_1 \in p^m\OO_{k_v}/p^{m+1}\OO_{k_v}} \psi_{k_v}\left(\left(ez_1+xz_1+x\frac{z_1^2}{2}\right)p^{-n_{\chi}}\right)\sum_{z_2\in p^{m+1}\OO_{k_v}/p^{n_{\chi}}\OO_{k_v}} \psi_{k_v}\left((e+x)z_2 p^{-n_{\chi}}\right).
\end{align*}
Here, we decompose $z\in p^m\OO_{k_v}/p^{n_{\chi}}\OO_{k_v}$ as a sum $z = z_1 + z_2$ with $z_1 \in p^m\OO_{k_v}/p^{m+1}\OO_{k_v}$ (lifted to $p^m\OO_{k_v}/p^{n_{\chi}}\OO_{k_v}$) and $z_2\in p^{m+1}\OO_{k_v}/p^{n_{\chi}}\OO_{k_v}$.
We also note that we use $(z_1+z_2)^2 \equiv z_1^2 \pmod{ p^{n_{\chi}}\OO_{k_v}}$ in the second equality above. The sum with respect to $z_2$ is equal to 
\[
\sum_{z_2\in p^{m+1}\OO_{k_v}/p^{n_{\chi}}\OO_{k_v}}  \psi_{k_v}\left((e+x)z_2 p^{-n_\chi}\right) = 
\begin{cases}
 p^{mr_{k_v}}    &\text{if } e+x \equiv 0\pmod{ p^{m}\OO_{k_v}}, \\
 0    &\text{otherwise}. 
\end{cases}
\]
Therefore, we have
\begin{align*}
\tau_k(\chi_{2, v})&=p^{mr_{k_v}}\theta_{\chi_{2, v}}(-ep^{-n_{\chi}})\psi_{k_v}(-ep^{-n_{\chi}})\sum_{z_1\in p^m\OO_{k_v}/p^{m+1}\OO_{k_v}} \psi_{k_v}\left(-e\frac{z_1^2}{2}p^{-n_{\chi}}\right)\\
& = p^{mr_{k_v}}\theta_{\chi^{\prime}_{2,p}}(-ep^{-n_{\chi}})^{r_{k_v}}\psi_{\Q_p}(-ep^{-n_{\chi}})^{r_{k_v}}\sum_{z_1\in p^m\OO_{k_v}/p^{m+1}\OO_{k_v}} \psi_{k_v}\left(-e\frac{z_1^2}{2}p^{-n_{\chi}}\right).
\end{align*}
The sum with respect to $z_1$ can be seen as a classical Gauss sum over the residue field $\mathbb{F}_v$ of $k_v$. Indeed, let $\displaystyle \left(\frac{\ast}{p}\right)$ be the Legendre symbol over $(\Z_p/p\Z_p)^{\times}$. By the identification
\[
p^m\OO_{k_v}/p^{m+1}\OO_{k_v} \simeq \F_v\; ; z_1=p^mz_0 \mapsto z_0 \mod p,
\]
we have
\begin{align*}
& \sum_{z_1\in p^m\OO_{k_v}/p^{m+1}\OO_{k_v}} \psi_{k_v}\left(-e\frac{z_1^2}{2}p^{-n_{\chi}}\right)
=\sum_{z_0\in \mathbb{F}_v} \psi_{k_v}\left(-e\frac{z_0^2}{2}p^{-1}\right)\\
& \quad =1+2\sum_{z_0\in (\mathbb{F}_v^{\times})^2} \psi_{k_v}\left(-e\frac{z_0}{2}p^{-1}\right)\\
& \quad = 1+\sum_{z_0\in \mathbb{F}_v^{\times}} \left(\frac{\NN_{k_v/\Q_p}(z_0)}{p}\right)\psi_{k_v}\left(-e\frac{z_0}{2}p^{-1}\right)+\sum_{z_0\in \mathbb{F}_v^{\times}} \psi_{k_v}\left(-e\frac{z_0}{2}p^{-1}\right)\\
& \quad = \sum_{z_0\in \mathbb{F}_v^{\times}} \left(\frac{\NN_{k_v/\Q_p}(z_0)}{p}\right)\psi_{k_v}\left(-e\frac{z_0}{2}p^{-1}\right), 
\end{align*}
and this is a classical Gauss sum over $\F_v$ attached to the Legendre symbol. Here, 
the last equality follows by 
$\sum_{z_0\in \mathbb{F}_v^{\times}} \psi_{k_v}\left(-e\dfrac{z_0}{2}p^{-1}\right)=-1.$ By Theorem \ref{classical HD}, we have
\[
\sum_{z_0\in \mathbb{F}_v^{\times}} \left(\frac{\NN_{k_v/\Q_p}(z_0)}{p}\right)\psi_{k_v}\left(-e\frac{z_0}{2}p^{-1}\right)= (-1)^{r_{k_v}-1}\left(\sum_{z_0\in \mathbb{F}_p^{\times}} \left(\frac{z_0}{p}\right)\psi_{\Q_p}\left(-e\frac{z_0}{2}p^{-1}\right)\right)^{r_{k_v}}. 
\]
Therefore, we obtain 
\begin{align*}
\tau_k(\chi_{2, v}) =  (-1)^{r_{k_v}-1}\left(p^m \theta_{\chi^{\prime}_{2,p}}(-ep^{-n_{\chi}})\psi_{\Q_p}(-ep^{-n_{\chi}})\sum_{z_0\in \mathbb{F}_p^{\times}} \left(\frac{z_0}{p}\right)\psi_{\Q_p}\left(-e\frac{z_0}{2}p^{-1}\right)\right)^{r_{k_v}}.
\end{align*}
On the other hand, by the similar computation for $\tau_{\Q}(\chi_{2, p}^{\prime})$ as above, 
we also obtain
\[
\tau_{\Q}(\chi_{2, p}^{\prime}) = p^m \theta_{\chi^{\prime}_{2,p}}(-ep^{-n_{\chi}})\psi_{\Q_p}(-ep^{-n_{\chi}})\sum_{z_0\in \mathbb{F}_p^{\times}} \left(\frac{z_0}{p}\right)\psi_{\Q_p}\left(-e\frac{z_0}{2}p^{-1}\right). 
\]
These two formulas imply Theorem \ref{app:HD} in the case where $n_{\chi}>1$ is odd. 

Finally we consider the case where $n_{\chi}>1$ is even. Put $n_{\chi}=2m$ ($m\in\Z_{>1}$). 
This case can be shown in a similar but easier way.
We consider a group isomorphism
\[
E_0 : p^m\OO_{k_v}/p^{n_{\chi}} \OO_{k_v} \rightarrow 1+p^m\OO_{k_v}/ 1+ p^{n_{\chi}}\OO_{k_v}\ ;\ \ z \mapsto 1+z.
\]
In the same way as in the even case, we can show that there exists $e\in (\Z_p/p^m\Z_p)^{\times}$ which satisfies
\[
\theta_{\chi_{2, v}}(E_0(z)) = \psi_{k_v}(ezp^{-n_{\chi}})
\]
for all $z \in p^m\OO_{k_v}/p^{n_{\chi}}\OO_{k_v}$, and 
\[
\theta_{\chi_{2, p}^{\prime}}(E_0(z)) = \psi_{\Q_p}(ezp^{-n_{\chi}})
\]
for all $z \in p^m\Z_p/p^{n_{\chi}}\Z_p$, where we restrict $E_0$ to $p^m\Z_p/p^{n_{\chi}}\Z_p \subset p^m\OO_{k_v}/p^{n_{\chi}}\OO_{k_v}$ in the second equality. 
By a similar computation as in the even case, we have 
\begin{align*}
\tau_k(\chi_{2, v}) &=\sum_{u \in \OO_{k_v}^{\times}/1+p^{n_{\chi}} \OO_{k_v}} \theta_{\chi_{2, v}}(up^{-n_{\chi}})\psi_{k_v}(up^{-n_{\chi}})\\
& = \sum_{x \in \OO_{k_v}^{\times}/1+p^m\OO_{k_v}}\sum_{y \in 1+p^m\OO_{k_v}/1+p^{n_{\chi}}\OO_{k_v}} \theta_{\chi_{2, v}}(xyp^{-n_{\chi}})\psi_{k_v}(xyp^{-n_{\chi}})\\
&=\sum_{x \in \OO_{k_v}^{\times}/1+p^m \OO_{k_v}}\sum_{z\in p^m\OO_{k_v}/p^{n_{\chi}}\OO_{k_v}} \theta_{\chi_{2, v}}\left(x E_0(z)p^{-n_{\chi}}\right)\psi_{k_v}\left(xE_0(z)p^{-n_{\chi}}\right)\\
&=\sum_{x \in \OO_{k_v}^{\times}/1+p^m \OO_{k_v}}\theta_{\chi_{2, v}}(xp^{-n_{\chi}})\psi_{k_v}(xp^{-n_{\chi}})\sum_{z\in p^m\OO_{k_v}/p^{n_{\chi}}\OO_{k_v}} \psi_{k_v}\left((e+x)zp^{-n_{\chi}}\right)\\
&=\theta_{\chi_{2, v}}(-ep^{-n_{\chi}})\psi_{k_v}(-ep^{-n_{\chi}})p^{r_{k_v}(n_{\chi}-m)}=\theta_{\chi_{2, p}^{\prime}}(-ep^{-n_{\chi}})^{r_{k_v}}\psi_{\Q_p}(-ep^{-n_{\chi}})^{r_{k_v}} p^{r_{k_v}(n_{\chi}-m)}.
\end{align*}
In the fifth equality, we use the fact that the sum with respect to $z$ is $0$ unless $x \equiv -e \pmod{ p^m\OO_{k_v}}$. 
By a similar computation for $\tau_{\Q}(\chi_{2, p}^{\prime})$, we also obtain
\[
\tau_{\Q}(\chi_{2, p}^{\prime}) = \theta_{\chi_{2, p}^{\prime}}(-ep^{-n_{\chi}})\psi_{\Q_p}(-ep^{-n_{\chi}}) p^{n_{\chi}-m}.
\]
These two formulas imply Theorem \ref{app:HD} in the case where $n_{\chi}>1$ is even.
\end{proof}

{
\bibliographystyle{abbrv}
\bibliography{biblio}
}

\end{document}